\documentclass{birkjour}

\usepackage{amssymb}
\usepackage{hyperref}

 \newtheorem{thm}{Theorem}[section]
 \newtheorem{cor}[thm]{Corollary}
 \newtheorem{lem}[thm]{Lemma}
 \newtheorem{prop}[thm]{Proposition}
 \theoremstyle{definition}
 \newtheorem{defn}[thm]{Definition}
 \theoremstyle{remark}
 \newtheorem{rem}[thm]{Remark}
 \newtheorem*{ex}{Example}
 \numberwithin{equation}{section}

\newcommand\wtsmash[1]{{\smash{\widetilde{#1}}\rlap{$\phantom{#1}$}}}
\newcommand\whsmash[1]{{\smash{\widehat{#1}}\rlap{$\phantom{#1}$}}}

\newcommand{\ud}{\,{\mathrm d}}

\newcommand{\spn}{{\rm span}}

\newcommand \R{{\mathbb R}}
\newcommand \C{{\mathbb C}}
\newcommand \D{{\mathbb D}}

\newcommand \Z{{\mathbb Z}}
\newcommand \rplus{{\R^{+}}}
\newcommand \cplus{{\C^{+}}}
\newcommand \zplus{{\Z^{+}}}

\newcommand \zero{\set{0}}

\newcommand{\BLO}{\mathcal B}

\newcommand{\Bscr}{\mathcal B}

\newcommand{\Dscr}{\mathcal D}
\newcommand{\Escr}{\mathcal E}
\newcommand{\Fscr}{\mathcal F}
\newcommand{\Gscr}{\mathcal G}
\newcommand{\Hscr}{\mathcal H}

\newcommand{\Rscr}{\mathcal R}
\newcommand{\Sscr}{\mathcal S}

\newcommand{\Uscr}{\mathcal U}
\newcommand{\Xscr}{\mathcal X}
\newcommand{\Yscr}{\mathcal Y}

\newcommand{\AB}{{A\& B}}
\newcommand{\CD}{{C\& D}}

\newcommand{\SysNode}{\bbm{\AB \cr \CD}}
\newcommand{\SmallSysNode}{\sbm{\AB \cr \CD}}

\newcommand{\dom}[1]{{{\rm dom}}\left(#1\right)}

\newcommand{\range}[1]{{\rm im}\bigl(#1\bigr)}

\newcommand{\Ker}[1]{{\rm ker}\left(#1\right)}

\newcommand{\Ipd}[2]{\left\langle #1 , #2 \right\rangle}

\newcommand{\Ipdp}[2]{\left( #1 , #2 \right)}

\newcommand{\set}[1]{\left\lbrace #1 \right\rbrace}

\newcommand{\bigmid}{\bigm \vert}

\newcommand{\biggmid}{\biggm\vert}

\newcommand{\bi}{\begin{itemize}}
\newcommand{\ei}{\end{itemize}}
\newcommand{\be}{\begin{enumerate}}
\newcommand{\ee}{\end{enumerate}}

\newcommand{\Afrak}{\mathfrak A}

\newcommand{\Dfrak}{\mathfrak D}

\newcommand{\Wfrak}{\mathfrak W}

\newcommand{\sbm}[1]{\left[\begin{smallmatrix}#1\end{smallmatrix}\right]}

\newcommand{\bbm}[1]{\begin{bmatrix}#1\end{bmatrix}}

\newcommand{\re}[1]{\mathrm{Re}\,#1}

\newcommand{\cspn}{\overline{\rm span}}

\newcommand{\cB}{{\mathcal B}}

\newcommand{\cD}{{\mathcal D}}

\newcommand{\cF}{{\mathcal F}} 
\newcommand{\cG}{{\mathcal G}}
\newcommand{\cH}{{\mathcal H}}

\newcommand{\cS}{{\mathcal S}}

\newcommand{\cU}{{\mathcal U}}
\newcommand{\cX}{{\mathcal X}}
\newcommand{\cY}{{\mathcal Y}}
\newcommand{\cW}{{\mathcal W}}

\newcommand{\bU}{{\mathbf U}}

\newcommand{\dA}{\mathbf A}
\newcommand{\dB}{\mathbf B}
\newcommand{\dC}{\mathbf C}
\newcommand{\dD}{\mathbf D}

\newcommand{\res}[1]{\mathrm{res}\left(#1\right)}

\makeatletter

\def\etv{& \hskip-.3em\vrule\hskip-.3em &} 
\def\smalletv{&\vrule&} 
\def\smallcrh{\vrule height0pt depth2\ex@ width0pt
\cr\noalign{\hrule}
\vrule height6.5\ex@ depth0pt width0pt}
\newbox\smallstrutbox
\setbox\smallstrutbox=\hbox{\vrule height6pt depth1.5pt width\z@}
\def\smallstrut{\relax\ifmmode\copy\smallstrutbox\else\unhcopy\smallstrutbox\fi}

\newenvironment{sysmatrix}{
\let\|=\etv
\hskip \arraycolsep
\begin{matrix}}
{\end{matrix}
\hskip \arraycolsep
}       

\newenvironment{smallsysmatrix}{\null\,\vcenter\bgroup
\let\|=\smalletv

\def\\{\smallstrut\math@cr}
\restore@math@cr\default@tag
\baselineskip\z@skip \lineskip\z@skip \lineskiplimit\lineskip
\ialign\bgroup\hfil$\m@th\scriptstyle##$\hfil&&\thickspace\hfil
$\m@th\scriptstyle##$\hfil\crcr
\crcr\noalign{\vskip -.3\ex@}%
}{\crcr\noalign{\vskip -.2\ex@}%
\crcr\egroup\egroup\,%
}

\makeatother

\begin{document}

%
%
%
%
%
%
%
%
%

\title[De Branges-Rovnyak realizations on the right half-plane]{De Branges-Rovnyak realizations of ope\-ra\-tor-valued Schur functions on the comp\-lex right half-plane}

\author{Joseph A.\ Ball}

\address{%
Department of Mathematics\\
Virginia Tech\\
Blacksburg, VA 24061\\
USA}

\email{joball@math.vt.edu}

\author{Mikael Kurula}
\address{\AA bo Akademi Mathematics\br
F\"anriksgatan 3B\br
FIN-20500 \AA bo\br
Finland}
\email{mkurula@abo.fi}

\thanks{The second author gratefully acknowledges support from the foundations of \AA bo Akademi and Ruth och Nils-Erik Stenb\"ack.}

\author{Olof J.\ Staffans}
\address{\AA bo Akademi Mathematics\br
F\"anriksgatan 3B\br
FIN-20500 \AA bo\br
Finland}
\email{staffans@abo.fi}

\author{Hans Zwart}
\address{University of Twente\br
Department of Applied Mathematics\br
P.O.\ Box 217\br
7500 AE Enschede\br
The Netherlands}
\email{h.j.zwart@math.utwente.nl}




\begin{abstract}
 We give a controllable energy-preserving and an observable co-energy-preserving de 
Branges-Rovnyak functional model realization of an arbitrary given operator Schur function defined on the complex right-half plane. We work the theory out fully in the right-half plane, without using results for the disk case, in order to expose the technical details of continuous-time systems theory. At the end of the article, we make explicit the connection to the corresponding classical de Branges-Rovnyak realizations for Schur functions on the complex unit disk.
\end{abstract}

\maketitle
\tableofcontents

\section{Introduction}

It essentially goes back to Kalman (with earlier roots in circuit theory from the middle of the twentieth 
century) that any rational function $\phi$ 
holomorphic in a neighborhood of the origin with values in 
the space $\cB(\cU, \cY)$ of bounded linear operators between two Hilbert spaces $\cU$ 
(the input space) and 
$\cY$ (the output space)  can be realized
as the transfer function of an input/state/output linear system, i.e., there is a Hilbert 
space $\cX$ 
(the state space) and 
a bounded operator system matrix $\bU: = \left[ \begin{smallmatrix} A & B \\ C & D \end{smallmatrix} \right] 
: \sbm{\cX \\ \cU} \to \sbm{\cX \\ \cY}$
so that $\phi(z)$ has the representation 
\begin{equation}   \label{transfunc}
    \phi(z) = D + z C (1 - zA)^{-1}B. 
\end{equation}If we associate with 
$\bU$ the discrete-time input/state/output system
\begin{equation}   \label{sys}
    \Sigma_\bU :\quad \left\{ \begin{array}{rcl}  x(t+1) & = & A x(t) + B u(t) \\
    y(t) & = & C x(t) + D u(t) \end{array}, \right.
\end{equation}
the meaning of \eqref{transfunc} is that $\phi$ is the {\em transfer 
function} of the i/s/o system $\Sigma_{\bU}$ in the following sense:
whenever the input string 
$\{u_{n}\}_{n \in {\mathbb Z}_{+}}$ is fed into the system 
\eqref{sys} with the initial condition $x(0) = 0$ on the state vector, 
the output string $\{y(n)\}_{n \in {\mathbb Z}_{+}}$ is produced, such that $ \widehat y(z) = \phi(z) \widehat u(z)$, where
$\widehat u$ and $\widehat y$ denote the $Z$-transform of 
$\{u(n)\}_{n \in {\mathbb Z}_{+}}$ and $\{y(n)\}_{n \in 
{\mathbb Z}_{+}}$:
\begin{equation}\label{eq:ztransf} 
  \widehat u(z) = \sum_{n=0}^{\infty} u(n) z^{n}, \quad \widehat y(z) 
  =\sum_{n=0}^{\infty} y_{n} z^{n}.
\end{equation}

In the infinite-dimensional setting, the fact that any contractive 
holomorphic operator-valued function can be represented in the form 
\eqref{transfunc} with $\bU$ unitary comes out of the Sz.-Nagy-Foia{\c s} 
model theory for completely non-unitary contraction operators; see 
\cite{SzNaFoBook10}).
 
There is a closely related but somewhat different theory of canonical 
functional models due to de Branges and Rovnyak \cite{deBrRo66, deBrRoBook} 
which relies on reproducing kernel Hilbert spaces. This is the direction we pursue in the present paper, assuming throughout that $\Uscr$ and $\Yscr$ are separable.

Let $\Gscr$ be a Hilbert space and let $\Bscr(\Gscr)$ denote the space of bounded linear operators on $\Gscr$. In general we say 
that a function $K : \Omega \times \Omega \to \cB(\cG)$ is a 
{\em positive kernel} on $\Omega$ if 
\begin{equation}   \label{posker}
 \sum_{i,j=1}^{N} \langle K(\omega_{i}, \omega_{j}) g_{j}, g_{i} 
 \rangle_{\cG} \ge 0
\end{equation}
for all choices of points $\omega_{1}, \dots, \omega_{N}$ in $\Omega$ 
and vectors $g_{1}, \dots, g_{N}\in\cG$.  The 
following theorem summarizes some useful equivalent characterizations 
of a positive $\cB(\cG)$-valued kernel on $\Omega$.

\begin{thm}  \label{T:RK}
    Given a Hilbert space $\cG$ and a function $K : \Omega 
    \times \Omega \to \cB(\cG)$, the following are equivalent:
    
    \begin{enumerate}
	\item The function $K$ is a positive kernel, i.e., condition 
	\eqref{posker} holds for all $\omega_{1}, \dots, \omega_{N}$ in $\Omega$ 
	and $g_{1}, \dots, g_{N} \in \cG$ for $N=1,2,\dots$.
	
	\item The function  $K$ is the \emph{reproducing kernel} of a reproducing kernel Hilbert 
	space $\cH(K)$, i.e., there is a Hilbert space $\cH(K)$ whose elements are functions 
	$f : \Omega \to \cG$ such that:
	\begin{enumerate}
	    \item For each $\omega \in \Omega$ and $g \in \cG$, the 
	    function $\zeta \mapsto K(\zeta, \omega) g$, $\zeta\in\Omega$, belongs to 
	    $\cH(K)$, and
	    \item the reproducing property
\begin{equation}\label{eq:reprodprop}
	    \langle f, K(\cdot, \omega) g \rangle_{\cH(K)} = \langle 
	    f(\omega), g \rangle_{\cG}
\end{equation}
	    holds for all $f \in \cH(K)$, $\omega\in\Omega$, and $g\in\cG$.
	     \end{enumerate}
	     
	     \item The function $K$ has a \emph{Kolmogorov decomposition}, i.e., there is a Hilbert space $\cF$ and 
	     a function $H 
	     : \Omega \to \cB(\cF, \cG)$ such that $K$ has the 
	     factorization
	     \begin{equation}    \label{Koldecom}
	     K(\zeta, \omega) = H(\zeta) H(\omega)^{*},\quad \zeta,\omega\in\Omega.
	     \end{equation}
	     Furthermore, one such factorization (often called canonical) is produced by 
	     taking $\cF = \cH(K)$ as defined in item 2 and $H(\zeta)$ equal to the 
	     point-evaluation map
	     $$ H(\zeta) = e(\zeta) : f \mapsto f(\zeta) ,\quad f \in \cH(K),\, \zeta\in\Omega.
	     $$
\end{enumerate}
\end{thm}

We will make frequent use of the following observation which is an immediate consequence of the reproducing property \eqref{eq:reprodprop}:

\begin{rem}\label{rem:kernelsdense}
In the notation of Theorem \ref{T:RK}, assume that $K$ is a reproducing kernel for the Hilbert space $\Hscr(K)$. Then the linear span 
$$\spn\set{\zeta \mapsto K(\zeta, \omega) g\mid \omega\in\Omega,\,g\in\Gscr}$$
is dense in $\Hscr(K)$.
\end{rem}

Given two separable Hilbert spaces $\cU$ and $\cY$,
we let $\cS(\D;\cU, \cY)$ denote the Schur class over the unit disk ${\mathbb D}$ consisting of functions
 $\phi : {\mathbb D} \to \cB(\cU, \cY)$ which are holomorphic on ${\mathbb D}$ with values $\phi(z)$  equal to contraction operators 
 from $\cU$ into $\cY$.
Given the Schur-class function $\phi$ on ${\mathbb D}$,  we associate the 
kernel
\begin{equation}\label{eq:discrDBobs}
  {\mathrm K}_{o}(z,w) = \frac{ 1 - \phi(z) \phi(w)^{*}}{1 - z \overline{w}}
\end{equation}
for $z, w $ in the unit disk ${\mathbb D}$.  It is well known that 
${\mathrm K}_{o}$ is a positive kernel; the proof is similar to Section 2 below. By the 
Moore-Aronszajn Theorem \cite[\S 2]{Aronsz50} one can associate the so-called reproducing-kernel Hilbert 
space ${\mathrm H}_{o}: = \cH({\mathrm K}_{o})$ to the kernel function ${\mathrm K}_o$. 
This space plays the role of the state space in the observable co-isometric (co-energy 
preserving) de Branges-Rovnyak canonical functional model for a Schur 
class function $\phi$.  We note that this functional model is of interest not only as an alternative to the Sz.-Nagy-Foia{\c s} model \cite{SzNaFoBook10}  for contraction operators 
(see \cite{deBrRo66, deBrangesEntire, BaKr87}), but also has found applications in the context of Lax-Phillips scattering theory \cite{NiVa89} and inverse scattering theory
\cite{AlDy84, AlDy85} as well as boundary Nevanlinna-Pick interpolation \cite{SaraBook, BoKh08}.
The following result can be found at least 
implicitly in the work of de Branges-Rovnyak and is given explicitly 
in this form in \cite{ADRSBook} and in \cite{BallBolo10}.

\begin{thm}  \label{T:coiso-real}
Suppose that the function $\phi$ is in the Schur class $\cS(\D;\cU, \cY)$ and let 
${\mathrm H_o} = \cH(\mathrm K_{o})$ be the associated de Branges-Rovnyak space with reproducing kernel \eqref{eq:discrDBobs}.  
Define operators $\mathrm A_o$, $\mathrm B_o$, $\mathrm C_o$, and $\mathrm D_o$ by
\begin{equation}\label{eq:deBobs}
 \begin{array}{ll}
      \mathrm A_o f:= z \mapsto {\displaystyle\frac{f(z) - f(0)}{z}}, & \quad 
      \mathrm B_ou := z \mapsto {\displaystyle\frac{\phi(z) - \phi(0)}{z}} u, \\
     \mathrm C_o f := f(0), & \quad \mathrm D_o u := \phi(0) u, \\ & f\in\mathrm H_o,\,u\in\Uscr,\,z\in\D.
  \end{array}
\end{equation}
  Then the operator matrix $\bU_o := \left[   \begin{smallmatrix} \mathrm A_o 
  & \mathrm B_o  \\ \mathrm C_o & \mathrm D_o \end{smallmatrix} \right]$ has the following 
  properties:
  \begin{enumerate} 
      \item The operator $\bU_o$ defines a co-isometry from
  $\sbm{{\mathrm H}_{o} \\ \cU}$ to  $\sbm{{\mathrm H}_{o} \\ \cY}$. 
  \item The pair $(\mathrm C_o,\mathrm A_o)$ is an {\em observable pair}, i.e.,
  $$
    \mathrm C_o \mathrm A_o^{n} f = 0 \text{ for all } n = 0 ,1,2, \dots \quad\Longrightarrow \quad
    f = 0 \text{ as an element of } {\mathrm H}_{o}.
  $$
  \item We recover $\phi(z)$ as $\phi(z) = \mathrm D_o + z \mathrm C_o (1 - z\mathrm A_o)^{-1} \mathrm B_o$, $z\in\D$.
  
  \item If $\left[ \begin{smallmatrix} A & B \\ C & D 
\end{smallmatrix} \right] : \sbm{\cX \\ \cU} \to \sbm{\cX \\ \cY}$ is another operator matrix with properties 1--3
above (with $\cX$ in place of $\mathrm H_{o}$), 
then there is a unitary operator $\Delta_o : {\mathrm H}_o \to \cX$ so 
that
$$
  \begin{bmatrix} \Delta & 0 \\ 0 & 1_{\cY} \end{bmatrix} \begin{bmatrix} 
      \mathrm A_o & \mathrm B_o \\ \mathrm C_o & \mathrm D_o \end{bmatrix} = \begin{bmatrix} A & B \\ C & 
      D \end{bmatrix} \begin{bmatrix} \Delta & 0 \\ 0 & 1_{\cU} 
  \end{bmatrix}.
  $$
\end{enumerate}
\end{thm}

If $\phi$ is in the Schur class $\cS(\D;\cU, \cY)$, then the function $\wtsmash\phi$ defined by 
$\wtsmash \phi(z):=\phi(\overline z)^*$, $z\in\D$, lies in $\cS(\D;\cY, \cU)$. Replacing $\phi$ by $\wtsmash\phi$ in \eqref{eq:discrDBobs} 
leads to the dual de Branges-Rovnyak kernel given by
\begin{equation}\label{eq:discrDBcontr}
  {\mathrm K}_{c}(z,w):= \frac{1 - \phi(\overline{z})^{*} 
\phi(\overline{w})}{ 1 - z \overline{w}}.
\end{equation}
The Hilbert space associated to this kernel plays the role of the state-space in the following controllable, 
isometric (energy-preserving) de Branges-Rovnyak canonical functional 
model:

\begin{thm}\label{T:iso-real}
Suppose that the function $\phi$ is in the Schur class $\cS(\D;\cU, \cY)$ 
and let ${\mathrm H}_{c} = \cH({\mathrm K}_{c})$ be the associated dual de Branges-Rovnyak 
space.  Define operators $\mathrm A_c$, $\mathrm B_c$, $\mathrm C_c$, and $\mathrm D_c$ by
\begin{equation}\label{eq:deBcontr}
\begin{array}{ll}
 \mathrm A_cg := z \mapsto z g(z) - \phi(\overline{z})^{*} \wtsmash g(0), 
    & \quad \mathrm B_cu := z \mapsto \big(1 - \phi(\overline{z})^{*} \phi(0)\big)u, \\
 \mathrm C_cg := \wtsmash g(0), & \quad \mathrm D_c u:= \phi (0) u,\\ 
 & g\in\mathrm H_c,\,u\in\Uscr,\, z\in\D,
\end{array}
\end{equation}
where $\wtsmash g(0)$ is the unique vector in $\cY$ such that
     \begin{equation} \label{tildeg(0)}
     \langle \wtsmash g(0), y \rangle_{\cY} = \left\langle g,
     z\mapsto\frac{\phi(\overline{z})^{*} - \phi(0)^{*}}{z} y 
     \right\rangle_{{\mathrm H}_{c}} \text{ for all } y \in \cY.
     \end{equation}
     Then the operator matrix $\bU_c := \left[   \begin{smallmatrix} 
     \mathrm A_c & \mathrm B_c \\ \mathrm C_c & \mathrm D_c \end{smallmatrix} \right]$ has the following 
     properties:
     \begin{enumerate} 
	 \item The operator $\bU_c$ defines an isometry from 
     $\sbm{{\mathrm H}_{c} \\ \cU}$ to  $\sbm{{\mathrm H}_{c} \\ \cY}$.
     
     \item The pair $(\mathrm A_c,\mathrm B_c)$ is a {\em controllable pair}, i.e.,
     $$\cspn\set{\mathrm A_c^{n} \mathrm B_c u\mid u\in\Uscr,\,n\geq0}  = {\mathrm H}_{c}.$$

     \item We recover $\phi(z)$ as $\phi(z) = \mathrm D_c + z \mathrm C_c (1 - z\mathrm A_c)^{-1} \mathrm B_c$, $z\in\D$.
     
     \item If $\left[ \begin{smallmatrix} A & B \\ C & D 
   \end{smallmatrix} \right] : \sbm{\cX \\ \cU} \to \sbm{\cX \\ \cY} $ is another operator matrix with properties 1--3 
   above (with $\cX$ in place of ${\mathrm H}_{c}$), 
   then there is a unitary operator $\Delta : 
   {\mathrm H}_{c} \to \cX$ so 
   that 
$$
\begin{bmatrix} \Delta  & 0 \\ 0 & 1_{\cY} \end{bmatrix} \begin{bmatrix} 
\mathrm A_c & \mathrm B_c \\ \mathrm C_c & \mathrm D_c \end{bmatrix} = \begin{bmatrix} A & 
B \\ C & D \end{bmatrix} \begin{bmatrix} \Delta & 0 \\ 0 & 1_{\cU} 
\end{bmatrix}.
$$
  \end{enumerate}
\end{thm}

The cases where the canonical model $\mathbf U_o$ and/or $\mathbf U_c$ is unitary can be characterized as follows:

\begin{thm}\label{thm:discrconschar}
The following assertions are equivalent:
\begin{enumerate}
\item The co-isometric observable canonical model ${\mathbf U}_{o}$ is unitary.
\item The following two conditions both hold:
\begin{align}   
 & {\rm H}_{o} \cap \{ \phi(\cdot) u \mid u \in \cU \} = \{0\}\quad\text{and}\quad
 \label{Uounitary1}  \\
 & \phi(z) u = 0 \text{ for all } z \in {\mathbb D} \quad\Longrightarrow\quad u = 0.
 \label{Uounitary2}
\end{align}
\item The maximal factorable minorant of $1 - \phi(z)^{*} \phi(z)$ is $0$, i.e., the only holomorphic $a \colon {\mathbb D} \to \cB(\cU, \cU')$ with the property
$$
  a(z)^{*}a(z) \le 1 - \phi(z)^{*} \phi(z),\quad z \in \C,~|z|=1
$$ 
is $a=0$.
\end{enumerate}

\noindent 
The following assertions are also equivalent:
\begin{enumerate}
\item The isometry ${\mathbf U}_{c}$ is unitary.
\item The following two conditions both hold:
\begin{align*}   
 & {\rm H}_{c} \cap \{ z\mapsto\phi(\overline z)^* y \mid y \in \cY \} = \{0\}\quad\text{and}\quad \\
 & \phi(\overline z)^* y = 0 \text{ for all } z \in {\mathbb D} \quad\Longrightarrow\quad y = 0.
\end{align*}
\item The maximal factorable minorant 
of $z\mapsto 1 - \phi(\overline z) \phi(\overline z)^{*}$ is $0$.
\end{enumerate}
\end{thm}

The equivalences of the conditions one and two can be found in \cite[Thms 3.2.3 and 3.3.3]{ADRSBook}. For instance, one easily sees that the conditions \eqref{Uounitary1} and \eqref{Uounitary2} both hold if and only if $\Ker{\mathbf U_o}=\zero$. In order to prove that the third assertion is equivalent to unitarity in the case of $\bU_o$, as a first step combine Lemma 8.2, Theorem 8.7, Corollary 8.8, and Theorem 9.1 in \cite{NiVa86bieb} to see that the zero-maximal-factorable-minorant condition on $1 - \phi(\cdot)^* \phi(\cdot)$ is equivalent to each column $\left[ \begin{smallmatrix} \mathrm A_o \\ \mathrm C_o \end{smallmatrix} \right]$ and $\left[ \begin{smallmatrix} \mathrm B_o \\ \mathrm D_o \end{smallmatrix}\right]$ of $\bU_o$ being isometric. It is then an elementary exercise to argue that the whole matrix $\bU_o = \left[ \begin{smallmatrix} \mathrm A_o & \mathrm B_o \\ \mathrm C_o & \mathrm D_o \end{smallmatrix} \right]$ is isometric if it is known to be contractive with each column isometric. The proof for the case of $\bU_c$ is the same, but with $\widetilde \phi$ in place of $\phi$ and with $\bU_c^*$ in place of $\bU_o$.

In addition to the functional models in Theorems \ref{T:coiso-real} and \ref{T:iso-real}, there is also a unitary functional model which combines $\mathbf U_o$ and $\mathbf U_c$; see e.g.\ Brodski\u\i\ \cite{Brod78}.

There is a parallel but less well developed theory 
 for the Schur class ${\mathcal S}({\mathbb C}^+; \cU, \cY)$ over the right half plane 
consisting of holomorphic functions on the right half plane ${\mathbb C}^+$ with values equal to contraction operators between the coefficient 
Hilbert spaces $\cU$ and $\cY$. See however \cite{Kuz96, Kr76} as well as \cite{LKMV95, BaVi03} for a more general algebraic curve setting.
In general, if  the $\cB(\cU, \cY)$-valued function $\varphi$ has the property that $\varphi$ extends to be holomorphic in a 
neighborhood of infinity  rather than in a neighborhood of the origin, 
it is natural to work with realizations of the form
\begin{equation}   \label{transfunc'}
\varphi(\mu) = D + C (\mu - A)^{-1} B.
\end{equation}
It is well known that,
given any $\cB(\cU, \cY)$-valued  function holomorphic on a  neighborhood of 
$\infty$ in the complex plane, there is a Hilbert space $\cX$ (the 
state space) and a system matrix
$$
\bU = \begin{bmatrix} A & B \\ C & D \end{bmatrix} : 
    \begin{bmatrix} \cX \\ \cU \end{bmatrix} \to \begin{bmatrix} \cX \\ \cY \end{bmatrix}
$$
so that $\varphi$ has a representation as in \eqref{transfunc'}. 
If we introduce the continuous-time input/state/output linear system
\begin{equation}   \label{CT-sys}
\Sigma_{\bU} : \left\{ \begin{array}{rcl} 
\dot x(t) & = & A x(t) + B u(t) \\
  y(t) & = & C x(t) + D u(t), \end{array} \right.
\end{equation}
then application of the Laplace transform
\begin{equation}\label{eq:laptransf}
  \widehat x(\mu) = \int_{0}^{\infty} e^{-\mu t} x(t)\,{\tt d}t
\end{equation}
leads to the relation
$$ 
   \widehat y(\mu) = \varphi(\mu) \widehat u(\mu)
$$
whenever $\big(u(\cdot), x(\cdot), y(\cdot)\big)$ is a trajectory of the system 
\eqref{CT-sys} with state-vector $x$ satisfying the zero initial 
condition $x(0) = 0$.

The generalized form for the operator matrix 
$\bU$ appropriate for the Schur class over ${\mathbb C}^{+}$ was 
first worked out by independently by \v{Smuljan} \cite{Smul86} and Salamon \cite{Sala87,Sala89}. Salamon gave a well-posed realization of an holomorphic function on $\cplus$ which is bounded on some complex right-half plane. Later, in \cite{ArNu96}, Arov-Nudelman specialized to the case of a Schur function, giving a passive realization. The generalized form for $\bU$ has since been 
refined into the notion of scattering 
conservative/energy-preserving/co-energy-preserving system node; see 
\cite{StafBook} for a comprehensive treatment, and also \cite{StafMTNS02, BaSt06}.   
  The analogue for the continuous-time 
setting of {\em co-isometric system matrix}
occurring in the discrete-time setting is a {\em co-energy preserving system 
node} while 
the analogue for the continuous-time setting of {\em isometric system 
matrix} occurring in the discrete-time setting is {\em energy 
preserving system node}  (precise definitions to come in Sections \ref{sec:sysnode}  below).

However, what has not been done to this point for the realization 
theory is the analogues of Theorems \ref{T:coiso-real} 
and \ref{T:iso-real} for $\varphi$ in the Schur class over 
${\mathbb C}^{+}$.  By using the right-half plane versions of the 
de Branges-Rovnyak kernels ${\mathrm K}_{o}$ and ${\mathrm K}_{c}$, 
namely,
\begin{equation}   \label{o/ckernels}
 K_{o}(\mu, \lambda) = \frac{1 - \varphi(\mu) \varphi(\lambda)^{*}}{ \mu + \overline{\lambda}},
 \quad
 K_{c}(\mu, \lambda) = \frac{ 1 - \varphi(\overline{\mu})^{*} 
 \varphi(\overline{\lambda})}{\mu + \overline{\lambda}},
\end{equation}
combined with the precise formalism of scattering energy-conserving 
and scattering co-energy-conserving operator nodes, in this 
paper we obtain complete analogues of Theorems \ref{T:coiso-real} and 
\ref{T:iso-real} for the continuous-time setting. Due to complications with unbounded operators and rigged Hilbert 
spaces, the formulas and analysis have a quite different flavor from 
that in the discrete-time/unit-disk setting.

The positivity of the kernels \eqref{o/ckernels} is proved in Section \ref{sec:statespaces}, and in Section \ref{sec:coenpres} we establish the following continuous-time analogue of Theorem \ref{T:coiso-real}:

\begin{thm}  \label{T:coenpres-real}
Suppose that the function $\varphi$ is in the Schur class $\cS(\cplus;\cU, \cY)$ and let 
$\cH_o = \cH(K_{o})$ be the associated de Branges-Rovnyak space with reproducing kernel $K_o$ in \eqref{o/ckernels}.
Define the following unbounded operator, which maps a dense subspace of $\sbm{\cH_{o}\\\Uscr}$ into $\sbm{\cH_{o}\\\Yscr}$:
 \begin{equation}\label{eq:coisomSdefIntro}
  \SysNode_o  : \bbm{x\\u}\mapsto\bbm{z\\y},\quad\text{where}
 \end{equation}
  \begin{equation} \label{defz'Intro}
  z(\mu) := \mu x(\mu)+\varphi(\mu)u-y,\quad \mu\in\cplus,\quad\text{and}
  \end{equation}
  \begin{equation}   \label{defy'Intro}
  y  := \lim_{\re\eta\to\infty} \eta x(\eta)+\varphi(\eta)u,\quad\text{defined on}
  \end{equation}
  \begin{equation*}
  \dom{ \SmallSysNode_o }  := \left\{  \bbm{x\\u}\in\bbm{\Hscr_o\\\Uscr}\bigmid 
  \exists y\in\Yscr:~ z \text{ defined in \eqref{defz'Intro} lies in } \cH_o \right\}.
  \end{equation*}
  Then for every $\sbm{ x \\ u } \in \dom{ \SmallSysNode_o}$, the $y \in \cY$ such that $z$ given in \eqref{defz'Intro} lies in $\cH_o$ is unique and it is given by \eqref{defy'Intro}. Moreover, the operator $\SmallSysNode_o$ has the following properties:
  \begin{enumerate} 
  \item The operator $\SmallSysNode_o$ is an observable co-energy-preserving system node.

  \item The operator $\SmallSysNode_o$ is a realization of $\varphi$, i.e., we recover $\varphi(\mu)$ through an appropriate generalization of \eqref{transfunc'}.
  
  \item If $\SmallSysNode : \sbm{\cX \\ \cU}\supset \dom{\SmallSysNode} \to \sbm{\cX \\ \cY}$ is another operator with properties 1--2
above (with $\cX$ in place of $\cH_{o}$), then there is a unitary operator $\Delta : \cH_o \to \cX$ so that $\sbm{ \Delta & 0 \\ 0 & 1_{\cU} }$ maps $\dom{\SmallSysNode_o}$ one-to-one onto $\dom{\SmallSysNode}$ and 
$$
  \bbm{\Delta & 0 \\ 0 & 1_{\cY}}\SysNode_o = \SysNode \bbm{ \Delta & 0 \\ 0 & 1_{\cU} }.
$$
Hence the system nodes $\sbm{ A \& B \\ C \& D }$ and $\SmallSysNode_o$ are unitarily similar.
\end{enumerate}
\end{thm}

It is also possible to decompose $\SmallSysNode_o$ into unbounded operators $A_o$, $B_o$, and $C_o$ which together with $\varphi$ determine $\SmallSysNode_o$ uniquely, similar to Theorem \ref{T:coiso-real}; see Section \ref{sec:sysnodedef} below. This involves a rigging of the state space and hence it is too technically involved to be presented in the introduction. We have the following analogue of Theorem \ref{T:iso-real}; the proofs and more details can be found in Section \ref{sec:enpres}:

\begin{thm}  \label{T:enpres-real}
Suppose that the function $\varphi$ is in the Schur class $\cS(\cplus;\cU, \cY)$ and let 
$\cH_c = \cH(K_{c})$ be the associated de Branges-Rovnyak space with reproducing kernel $K_c$ in \eqref{o/ckernels}.
There exists a system node $\SmallSysNode_c:\sbm{\Hscr_c\\\Uscr}\supset \dom{\SmallSysNode_c}\to\sbm{\Hscr_c\\\Yscr}$ that for arbitrary $\sbm{x\\u}$ in its domain and $\lambda\in\cplus$ satisfies
\begin{equation}\label{eq:contrrealgenIntro}
  \SysNode_c\bbm{x\\u} = \bbm{\mu\mapsto-\mu x(\mu)-\varphi(\overline\mu)^*\gamma_\lambda+\big(1-\varphi(\overline\mu)^*\varphi(\overline\lambda)\big)u
    \\\gamma_\lambda + \varphi(\overline\lambda)u},
\end{equation}
$\mu\in\cplus$, where $\gamma_\lambda\in\Yscr$ is uniquely determined by $\lambda$ and $\sbm{x\\u}$.

Moreover, the operator $\SmallSysNode_c$ has the following properties:
  \begin{enumerate} 
  \item The operator $\SmallSysNode_c$ is a controllable energy-preserving system node.

  \item The operator $\SmallSysNode_c$ is a realization of $\varphi$, i.e., we recover $\varphi(\mu)$ through the appropriate generalization of \eqref{transfunc'} mentioned earlier.
  
  \item If $\SmallSysNode : \sbm{\cX \\ \cU}\supset \dom{\SmallSysNode} \to \sbm{\cX \\ \cY}$ is another operator matrix with properties 1--2
above (with $\cX$ in place of $\cH_c$), then there is a unitary operator $\Delta : \cH_c \to \cX$ so that $\sbm{ \Delta & 0 \\ 0 & 1_{\cU} }$ maps $\dom{\SmallSysNode_c}$ one-to-one onto $\dom{\SmallSysNode}$ and 
$$
  \bbm{\Delta & 0 \\ 0 & 1_{\cY}}\SysNode_c = \SysNode \bbm{ \Delta & 0 \\ 0 & 1_{\cU} }.
$$
\end{enumerate}
\end{thm}

While the papers \cite{ArNu96} and \cite{StafMTNS02} worked with 
linear-fractional change of variables to derive the continuous-time 
result from the discrete-time result, a more direct geometric 
approach based on the ``lurking isometry'' technique was used in \cite{BaSt06}.
The approach in the present paper is similar to the single-variable specialization of the work of Ball-Bolotnikov \cite{BallBolo10} for the discrete-time setting, to some extent using intuition from \cite{Kuru10}. The main difference compared to \cite{BaSt06} is that the canonical form of the Kolmogorov 
   factorization of the kernel $K_{c}$ (as given in part 3 of 
   Theorem \ref{T:RK}) leads to \emph{explicit functional formulas} for the system-nodes $\SmallSysNode_o$ and $\SmallSysNode_c$ above.
It should also be pointed out that \emph{conservative} realizations are presented in \cite{BaSt06} (and many of the other references below), but in the present paper we study energy preserving and co-energy preserving realizations, which are in a certain sense only semi-conservative.

We mention that other work of de Branges-Rovnyak (the first part of 
\cite{deBrRo66}) and  of de Branges \cite{deBrangesEntire} uses reproducing kernel 
Hilbert spaces consisting of entire functions based on positive 
kernels associated with Nevanlinna-class rather than Schur-class 
functions. (The Nevanlinna class consists of holomorphic, even entire, functions mapping the 
upper half plane into an operator with positive imaginary part.) This leads to models for symmetric 
operators with equal deficiency indices. See \cite{BeHaSn08,BeHaSn09} for recent developments in this direction, which is separate 
from what we pursue here.

Also in \cite{BeHaSn08,BeHaSn09} a linear-fractional transformation is used to transfer knowledge of Schur 
functions on $\D$ to Nevanlinna families on $\C\setminus\R$. In the present article we avoid the use of such 
transformations in the development of the realization theory in order to expose the intricacies 
of the continuous-time case; only in Section \ref{sec:recovering} we describe how to recover the original de Branges-Rovnyak 
models from the models we present in Sections \ref{sec:enpres} and \ref{sec:coenpres} using a linear-fractional transformation. A functional model (as a self-adjoint linear relation) for arbitrary normalized generalized Nevanlinna pairs has been worked out directly in $\C\setminus\R$ in \cite{Nei10}.

A general unifying formulation of the de Branges-Rovnyak models has recently been worked out by Arov-Kurula-Staffans (see \cite{ArKuStCanon})
for the continuous-time setting as an extension to continuous time of the earlier discrete-time realization results in \cite{ArSt2Canon,ArStConserv}. It is possible to derive Theorems \ref{T:coenpres-real} and \ref{T:enpres-real} from \cite{ArKuStCanon} and the method, outlined in Section \ref{sec:conc} below is in principle straightforward. However, filling in the details is a rather lengthy process, and for this reason we have chosen to give direct proofs of Theorems \ref{T:coenpres-real} and \ref{T:enpres-real} here that do not rely on \cite{ArKuStCanon}.

There have also been a number of extensions of Theorems 
\ref{T:coiso-real} and \ref{T:iso-real} to multi-variable settings;
see \cite{BallBolo10} for ball and polydisk versions and \cite{AMcCY12, 
BK-V12} for polyhalfplane versions.

\kern 5mm\noindent
{\bf Notation}.  

\newbox\BoxA

\newcommand{\Symb}[2]{\vskip 3pt\noindent
\setbox\BoxA=\hbox{$#1$:\quad}%
\hangindent=3 cm%
\ifdim\wd\BoxA > 3.3cm \unhbox\BoxA #2.\else\hbox to 3cm{$#1$:\hfil}#2.\fi}

\Symb{\cplus}{The complex right-half plane $\set{\lambda\in\C\mid \re\lambda>0}$}
\Symb{\overline X}{The closure of the normed space $X$}
\Symb{\Ipdp\cdot\cdot_\Xscr,\|\cdot\|_\Xscr}{The inner product and norm of $\Xscr$, respectively}
\Symb{\spn \,\Xi}{The linear span of the set $\Xi$}
\Symb{\BLO(\Uscr,\Yscr), \BLO(\Uscr)}{The space of bounded linear operators from
$\Uscr$ to $\Yscr$ and on $\Uscr$, respectively}
\Symb{\dom A,\range A}{The domain and range of the operator $A$}
\Symb{\Ker A,\res A}{The null-space and the resolvent set of the operator $A$}
\Symb{\Xscr_1\subset\Xscr\subset\Xscr_{-1}}{Rigged Hilbert spaces associated to $A:\Xscr\supset\dom A\to\Xscr$, with norms constructed using some $\beta\in\cplus$}
\Symb{\Xscr_1^d\subset\Xscr\subset\Xscr_{-1}^d}{The rigged Hilbert spaces associated to $A^*$, with norms constructed using $\overline\beta\in\cplus$, where $\beta$ is used in the rigging corresponding to $A$. $\Xscr_{\pm1}^d$ is identified with the dual of $\Xscr_{\mp1}$ using $\Xscr$ as pivot space}
\Symb{A|_\Xscr}{The unique extension of the operator $A\in\Bscr(\Xscr_1,\Xscr)$ to an operator in $\Bscr(\Xscr,\Xscr_{-1})$}
\Symb{1_\Xscr, 1}{The identity operator on $\Xscr$}
\Symb{\Uscr,\Yscr}{Separable Hilbert spaces, the input and output space, respectively}
\Symb{\sbm{\Xscr\\\Uscr}}{The orthogonal direct sum of the Hilbert spaces $\Xscr$ and $\Uscr$}
\Symb{e(\mu)}{The (bounded) point-evaluation operator in $H^2(\cplus;\Uscr)$ and $H^2(\cplus;\Yscr)$}
\Symb{e(\lambda)^*}{The (bounded) adjoint of $e(\lambda)$. Premultiplies an element of $\C$ or a vector space by the (scalar) kernel $k(\mu,\lambda)=\frac {1}{\mu+\overline\lambda}$ of $H^2(\C)$, so that $e(\lambda)^*u$ is the function $\mu\mapsto\frac{u}{\mu+\overline\lambda}$, $\mu,\lambda\in\cplus$, $u\in\Uscr$}
\Symb{\Sscr(\cplus;\Uscr,\Yscr)}{The Schur class on the right-half plane which consists of $\BLO(\Uscr,\Yscr)$-valued holomorphic functions whose values are contractions}
\Symb{M_\varphi}{The multiplication operator on $H^2(\cplus;\Uscr)$ with symbol $\varphi\in\Sscr(\cplus;\Uscr,\Yscr)$, i.e., $(M_\varphi f)(\lambda)=\varphi(\lambda)f(\lambda)$, $\lambda\in\cplus$}
\Symb{\SmallSysNode_o}{The observable co-energy-preserving functional model for $\varphi\in\Sscr(\cplus;\Uscr,\Yscr)$}
\Symb{K_o}{The reproducing kernel $K_o(\mu,\lambda)=\frac{1_\Yscr-\varphi(\mu)\varphi(\lambda)^*}{\mu+\overline\lambda}$; takes values in $\BLO(\Yscr)$}
\Symb{\Hscr_o}{The de Branges space with reproducing kernel $K_o$. This is the state space for $\SmallSysNode_o$ and it is contractively contained in $H^2(\cplus;\Yscr)$}
\Symb{e_o(\mu)}{The point-evaluation operator in $\Hscr_o$}
\Symb{e_o(\lambda)^*}{The adjoint of $e_o(\lambda)$, maps $y\in\Yscr$ into $K_o(\cdot,\lambda)y$, $\lambda\in\cplus$}
\Symb{\wtsmash\varphi}{The function $\wtsmash\varphi(\mu)=\varphi(\overline\mu)^*$, $\mu\in\cplus$, which is an element of $\Sscr(\cplus;\Yscr,\Uscr)$ if $\varphi\in\Sscr(\cplus;\Uscr,\Yscr)$}
\Symb{\iota}{The inclusion operator $\Hscr_o\to H^2(\cplus;\Yscr)$, with $\iota^*=(1-M_\varphi M_\varphi^*)$, or $\Hscr_c\to H^2(\cplus;\Uscr)$, with $\iota^*=(1-M_{\wtsmash\varphi} M_{\wtsmash\varphi}^*)$}
\Symb{\SmallSysNode_c}{The controllable energy-preserving functional model for $\varphi\in\Sscr(\cplus;\Uscr,\Yscr)$}
\Symb{K_c}{The reproducing kernel $K_c(\mu,\lambda)=\frac{1_\Uscr-\wtsmash\varphi(\mu)\wtsmash\varphi(\lambda)^*}{\mu+\overline\lambda}$; takes values in $\BLO(\Uscr)$}
\Symb{\Hscr_c}{The de Branges space with reproducing kernel $K_c$. This is the state space for the $\SmallSysNode_c$, contractively contained in $H^2(\cplus;\Uscr)$}
\Symb{e_c(\mu)}{The point-evaluation operator in $\Hscr_c$}
\Symb{e_c(\lambda)^*}{The adjoint of $e_c(\lambda)$, maps $u\in\Uscr$ into $K_c(\cdot,\lambda)u$, $\lambda\in\cplus$}
\Symb{\Delta}{Unitary intertwinement operator from $\Hscr_o$ or $\Hscr_c$ to $\Xscr$}
\Symb{\Xi_\alpha}{Unitary intertwinement operator from $\mathrm H_{o,\alpha}$ to $\Hscr_o$ or from $\mathrm H_{c,\overline\alpha}$ to $\Hscr_c$}

\section{The de Branges-Rovnyak spaces $\Hscr_o$ and $\Hscr_c$ over $\cplus$}\label{sec:statespaces}

The topic of this section is the development of the state spaces of the functional models presented in the 
introduction. We begin by proving that the kernels \eqref{o/ckernels} are positive kernels,
and therefore reproducing kernels of $\Hscr_o$ and $\Hscr_c$. The reader is assumed to be 
familiar with Hardy spaces over $\cplus$; otherwise see e.g.\ \cite[Sect. A.6]{CuZwBook}. It is 
important that $\Uscr$ and $\Yscr$ are separable.

Every $\varphi\in\Sscr(\cplus;\Uscr,\Yscr)$ lies in $H^\infty(\cplus;\Bscr(\Uscr,\Yscr))$ 
and therefore the multiplication operator $M_\varphi$ with symbol $\varphi$ maps 
$H^2(\cplus;\Uscr)$ into $H^2(\cplus;\Yscr)$, and $\|M_\varphi\|=\|\varphi\|_{H^\infty}$; 
see \cite[Theorem A.6.26]{CuZwBook}. We need the following lemma in order to show that the kernel $K_o(\mu,\lambda)$ is nonnegative:

\begin{lem}\label{lem:MphiProps}
Let $\varphi\in\Sscr(\cplus;\Uscr,\Yscr)$ and denote the point-evaluation operator in $H^2(\cplus;\Yscr)$ by $e_{H^2(\cplus;\Yscr)}(\cdot)$. 
The following claims are true:
\begin{enumerate}
\item The adjoint of $e_{H^2(\cplus;\Yscr)}(\lambda)$ is the operator of premultiplication with the reproducing kernel $k_\Yscr$ of $H^2(\cplus;\Yscr)$: 
$$
 e(\lambda)^*y=\mu\mapsto k_\Yscr(\mu,\lambda)y,\quad y\in\Yscr,\,\mu,\lambda\in\cplus, \quad k_\Yscr(\mu,\lambda) = \frac {1_\Yscr}{\mu+\overline\lambda}.
$$
\item The operator $M_{\varphi}^*$ has the following action on the kernel functions in $H^2(\cplus;\Yscr)$:
$$ M_\varphi^*\, e_{H^2(\cplus;\Yscr)}(\lambda)^*y=e_{H^2(\cplus;\Uscr)}(\lambda)^*\,\varphi(\lambda)^*y,\quad \lambda\in\cplus,\,y\in \Yscr.$$
\item The function  $K_o$ defined in \eqref{o/ckernels} can be factored as 
\begin{equation}\label{eq:CoEnPresKernOlofFact}
\begin{aligned}
  K_o(\mu,\lambda)&=e_{H^2(\cplus;\Yscr)}(\mu)\,(1_{H^2(\cplus;\Yscr)}-M_{\varphi}M_{\varphi}^*)\,e_{H^2(\cplus;\Yscr)}(\lambda)^*,\qquad\\&\mu,\lambda\in\cplus.
\end{aligned}
\end{equation}
\end{enumerate}
\end{lem}

In the sequel we simplify the notation, so that $k(\cdot,\lambda)$ denotes a kernel function in $H^2(\cplus;\Uscr)$, $H^2(\cplus;\Yscr)$, or $H^2(\cplus;\C)$, where it is clear from the context which one to choose. Similarly, the point-evaluation operator at $\mu$ on a possibly vector-valued $H^2$ space is simply denoted by $e(\mu)$.

\begin{proof}
We have the following short arguments:
\begin{enumerate}
\item It follows from residue calculus that $k$ is the reproducing kernel of $H^2(\cplus)$; see  \cite{HoffmanBook}.  That 
$e(\lambda)^*y=k(\cdot,\lambda)y$ then follows from the reproducing kernel property \eqref{eq:reprodprop}. 
\item As probably first observed in \cite{ShiWa}, by the reproducing kernel property \eqref{eq:reprodprop}, we have for all $u\in H^2(\cplus;\Uscr)$, $y\in\Yscr$, 
and $\lambda\in\cplus$:
$$
\begin{aligned}
  \Ipdp{u}{M_{\varphi}^*e(\lambda)^*y}_{H^2(\cplus;\Uscr)} &= \Ipdp{M_{\varphi}u}{e(\lambda)^*y}_{H^2(\cplus;\Yscr)} =
   \Ipdp{(M_{{\varphi}}u)(\lambda)}{y}_\Yscr \\
  &= \Ipdp{\varphi(\lambda)u(\lambda)}{y}_\Yscr = \Ipdp{u}{e(\lambda)^*\varphi(\lambda)^*y}_\Uscr.
\end{aligned}
$$

\item For all $\mu,\lambda\in \cplus$ and $y,\gamma\in\Yscr$, by using assertion 2 (in the fourth equality) we have:
\begin{equation}\label{eq:CoIsomKernOlofFactor}
\begin{aligned}
  \Ipdp{K_o(\mu,\lambda)y}{\gamma}_\Yscr &= \Ipdp{\frac{1}{\mu+\overline\lambda}y}{\gamma}_\Yscr-\Ipdp{\frac{\varphi(\mu)
  \varphi(\lambda)^*}{\mu+\overline\lambda}y}{\gamma}_\Yscr\\
  &= \Ipdp{k(\mu,\lambda)y}{\gamma}_\Yscr-\Ipdp{k(\mu,\lambda) \varphi(\lambda)^*y}{\varphi(\mu)^*\gamma}_\Uscr\\
  &= \Ipdp{e(\lambda)^*y}{e(\mu)^*\gamma}_{H^2(\cplus;\Yscr)}\\
  &\qquad -\Ipdp{e(\lambda)^*\varphi(\lambda)^*y}{e(\mu)^*\varphi(\mu)^*\gamma}_{H^2(\cplus;\Uscr)}\\
  &= \Ipdp{e(\lambda)^*y}{e(\mu)^*\gamma}_{H^2(\cplus;\Yscr)}\\
  &\qquad -\Ipdp{M_{\varphi}^*e(\lambda)^*y}{M_{\varphi}^*e(\mu)^*\gamma}_{H^2(\cplus;\Uscr)}\\
  &= \Ipdp{(1-M_{\varphi}M_{\varphi}^*)e(\lambda)^*y}{e(\mu)^*\gamma}_{H^2(\cplus;\Yscr)}\\
  &= \Ipdp{e(\mu)(1-M_{\varphi}M_{\varphi}^*)e(\lambda)^*y}{\gamma}_\Yscr,
\end{aligned}
\end{equation}
\end{enumerate}

\ and this completes the proof.
\end{proof}

From this lemma it is easy to show that $K_o$ is a positive kernel.

\begin{thm}\label{thm:coisomkernelnonneg}
If $\varphi\in \Sscr(\cplus;\Uscr,\Yscr)$, then the function $K_o(\mu, \lambda)$ defined in \eqref{o/ckernels} is a positive kernel.
\end{thm}

\begin{proof}
For $\varphi \in \Sscr(\cplus;\Uscr,\Yscr)$, the multiplication operator $M_\varphi : H^2({\mathbb C}^+; \cU) \to H^2({\mathbb C}^+; \cY)$ is contractive,
  $\|M_\varphi \| \le 1$, since $\|\varphi\|_{H^\infty(\cplus)}\leq1$.  Hence $1 - M_\varphi M_\varphi^* \ge 0$ as an operator on $H^2({\mathbb C}^+; \cY)$ and thus it has a bounded
   positive square root  $(1 - M_\varphi M_\varphi^*)^{1/2}$ on $H^2({\mathbb C}^+; \cY)$.  From the identity \eqref{eq:CoEnPresKernOlofFact} we see that 
   $K_o(\mu, \lambda)$ has a Kolmogorov decomposition \eqref{Koldecom} with
   $$
   H(\mu) = e(\mu) (1 - M_\varphi M_\varphi^*)^{1/2} : \cH(K_o) \to \cY.
   $$
   We conclude from Theorem \ref{T:RK} that $K_o$ is a positive kernel.
\end{proof}

We denote the Hilbert space with reproducing kernel $K_o$ by $\Hscr_o:=\Hscr(K_o)$. Replacing $\varphi$ by $\wtsmash\varphi(\mu):=\varphi(\overline\mu)^*$, $\mu\in\cplus$, and 
swapping the roles of $\Uscr$ and $\Yscr$, we turn the kernel $K_o$ into the kernel $K_c$ in
\eqref{o/ckernels}. Applying Lemma \ref{lem:MphiProps} and Theorem \ref{thm:coisomkernelnonneg} 
to $\wtsmash\varphi$, we obtain the following result:

\begin{cor}\label{cor:isomkernelpos}
If $\varphi\in\Sscr(\cplus;\Uscr,\Yscr)$ then the $\Bscr(\Uscr)$-valued function $K_c(\mu,\lambda)$ 
is a positive kernel on $\cplus\times\cplus$. Denoting $\Hscr_c:=\Hscr(K_c)$, we have that the kernel functions of $\Hscr_c$ and $H^2(\cplus;\Uscr)$ are related by $K_c(\cdot,\lambda)u=(1-M_{\wtsmash\varphi}M_{\wtsmash\varphi}^*)\,k(\cdot,\lambda)u$ for all $\lambda\in\cplus$ and $u\in\Uscr$.
\end{cor}

An equivalent way of defining $\Hscr_o$ is to set
$$
  \Hscr_o := \{ f \colon \cplus \underset{\rm holomorphic}\to \cY \bigmid \|f\|_{\Hscr_o} < \infty\},
$$
and to define the norm in $\Hscr_o$ by
$$
  \| f\|^{2}_{\cH_o}: = \sup \{ \| f + M_\varphi \,g\|^{2}_{H^{2}(\cplus;\Yscr)} - \|g \|^{2}_{H^{2}(\cplus;\cU)} \bigmid g \in H^{2}(\cplus;\cU)\}.
$$
It can be shown that this norm equals the norm induced by the reproducing kernel $K_o$. This corresponds to the original definition of $\mathrm H_o$ by de Branges and Rovnyak. To give the uninitiated reader better perspective on de Branges-Rovnyak spaces, we further mention the following well-known operator-range characterization of these spaces. For further development of this point of view in the unit disk setting see e.g.~\cite{SaraBook}.

\begin{thm}   \label{T:operatorrange}
Let $\varphi$ be a function in the Schur class $\cS({\mathbb C}^+; \cU, \cY)$.  Then:
\begin{enumerate}
\item The space $\cH_o$ can be identified as a set with the operator range
\begin{equation}  \label{operatorrange}
  \cH_o = \range{(1 - M_\varphi M_\varphi^*)^{1/2}} \subset H^2({\mathbb C}^+; \cY)
 \end{equation}
 with norm given by
 \begin{equation}   \label{lifted norm}
 \| (1 - M_\varphi M_\varphi^*)^{1/2} g \|_{\cH_o} = \| Q g \|_{H^2(\cplus;\Yscr)},\quad g\in H^2(\cplus;\Yscr),
 \end{equation}
 where $Q$ is the orthoprojection of $H^2({\mathbb C}^+; \cY)$ onto $\left( \Ker{1 - M_\varphi M_\varphi^*} \right)^\perp$.
 
 \item The inclusion map
 $$
 \iota  : f \in \cH_o \mapsto f \in H^2({\mathbb C}^+;, \cY)
 $$
 is contractive, i.e.,
 $$
 \| f \|_{H^2({\mathbb C}^+; \cY)} \le \| f \|_{\cH_o} \text{ for all } f \in \cH_o,
 $$
 with adjoint $\iota^* : H^2({\mathbb C}^+; \cY)\to\Hscr_o$  given by
 $$
  \iota^*  = 1 - M_\varphi M_\varphi^*.
 $$
  \end{enumerate}
   Analogous results with $\cH_c$ in place of $\cH_o$ are obtained by replacing $\varphi$ by $\wtsmash \varphi$.
\end{thm}
\begin{proof}  The result is well known among experts but we provide 
 a proof for the sake of completeness. The first step is to prove Assertion 1.

Define the space $\wtsmash \cH_{o}$ by 
$$
\wtsmash \cH_{o} := \range {1 - M_{\varphi} 
M_{\varphi}^{*})^{1/2}} \subset H^{2}({\mathbb C}^{+}; \cY)
$$
with norm given by \eqref{lifted norm} and let $f \in \wtsmash \cH_{o}$.
Set $W = 1 - M_{\varphi} M_{\varphi}^{*}$ on $H^{2}({\mathbb C}^{+}; 
\cY)$, so that $\wtsmash \cH_{o} = \range {W^{1/2}}$.  
From \eqref{eq:CoEnPresKernOlofFact} we see that
$e_{o}(\lambda)^{*} = W e(\lambda)^*$, so in 
particular $e_{o}(\lambda)^{*} y \in \wtsmash \cH_{o}$ for each 
$\lambda \in {\mathbb C}^{+}$ and $y \in \cY$.  Furthermore, for $f = 
W^{1/2} g \in \wtsmash \cH_{o}$, we compute using \eqref{lifted norm}:
\begin{align*}
    \langle f, e_{o}(\lambda)^{*} y \rangle_{\wtsmash \cH_{o}} & = 
    \langle W^{1/2} g, W e(\lambda)^{*}y \rangle_{\wtsmash \cH_{o}} = 
    \langle Qg, QW^{1/2} e(\lambda)^{*}y \rangle_{H^{2}({\mathbb C}^{+}; \cY)} \\
  & = \langle W^{1/2} g, e(\lambda)^{*}y \rangle_{H^{2}({\mathbb C}^{+}; \cY)}
  = \langle f(\lambda), y \rangle_{\cY}.
\end{align*}
This shows that $e_{o}(\lambda)^{*}=K_o(\cdot,\lambda)$ works as the reproducing 
kernel for the space $\wtsmash \cH_{o}$, and since the positive kernel $e_o(\lambda)^*$ determines its reproducing kernel Hilbert space uniquely, we conclude that $\cH_{o} = \wtsmash \cH_{o}$.

Contractive containment of $\Hscr_o$ in $H^2(\cplus;\Yscr)$ follows from the following observation:
\begin{align*}
    \| f \|_{\cH_{0}} & = \| g \|_{H^{2}({\mathbb C}^{+}; \cY)} \ge 
    \|(1 - M_{\varphi} M_{\varphi}^{*})^{1/2}\, g \|_{H^{2}({\mathbb 
    C}^{+}; \cY)} = \| f \|_{H^{2}({\mathbb C}^{+}; \cY)},
    \end{align*}
where we used that $1 - M_{\varphi} M_{\varphi}^{*}$ is contractive on $H^{2}({\mathbb C}^{+}; \cY)$.

Since $e_o(\lambda)$ is the restriction of $e_{H^2(\cplus;\Yscr)}$ to $\Hscr_o$, the identity \eqref{eq:CoEnPresKernOlofFact} amounts to the 
    operator identity
\begin{equation}\label{eq:HoEvalOp}
    e_{o}(\lambda)^{*} = (1 - M_{\varphi} M_{\varphi}^{*}) 
    e(\lambda)^{*},\quad \lambda\in\cplus.
\end{equation}
Using \eqref{eq:HoEvalOp}, we obtain that $\iota^* e(\lambda)^*y = (1 - M_\varphi M_\varphi^*)\,e(\lambda)^*y$ for all $\lambda\in\cplus$ and $y\in\Yscr$. Indeed, it holds for all $x\in\Hscr_o$ that
$$
  \Ipdp{(1 - M_\varphi M_\varphi^*)\,e(\lambda)^*y}{x}_{\Hscr_o}=\Ipd{y}{x(\lambda)}_\Yscr=\Ipdp{e(\lambda)^*y}{\iota_o x}_{H^2(\cplus;\Yscr)},
$$
and taking limits of finite linear combinations, we obtain that $\iota^*  = 1 - M_\varphi M_\varphi^*$.
\end{proof}

Recall that $\varphi\in\Sscr(\cplus;\Uscr,\Yscr)$ is called \emph{inner} if $\varphi$ has isometric boundary values a.e.\ on the imaginary line.

\begin{cor}\label{cor:inner}
If $\varphi$ is inner, then $M_\varphi$ is isometric from $H^2(\cplus;\Uscr)$ into $H^2(\cplus;\Yscr)$, and $(1-M_\varphi M_\varphi^*)^{1/2}=1-M_\varphi M_\varphi^*$. The operator $1-M_\varphi M_\varphi^*$ is the orthogonal projection of $H^2(\cplus;\Yscr)$ onto $H^2(\cplus;\Yscr)\ominus \big(M_\varphi\,H^2(\cplus;\Uscr)\big)$ and this orthogonal complement equals $\Hscr_o$ isometrically.
\end{cor}
\begin{proof}
That $M_\varphi$ is isometric follows from
$$
\begin{aligned}
  \Ipdp{M_\varphi f}{M_\varphi f}_{H^2(\cplus;\Yscr)} &= \frac{1}{2\pi} \int_\R \Ipdp{\varphi(i\omega)f(i\omega)}{\varphi(i\omega)f(i\omega)}_\Yscr\ud\omega \\
   &= \Ipdp{f}{f}_{H^2(\cplus;\Uscr)}.
\end{aligned}
$$
From the isometricity of $M_\varphi$ it follows that $(1-M_\varphi M_\varphi^*)^2=1-M_\varphi M_\varphi^*\geq0$, so that $(1-M_\varphi M_\varphi^*)^{1/2}=1-M_\varphi M_\varphi^*$. This is the orthogonal projection onto $\big(M_\varphi\,H^2(\cplus;\Uscr)\big)^\perp$, since $M_\varphi M_\varphi^*$ is the orthogonal projection onto $M_\varphi\,H^2(\cplus;\Uscr)$. By \eqref{operatorrange}, $\Hscr_o = \range{1-M_\varphi M_\varphi^*}=\Ker{1-M_\varphi M_\varphi^*}^\perp$ and hence $Q$ in \eqref{lifted norm} coincides with $1-M_\varphi M_\varphi^*$. Then \eqref{lifted norm} precisely says that $\Hscr_o$ is isometrically contained in $H^2(\cplus;\Yscr)$.
\end{proof}

When $\varphi$ is not inner, $M_\varphi H^2(\cplus;\Uscr,\Yscr)$ and $\Hscr_o$ are not orthogonal, but more general \emph{complements in the sense of de Branges}, cf.\ \cite{Ando90} or \cite[\S 1.5]{ADRSBook}.

The following limits will be encountered frequently in the sequel.

\begin{prop}\label{prop:xtendstozero}
Every $x$ in $H^2(\cplus;\Yscr)$ satisfies $x(\mu)\to 0$ in $\Yscr$ 
as $\re \mu\to+\infty$. More precisely,
\begin{equation}\label{eq:H2Zspeed}
  \|x(\mu)\|_\Yscr \leq \frac{\|x\|_{H^2(\cplus;\Yscr)}}{\sqrt {2\re\mu}},\quad \mu\in\cplus.
\end{equation}
It also holds that
\begin{equation}\label{eq:HcZspeed}
  \|x(\mu)\|_\Yscr \leq \frac{\|x\|_{\Hscr_o}}{\sqrt {2\re\mu}},\quad \mu\in\cplus,
\end{equation}
and in particular the only constant function in $\Hscr_o$ is the zero function. The corresponding claims hold for $H^2(\cplus;\Uscr)$ and $\Hscr_c$.
\end{prop}
\begin{proof}
We verify the assertion only for $H^2(\cplus;\Yscr)$ and $\Hscr_o$. By the Cauchy-Schwarz inequality, we have for all $x\in H^2(\cplus;\Yscr)$ that
$$
\begin{aligned}
  |\Ipdp{x(\mu)}{y}_\Yscr| &= \big|\Ipdp{x}{e(\mu)^*y}_{H^2(\cplus;\Yscr)}\big| \\
  &\leq \|x\|_{H^2(\cplus;\Yscr)} \|e(\mu)^*y\|_{H^2(\cplus;\Yscr)} \\
  &= \|x\|_{H^2(\cplus;\Yscr)} \Ipdp{e(\mu)^*y}{e(\mu)^*y}_{H^2(\cplus;\Yscr)}^{1/2} \\
  &= \|x\|_{H^2(\cplus;\Yscr)} \Ipdp{\frac{1}{\mu+\overline\mu} y}{y}_{\Yscr}^{1/2} \\
  &\leq \frac{\|x\|_{H^2(\cplus;\Yscr)} \|y\|_{\Yscr}}{\sqrt {2\re\mu}}, \quad \mu\in\cplus.
\end{aligned}
$$
From here we obtain \eqref{eq:H2Zspeed}:
$$
  \|x(\mu)\|_\Yscr = \sup_{0\neq y\in\Yscr} \frac{|\Ipdp{x(\mu)}{y}_\Yscr|}{\|y\|}\leq 
  \frac{\|x\|_{H^2(\cplus;\Yscr)}}{\sqrt {2\re\mu}}, \quad \mu\in\cplus.
$$
Now \eqref{eq:HcZspeed} follows from \eqref{eq:H2Zspeed} combined with the facts  $\Hscr_o\subset H^2(\cplus;\Yscr)$ and $\|x\|_{H^2(\cplus;\Yscr)}\leq\|x\|_{\Hscr_o}$ for all $x\in\Hscr_o$; see Theorem \ref{T:operatorrange}.
\end{proof}

In the next section we recall the necessary background on continuous-time linear system nodes.

\section{Background on system nodes}\label{sec:sysnode}

\subsection{Definition of a system node and its transfer function}\label{sec:sysnodedef}

In this section we recall the needed concepts from the theory of infinite-dimensional 
linear systems in continuous time. A comprehensive exposition of this theory can be found
e.g.\ in \cite{StafBook}. For more details on the following few paragraphs, see 
Definition 3.2.7 and Section 3.6 of \cite{StafBook}. 

The \emph{resolvent set} $\res A$ of a closed operator $A$ on the Hilbert space $\Xscr$ 
is the set of all $\mu\in\C$ such that $\mu-A$ maps $\dom A$ one-to-one onto $\Xscr$.
The generator $A$ of a $C_0$ semigroup is closed and $\dom{A}$ dense in $\Xscr$; 
see e.g.\ \cite[Theorem 1.2.7]{PazyBook}. Moreover, the resolvent set of a $C_0$ 
semigroup generator contains some complex right-half plane. For such a generator, $\dom A$ 
is a Hilbert space with the inner product
\begin{equation}\label{eq:domAip}
	\Ipdp xz_{\dom A}=\Ipdp{(\beta-A)x}{(\beta-A)z}_\Xscr,
\end{equation}
where $\beta$ is some fixed but arbitrary complex number in $\res{A}$.

Thus $\Xscr_1:=\dom{A}$ with the norm $\|x\|_1:=\|(\beta-A)x\|_\Xscr$ is a dense 
subspace of $\Xscr$. It follows immediately from \eqref{eq:domAip} that $A$ maps 
$\dom A=\Xscr_1$ with this norm continuously into $\Xscr$. Denote by $\Xscr_{-1}$ the completion of $\Xscr$ with respect 
to the norm $\|x\|_{-1}=\|(\beta-A)^{-1}x\|_\Xscr$. The operator $A$ can then 
also be considered as a continuous operator which maps the dense subspace $\Xscr_1$ 
of $\Xscr$ into $\Xscr_{-1}$, and we denote the unique continuous extension of $A$ to
an operator $\Xscr\to\Xscr_{-1}$ by $A|_\Xscr$. Note that $\res A=\res{A|_\Xscr}$ and 
that $(\beta-A|_\Xscr)^{-1}$ maps $\Xscr_{-1}$ unitarily onto $\Xscr$.

The triple $\Xscr_1\subset\Xscr\subset\Xscr_{-1}$ is called a \emph{Gelfand triple}, 
and the three spaces are also said to be \emph{rigged}. The spaces $\Xscr_{-1}$ corresponding to two different choices of
$\beta \in \res A$ can be identified with each other as topological
vector spaces, and although the norms will be different they are
equivalent to each other. The norms of $\Xscr_1$ corresponding to two different choices of $\beta\in\res A$ will also be equivalent. Hence $(\alpha-A)^{-1}$ is an isomorphism from $\Xscr$ to $\Xscr_1$ and $(\alpha-A|_\Xscr)$ is an isomorphism from $\Xscr$ to $\Xscr_{-1}$ for all $\alpha\in\res{A}$, and these operators are unitary for $\alpha=\beta$.

\begin{defn}\label{def:altsysnode}
A linear operator
$$\SysNode:\bbm{\Xscr\\\Uscr}\supset\dom {\SmallSysNode}\to \bbm{\Xscr\\\Yscr}$$
is called a \emph{system node} on the triple $(\Uscr,\Xscr,\Yscr)$ of Hilbert spaces if it has all of the following properties:
\begin{enumerate}
\item The operator $\SmallSysNode$ is closed.
\item The operator 
\begin{equation}\label{eq:Adef}
\begin{aligned}
	Ax &:= \bbm{\AB}\bbm{x\\ 0}\quad\text{defined on}\\
	\dom A &:= \set{x \in \Xscr \bigmid \bbm{x \\ 0} \in \dom{\SmallSysNode}},
\end{aligned}
\end{equation}
is the generator of a $C_0$-semigroup on $\Xscr$.
\item The operator $\bbm{\AB}$ can be extended to an operator $\bbm{A|_\Xscr&B}$ that maps 
$\sbm{\Xscr\\\Uscr}$ continuously into $\Xscr_{-1}$.
\item The domain of $\SmallSysNode$ satisfies the condition
$$
	\dom {\SmallSysNode}=\left\{\bbm{x\\u}\in\bbm{\Xscr\\\Uscr}\bigmid A|_\Xscr 
	x+Bu\in\Xscr\right\}.
$$
\end{enumerate}
When these conditions are satisfied, $\Uscr$, $\Xscr$, and $\Yscr$ are called the \emph{input space}, \emph{state space}, and \emph{output space}, respectively, of the system node.
\end{defn}

It was mentioned in the introduction that the definition of the operator-valued function 
$\mu\mapsto C(\mu-A)^{-1}B+D$ can be extended to arbitrary system nodes. 
This is often done as follows. By \cite[Lemma  4.7.3]{StafBook}, 
$\sbm{1&(\alpha-A|_\Xscr)^{-1}B\\0&1}$ maps $\sbm{\dom{A}\\\Uscr}$ one-to-one onto $\dom{\SmallSysNode}$ for every system node $\SmallSysNode$ and $\alpha\in\res A$, and this allows us to express the domain of $\SmallSysNode$ as
\begin{equation}\label{eq:SSplitResDom}
\begin{aligned}
  \dom{\SmallSysNode}&=\bbm{\dom {A}\\\zero}\dotplus \bbm{(\alpha-A|_{\Xscr})^{-1}B\\1}\Uscr\\
\end{aligned}
\end{equation}
Note in particular that $\sbm{(\alpha-A|_{\Xscr})^{-1}B\\1}$ maps $\Uscr$ into the domain of
$\SmallSysNode$.

\begin{defn}{\cite[Definition 4.7.4]{StafBook}}\label{def:transfun}
The operators $A$ and $B$ in Definition \ref{def:altsysnode} are the \emph{main operator} 
and \emph{control operator} of the system node $\SmallSysNode$, respectively. The \emph{observation operator} $C:\dom A\to\Yscr$ of $\SmallSysNode$ is the operator
\begin{equation}\label{eq:Cdef}
  Cx:=\bbm{\CD}\bbm{x\\0},\quad x\in\dom A,
\end{equation}
and the \emph{transfer function} $\whsmash\Dfrak:\res A \to \Bscr(\Uscr,\Yscr)$ 
of $\SmallSysNode$ is the operator-valued holomorphic function
\begin{equation}\label{eq:transfundef}
  \whsmash\Dfrak(\mu):=\bbm{\CD}\bbm{(\mu-A|_\Xscr)^{-1}B\\1},\quad \mu\in\res A.
\end{equation}

By a \emph{realization} of a given analytic function $\varphi$, we mean a system node $\SmallSysNode$ whose transfer function $\whsmash \Dfrak$ coincides with $\varphi$ on some right-half plane
$$
  \C_\omega^+:=\set{\mu\in\C\mid \re\mu>\omega}\subset\res A\cap\dom{\varphi},\quad\omega\in\R.
$$
\end{defn}

Regarding the last sentence of Definition \ref{def:transfun}, we consider two analytic functions $f$ and $g$ with $\dom f,\dom g\subset\C$ to be \emph{identical} if there exists some complex right-half plane $\C^+_\omega\subset\dom f\cap\dom g$, such that $f$ and $g$ coincide on $\C^+_\omega$. In this paper we can usually take $\omega=0$, so that $\C^+_\omega=\cplus$.

Since $(\alpha-A)^{-1}$ maps $\Xscr$ one-to-one onto $\dom A$, we have that the operator
$\sbm{(\alpha-A)^{-1}&(\alpha-A|_\Xscr)^{-1}B\\0&1}$ maps $\sbm{\Xscr\\\Uscr}$ one-to-one onto $\dom{\SmallSysNode}$  for every $\alpha\in\res A$, cf.\ \eqref{eq:SSplitResDom}. The system node satisfies
\begin{equation}\label{eq:SSplitRes}
\begin{aligned}
  &\SysNode \bbm{(\alpha-A)^{-1}x&(\alpha-A|_\Xscr)^{-1}Bu\\0&u} \\
  &\qquad =\bbm{A(\alpha-A)^{-1}x&\alpha (\alpha-A|_\Xscr)^{-1}Bu\\C(\alpha-A)^{-1}x&\whsmash\Dfrak(\alpha)u}
\end{aligned}
\end{equation} 
for all $\alpha\in\res A$ and $x\in\Xscr$, $u\in\Uscr$. By the closed graph theorem, 
$C(\alpha-A)^{-1}$ is bounded from $\Xscr$ into $\Yscr$, and therefore $C$ maps 
$\dom A$ boundedly into $\Yscr$. Similarly, $\whsmash\Dfrak(\alpha)$ is bounded from $\Uscr$ into $\Yscr$ for all $\alpha\in \res A$. It is part of condition 3 in Definition \ref{def:altsysnode} that $B$ maps $\Uscr$ boundedly into $\Xscr_{-1}$. 

The equation \eqref{eq:SSplitRes} can equivalently be written, still for arbitrary $\alpha\in\res A$:
\begin{equation}\label{eq:SSplitRes2}
\begin{aligned}
  \SysNode &= \bbm{A(\alpha-A)^{-1}&\alpha (\alpha-A|_\Xscr)^{-1}B\\C(\alpha-A)^{-1}
  &\whsmash\Dfrak(\alpha)} \\ 
  &\qquad \times \bbm{(\alpha-A)^{-1}&(\alpha-A|_\Xscr)^{-1}B\\0&1}^{-1}\bigg|_{\dom{\SmallSysNode}} \\
  &= \bbm{A&\alpha (\alpha-A|_\Xscr)^{-1}B\\C&\whsmash\Dfrak(\alpha)} \\
  &\qquad \times \bbm{1&-(\alpha-A|_\Xscr)^{-1}B\\0&1}\bigg|_{\dom{\SmallSysNode}},
\end{aligned}
\end{equation} 
where $\dom{\SmallSysNode}$ is given in \eqref{eq:SSplitResDom}. In particular,
\begin{equation}\label{eq:CDreconstr}
\begin{aligned}
  \bbm{\CD}\bbm{x\\u} &= C\big(x-(\alpha-A|_\Xscr)^{-1}Bu\big)+\whsmash\Dfrak(\alpha)u,\\ \bbm{x\\u} &\in \dom{\SmallSysNode},
\end{aligned}
\end{equation}
for any arbitrary $\alpha\in\res A$.

\begin{rem}\label{rem:reconstruct}
By \cite[Lem.\ 4.7.6]{StafBook}, we can reconstruct a system node $\SmallSysNode$ from its operators $A$, $B$, $C$, and $\whsmash\Dfrak(\alpha)$, for one arbitrary $\alpha\in\res{A}$, in the following way: The space $\Xscr_{-1}$ is obtained as the co-domain of $B$, and we can then extend $A:\dom{A}\to\Xscr$ continuously into $A|_\Xscr:\Xscr\to\Xscr_{-1}$. Then we define $\AB$ via:
$$
\begin{aligned}
    \dom{\bbm{\AB}} &:=\set{\bbm{x\\u}\in\bbm{\Xscr\\\Uscr}\bigmid A|_{\Xscr_{-1}}x+Bu\in\Xscr},\\
    \bbm{\AB}&:=\bbm{A|_\Xscr&B}\Big|_{\dom{\SmallSysNode}},
\end{aligned}
$$ 
and finally we define $\bbm{\CD}$ on $\dom{\bbm{\AB}}=\dom{\bbm{\CD}}$ by \eqref{eq:CDreconstr}.
\end{rem}

We will use the following variants of controllability and observability:

\begin{defn}\label{def:iomapscontrobs}
Let $\SmallSysNode$ be a system node and denote the component of $\res A$ that contains some right-half plane by $\rho_\infty(A)$.

We say that $\SmallSysNode$ is \emph{controllable} if 
$$
  \spn\set{(\mu-A|_\Xscr)^{-1}Bu\mid \mu\in\rho_\infty(A),\,u\in\Uscr}
$$ 
is dense in the state space $\Xscr$. The system node $\SmallSysNode$ is \emph{observable} if
$$
  \bigcap_{\mu\in\rho_\infty(A)} \Ker{C(\mu-A)^{-1}}=\zero.
$$
\end{defn}

As a consequence of \cite[Cor.\ 9.6.2 and 9.6.5]{StafBook}, it suffices to take the linear span or intersection only over a subset $\Omega\subset\rho_\infty(A)$ with a cluster point in $\rho_\infty(A)$ instead of over the whole set $\rho_\infty(A)$.

\begin{lem}\label{lem:contrdense}
Let $\SmallSysNode$ be a controllable system node on $(\Uscr,\Xscr,\Yscr)$ and fix $\alpha\in\res A$ arbitrarily. Assume that $\Omega\subset\rho_\infty(A)$ has a cluster point in $\rho_\infty(A)$. Then the linear span
\begin{equation}\label{eq:densespan}
  \spn\set{(\mu-A|_\Xscr)^{-1}Bu-(\alpha-A|_\Xscr)^{-1}Bu
    \mid \mu\in\Omega,\,u\in\Uscr}
\end{equation}
is a dense subspace of both $\dom{A}$ (with respect to the graph norm of $A$) and of $\Xscr$, and the linear span
\begin{equation}\label{eq:densespan2}
  \spn\set{\bbm{(\mu-A|_\Xscr)^{-1}Bu\\u} \bigmid \mu\in\Omega,\,u\in\Uscr}
\end{equation}
is a dense subspace of $\dom{\SmallSysNode}$ with respect to the graph norm of $\SmallSysNode$.
\end{lem}
\begin{proof}
Let $\Escr$ denote the linear span in \eqref{eq:densespan}. By \eqref{eq:SSplitResDom}, 
$$
  \bbm{(\mu-A|_\Xscr)^{-1}Bu-(\alpha-A|_\Xscr)^{-1}Bu\\u-u}\in\dom{\SmallSysNode}
$$
for all $\mu\in\rho_\infty(A)$ and $u\in\Uscr$, and therefore
$$
  (\mu-A|_\Xscr)^{-1}Bu-(\alpha-A|_\Xscr)^{-1}Bu\in\dom A
$$
for all $\mu\in\rho_\infty(A)$ and $u\in\Uscr$ by \eqref{eq:Adef}. Hence $\Escr\subset\dom A$.

The resolvent identity gives
$$
  \big((\mu-A|_\Xscr)^{-1}Bu-(\alpha-A|_\Xscr)^{-1}\big)Bu = 
  (\alpha-\mu)(\alpha-A)^{-1}(\mu-A|_\Xscr)^{-1}Bu,
$$
where $(\alpha-A)^{-1}$ is an isomorphism from $\Xscr$ to $\dom A$, and so $\Escr$ is dense in $\dom A$ if and only if
\begin{equation}\label{eq:firstdensespan}
  \spn\set{(\alpha-\mu)(\mu-A|_\Xscr)^{-1}Bu \mid \mu\in\Omega,\,u\in\Uscr}
\end{equation}
is dense in $\Xscr$. It is easy to see that this linear span is the same as
$$
  \spn\set{(\mu-A|_\Xscr)^{-1}Bu \mid \mu\in\Omega\setminus\set\alpha,\,u\in\Uscr},
$$
and this space is dense in $\Xscr$, since $\Omega\setminus\set\alpha$ has a cluster point in $\rho_\infty(A)$ and $\SmallSysNode$ is assumed controllable. We have proved that \eqref{eq:densespan} is dense in $\dom A$. Since $\dom A$ is dense in $\Xscr$, it now follows automatically that \eqref{eq:densespan} is dense in $\Xscr$.

According to \cite[Lemma 4.7.3(ix)]{StafBook}, the following norm is equivalent to the norm on $\dom{\SmallSysNode}$ induced by the graph of $\SmallSysNode$:
$$
  \left\|\bbm{x\\u}\right\|_\alpha:=\left\| \bbm{1&-(\alpha-A|_\Xscr)^{-1}B\\0&1}\bbm{x\\u} \right\|_{\sbm{\dom A\\\Uscr}},
$$
where $\dom A$ is equipped with the graph norm of $A$. Therefore the denseness of \eqref{eq:densespan2} follows if we can show that
\begin{equation}\label{eq:densespan3}
  \bbm{1&-(\alpha-A|_\Xscr)^{-1}B\\0&1}\spn\set{\bbm{(\mu-A|_\Xscr)^{-1}Bu\\u} \bigmid \mu\in\Omega,\,u\in\Uscr}
\end{equation}
is dense in $\sbm{\dom A\\\Uscr}$.

Fix $\sbm{x\\u}\in\sbm{\dom A\\\Uscr}$ arbitrarily. We will show that $\sbm{x\\u}$ can be approximated arbitrarily well by an element of the linear span in \eqref{eq:densespan3}, in the norm of $\sbm{\dom A\\\Uscr}$. By the above, we can approximate $x$ by an element in $\Escr$, say
$$
  \|x-x_N\|_{\dom A}<\varepsilon,\quad \text{with}\quad x_N=\sum_{k=1}^N (\mu_k-A|_\Xscr)^{-1}Bu_k-(\alpha-A|_\Xscr)^{-1}Bu_k.
$$
Setting $v_N:=u-\sum_{k=1}^N u_k$, we obtain
$$
\begin{aligned}
  \bigg\|\bbm{x\\u} &- \bbm{1&-(\alpha-A|_\Xscr)^{-1}B\\0&1}\sum_{k=1}^N\bbm{(\mu_k-A|_\Xscr)^{-1}Bu_k\\u_k} \\
  &\qquad- \bbm{1&-(\alpha-A|_\Xscr)^{-1}B\\0&1}\bbm{(\alpha-A|_\Xscr)^{-1}Bv_N\\v_N}\bigg\|_{\sbm{\dom A\\\Uscr}} \\
  &= \bigg\|\bbm{x-x_N\\0}\bigg\|_{\sbm{\dom A\\\Uscr}}<\varepsilon,
\end{aligned}
$$
and hence the linear span in \eqref{eq:densespan2} is dense in $\dom{\SmallSysNode}$.
\end{proof}

We next recall some properties of (scattering) passive systems, including some very recent developments.

\subsection{Scattering dissipative operators and passive system nodes}

The following is a recent idea from \cite[Def.\ 2.1]{StafMaxDiss}; see also \cite{StWe12I,StWe12II}:

\begin{defn}\label{def:scattdiss}
An operator $\SmallSysNode:\sbm{\Xscr\\\Uscr}\supset\dom{\SmallSysNode}\to\sbm{\Xscr\\\Yscr}$ is called \emph{scattering dissipative} if it
satisfies for all $\sbm{x\\u}\in\dom{\SmallSysNode}$:
\begin{equation}\label{eq:scattdissop}
  \Ipdp{z}{x}_\Xscr + \Ipdp{x}{z}_\Xscr \leq 
  \Ipdp{u}{u}_\Uscr - \Ipdp{y}{y}_\Yscr,\quad \bbm{z\\y}=\SysNode\bbm{x\\u}.
\end{equation}
If such an operator $\SmallSysNode$ has no
proper extension which still satisfies \eqref{eq:scattdissop}, then $\SmallSysNode$ 
is said to be \emph{maximal scattering dissipative}. If \eqref{eq:scattdissop} holds 
with equality then $\SmallSysNode$ is called \emph{scattering isometric}.
\end{defn}

Note that $\SmallSysNode$ is scattering isometric if and only if for all 
$\sbm{x_1\\u_1},\sbm{x_2\\u_2}\in\dom{\SmallSysNode}$:
\begin{equation}\label{eq:scattdissoppolar}
\begin{aligned}
  \Ipdp{z_1}{x_2}_\Xscr + \Ipdp{x_1}{z_2}_\Xscr &= \Ipdp{u_1}{u_2}_\Uscr - \Ipdp{y_1}{y_2}_\Yscr,\\ 
  \bbm{z_k\\y_k} &= \SysNode\bbm{x_k\\u_k},
\end{aligned}
\end{equation}
as can be seen by polarizing \eqref{eq:scattdissop}, i.e., by considering 
$\sbm{x\\u}=\sbm{x_1\\u_1}+\lambda\sbm{x_2\\u_2}$ and letting $\lambda$ vary over $\C$. 

The following definition differs from the standard definition of a passive system 
node, but combining the fact that $\res A$ contains some right-half plane with \cite[Theorem 11.1.5]{StafBook}, see in particular assertion (iii), we obtain that the two definitions are equivalent:

\begin{defn}\label{def:scattpass}
A system node is said to be \emph{passive} if it is a scattering dissipative operator. 
The system node is \emph{energy preserving} if it is scattering isometric.
\end{defn}

The type of passivity in Definition \ref{def:scattpass} is commonly called \emph{scattering passivity}, where the word ``scattering'' refers to the fact that we use the expression $\|u(t)\|^2-\|y(t)\|^2$ 
to measure the power absorbed by the system from its surroundings at time $t\geq0$. 
See the introduction to \cite{StafMTNS02} for more details on this.

\begin{lem}\label{lem:dissmaxpass}
Let $\SmallSysNode$ be a scattering dissipative operator mapping its domain $\dom{\SmallSysNode}\subset\sbm{\Xscr\\\Uscr}$ into $\sbm{\Xscr\\\Yscr}$. Then $\SmallSysNode$ is a system node if and only if it is closed and $\sbm{\sbm{1&0}-\sbm{\AB}\\\sbm{0&\sqrt2}}$ maps $\dom{\SmallSysNode}$ onto a dense subspace of $\sbm{\Xscr\\\Uscr}$. When this is the case, $\SmallSysNode$ is passive.
\end{lem}
\begin{proof}
We begin with the \emph{if} direction. Assume therefore that $\SmallSysNode$ is a closed scattering dissipative operator and that $\sbm{\sbm{1&0}-\sbm{\AB}\\\sbm{0&\sqrt2}}\dom{\SmallSysNode}$ is dense in $\sbm{\Xscr\\\Uscr}$. Then the so-called internal Cayley transform
$$
  \mathbf T:=\bbm{-1&0\\0&0}+\left(\bbm{\sqrt 2&0\\0&0}+\bbm{0\\\bbm{\CD}}\right)E^{-1},
$$
defined on $\dom {\mathbf T}:=\range{E}$, where
$$
  E=\bbm{1/\sqrt2&0\\0&1}-\bbm{\bbm{\AB}/\sqrt2\\0},\quad \dom E=\dom{\SmallSysNode},
$$
is contractive (on its domain) by Lemma 2.2, Theorem 2.3(i), and the text in between, in \cite{StafMaxDiss}. Moreover, $\dom{\mathbf T}=\sbm{\sbm{1&0}-\sbm{\AB}\\\sbm{0&\sqrt2}}\dom{\SmallSysNode}$, dense in $\sbm{\Xscr\\\Uscr}$ by assumption. By \cite[Thm 2.3]{StafMaxDiss}(iv), it follows from the closedness of $\SmallSysNode$ that $\dom{\mathbf T}$ is closed, and hence $\dom{\mathbf T}=\sbm{\Xscr\\\Uscr}$. This in turn implies that $\mathbf T$ has no proper extensions to a contraction on $\sbm{\Xscr\\\Uscr}$, and therefore $\SmallSysNode$ has no scattering dissipative extension by \cite[Thm 2.3(iii)]{StafMaxDiss}. Hence, $\SmallSysNode$ is maximal scattering dissipative. Theorem 2.5 of \cite{StafMaxDiss} now gives that $\SmallSysNode$ is a passive system node.

Conversely, for the \emph{only-if} direction, assume that $\SmallSysNode$ is a scattering-dissipative system node, i.e., a passive system node according to Definition \ref{def:scattpass}. Then \cite[Thm 2.5]{StafMaxDiss} gives that $\SmallSysNode$ is closed and maximal scattering dissipative, and now \cite[Thm 2.4]{StafMaxDiss} finally yields that $\dom{\mathbf T}=\sbm{\Xscr\\\Uscr}$.
\end{proof}

\begin{lem}\label{lem:cplusres}
For a passive system node with state space $\Xscr$ and main operator $A$, we have 
$\cplus\subset\res A=\res{A|_{\Xscr}}$. 
\end{lem}

This lemma follows from \cite[Theorem 11.1.5(viii)]{StafBook} and the rigging procedure 
described at the beginning of Section \ref{sec:sysnode}. Hence, when discussing controllability and observability of passive systems, we always take $\rho_\infty(A)=\cplus$.

\subsection{Dual system nodes}  \label{sec:dualnode}

If $A$ generates a $C_0$-semigroup $\Afrak$ on the Hilbert space $\Xscr$, 
then $A^*$ generates the $C_0$-semigroup $t\mapsto (\Afrak^t)^*$, according to
\cite[Theorem 3.5.6]{StafBook}. Clearly $\res{A^*}=\set{\overline\mu\in\C\mid \mu\in\res A}$, and we denote the Gelfand triple corresponding to $A^*$ and $\overline\beta\in\res{A^*}$ by $\Xscr_1^d\subset\Xscr\subset\Xscr_{-1}^d$, where $\beta\in\res A$ is used in the rigging $\Xscr_1\subset\Xscr\subset\Xscr_{-1}$. In particular, $\Xscr_1^d=\dom{A^*}$.

This makes it possible to identify the dual of $\Xscr_1=\dom{A}$ with $\Xscr_{-1}^d$ using $\Xscr$ as pivot space:
$$
  \langle x,z\rangle_{\langle\Xscr_1,\Xscr_{-1}^d\rangle} := 
  \Ipdp{x}{z}_\Xscr,\quad x\in\dom A,\, z\in\Xscr.
$$
Similarly, the dual of $\dom{A^*}$ is identified with $\Xscr_{-1}$
using $\Xscr$ as pivot space.

\begin{prop}\label{prop:dualsysnode}
Every system node $\SmallSysNode$ on the triple $(\Uscr,\Xscr,\Yscr)$ of Hilbert-spaces has the following properties:
\begin{enumerate}
\item The adjoint $\SmallSysNode^*$ is a system node on $(\Yscr,\Xscr,\Uscr)$. The main operator of
$\SmallSysNode^*$ is $A^d=A^*$, the control operator is 
$B^d=C^*\in\Bscr(\Yscr,\Xscr_{-1}^d)$, the observation operator is 
$C^d=B^*\in\Bscr(\Xscr_1^d,\Uscr)$, and the transfer function satisfies 
$\whsmash\Dfrak^d(\lambda)=\whsmash\Dfrak(\overline\lambda)^*$ for all 
$\lambda\in\res{A^*}$, where $\whsmash\Dfrak$ is the transfer function of $\SmallSysNode$.

\item The system node $\SmallSysNode$ 
is passive if and only if $\SmallSysNode^*$ is passive.

\item The system node $\SmallSysNode$ is controllable if and only if 
$\SmallSysNode^*$ is observable and vice versa. 
\end{enumerate}
\end{prop}

For a proof of the first statement see \cite[Lemma 6.2.14]{StafBook}. The second statement follows from \cite[Lemma 11.1.4]{StafBook}; recall that passivity implies well-posedness. The third claim follows immediately on combining the first statement with Definition \ref{def:iomapscontrobs}.

\begin{defn}\label{def:dualsysnode}
The (possibly unbounded) adjoint $\sbm{A^d\&B^d\\C^d\&D^d}:=\SmallSysNode^*$ of a system node 
$\SmallSysNode$ is called the \emph{causal dual system node}, or shortly just the \emph{dual}, 
of $\SmallSysNode$.

We say that a system node is \emph{co-energy preserving} if its dual system node is 
energy preserving. A system node that is both energy preserving and co-energy preserving is called \emph{conservative}.
\end{defn}

Clearly a system node $\SmallSysNode$ is conservative if and only if the dual system node $\SmallSysNode^*$ is conservative. Energy preservation is also clearly a necessary condition for conservativity, and the following important result provides a converse:

\begin{thm}\label{thm:msw}
For every energy-preserving system node $\SmallSysNode$, the following hold:
\begin{enumerate}
\item The operator $\sbm{\sbm{1&0}\\\CD}$ maps $\dom{\SmallSysNode}$ into $\dom{\SmallSysNode^*}$ and
\begin{equation}\label{eq:enprescontdual}
  \SysNode^*\bbm{\bbm{1&0}\\\CD}=\bbm{-\AB\\\bbm{0&1}}\quad \text{on}\quad \dom{\SmallSysNode}.
\end{equation}

\item The following conditions are equivalent:
\begin{enumerate}
\item The system node $\SmallSysNode$ is conservative.
\item The operator $\sbm{\sbm{1&0}\\\CD}$ maps $\dom{\SmallSysNode}$ \emph{onto} $\dom{\SmallSysNode^*}$.
\item The range of $\SmallSysNode_c+\sbm{\overline\alpha&0\\0&0}$ is dense in $\sbm{\Hscr_c\\\Yscr}$ for some, or equivalently for all, $\alpha\in\cplus$.
\end{enumerate}
\end{enumerate}
\end{thm}

This follows by taking $R=1_U$, $P=1_\Xscr$, and $J=1_\Yscr$ in \cite[Thms 3.2 and 4.2]{MaStWe06}; see also \cite{Kuru10} for coordinate-free formulations of parts of this section. We now finally arrive at the main part of the article: a study of the continuous-time analogue of the controllable energy-preserving model in Theorem \ref{T:iso-real}.

\section{The controllable energy-preserving functional model}\label{sec:enpres}

In this section we present the controllable energy-preserving model 
realization, which uses $\Hscr_c$ as state space. Later, in Section \ref{sec:coenpres}, we show how the results for the observable co-energy-preserving 
functional-model system node can be reduced to the results of this section.

\subsection{Definition and immediate properties}\label{sec:contrintrod}

Let $\varphi\in\Sscr(\cplus;\Uscr,\Yscr)$ where $\Uscr$ and $\Yscr$ are separable Hilbert 
spaces. As before, let $\Hscr_c$ denote the Hilbert space whose reproducing kernel is 
\begin{equation}\label{eq:isomkernel}
K_{c}(\mu, \lambda) = \frac{ 1 - \varphi(\overline{\mu})^{*} 
 \varphi(\overline{\lambda})}{\mu + \overline{\lambda}}
\end{equation}
and let $e_c(\cdot)$ be the point-evaluation mapping on $\Hscr_c$, so that $e_c(\lambda)^*u=K_c(\cdot,\lambda)u$ for all $\lambda\in\cplus$ and $u\in\Uscr$. Introduce the mapping
\begin{equation}\label{eq:SodDefprel}
  \SysNode_c:\bbm{e_c(\overline\lambda)^*u\\u}\mapsto 
  \bbm{\,\lambda e_c(\overline\lambda)^*u\\\varphi(\lambda)u},
  \quad u\in\Uscr,\,\lambda\in\cplus.
\end{equation}

In the following lemma we show that $\SmallSysNode_c$ in \eqref{eq:SodDefprel} can be extended to a closable linear operator
\begin{equation}\label{eq:SodDef}
\begin{aligned}
  \SysNode_c&:\bbm{\Hscr_c\\\Uscr}\supset \Dscr_0\to\bbm{\Hscr_c\\\Yscr}, \quad\text{where}\\
  \cD_0 &: = \spn\set{\bbm{e_c(\overline\lambda)^*u\\u}\bigmid \lambda\in\cplus,\,u\in\Uscr}.
\end{aligned}
\end{equation}

\begin{lem}\label{lem:contrclosable}
The formula \eqref{eq:SodDefprel} extends via linearity and limit-closure to define a scattering-isometric closed linear operator $\SmallSysNode_c$.
\end{lem}

\begin{proof}
By \eqref{eq:isomkernel} and the equality $K_c(\overline\lambda_2,\overline\lambda_1)=e_c(\overline\lambda_2)e_c(\overline\lambda_1)^*$, 
we have for all $\lambda_k\in\cplus$ and $u_k\in\Uscr$, $k=1,2$, that
\begin{equation}\label{eq:isomkernelisom}
\begin{aligned}
  \Ipdp{u_1}{u_2}_\Uscr - \Ipdp{\varphi(\lambda_1)u_1}{\varphi(\lambda_2)u_2}_\Yscr &= 
  (\overline\lambda_2+\lambda_1)\Ipdp{e_c(\overline\lambda_1)^*u_1}{e_c(\overline\lambda_2)^*u_2}_{\Hscr_c} =\\
  \Ipdp{\lambda_1 e_c(\overline\lambda_1)^*u_1}{e_c(\overline\lambda_2)^*u_2}_{\Hscr_c} &+ 
    \Ipdp{e_c(\overline\lambda_1)^*u_1}{\lambda_2 e_c(\overline\lambda_2)^*u_2}_{\Hscr_c}.
\end{aligned}
\end{equation}
If we for $k = 1,2$ set
$$
\begin{aligned}
  \bbm{ x_k \\ u_k}  &:=  \bbm{ e_c(\overline\lambda_k)^*u_k \\ u_k}\quad \text{and}\quad
  \bbm{ z_k \\ y_k } := \bbm{A \& B \\ C \& D }_c \bbm{x_k \\ u_k } = 
    \bbm{ \lambda_k e_c(\overline\lambda_k)^* u_k \\ \varphi(\lambda_k) u_k },
\end{aligned}
$$
then \eqref{eq:isomkernelisom} can be expressed as 
\begin{equation}\label{eq:isomisom}
\begin{aligned}
  \Ipdp{u_1}{u_2}_\Uscr - \Ipdp{y_1}{y_2}_\Yscr &= \Ipdp{z_1}{x_2}_{\Hscr_c} +
    \Ipdp{x_1}{z_2}_{\Hscr_c},\\ 
  \bbm{z_k\\y_k}&=\SysNode_c\bbm{x_k\\u_k},
\end{aligned}
\end{equation}
for all $\sbm{x_k \\ u_k}=\sbm{ e_c(\overline\lambda_k)^*u_k \\ u_k}$, $k=1,2$.
If we formally extend the definition \eqref{eq:SodDefprel} of $\sbm{ A \& B \\ C \& D}_c$ to all of $\cD_0$ by taking linear combinations
(where at this stage $\sbm{A \& B \\ C \& D}_c$ may a priori be ill-defined, so that $\sbm{A \& B \\ C \& D}_c \sbm{ x \\ u}$ depends on the choice of linear combination 
$\sbm{x \\ u} =  \sum_{k=1}^N \sbm{e_c(\overline\lambda_k)^* u_k \\ u_k}$ chosen to represent $\sbm{x \\ u}$), then the identity
\eqref{eq:isomisom} continues to hold for all $\sbm{ x_1 \\ u_1}, \sbm{x_2 \\ u_2}$ in the span $\cD_0$.

We now show that this implies that $\SmallSysNode_c$ in \eqref{eq:SodDef} is well-defined and 
closable. Suppose that $x_n$, $u_n$, $z_n$, and $y_n$ are sequences such that
\begin{equation}   \label{hypothesis}
\begin{aligned}
  \bbm{z_n\\y_n}=\bbm{ A \& B \\ C \& D}_c \bbm{x_n\\u_n}&\to \bbm{z\\y} ~\text{in}~\bbm{\Hscr_c\\\Yscr} \quad \text{and}\\
    \bbm{x_n\\u_n}&\to\bbm{0\\0}  ~\text{in}~\bbm{\Hscr_c\\\Yscr}.  
\end{aligned}
\end{equation}
To establish that $\sbm{ A \& B \\ C \& D}_c$ is closable, we need to show that 
$\sbm{z\\ y}=\sbm{0\\0}$.  The special case where $\sbm{x_n \\ u_n} = \sbm{ 0 \\ 0}$ for all $n$ is exactly what is needed to see that $\sbm{A \& B \\ C \& D}_c$ is well-defined; in this way well-definedness and closability are simultaneously handled in a single argument.

Using \eqref{eq:isomisom} and the continuity of the inner product, the hypothesis  \eqref{hypothesis} implies that
$$
  -\|y\|_\Yscr^2=\Ipdp{0}{0}_\Uscr - \Ipdp{y}{y}_\Yscr = \Ipdp{z}{0}_{\Hscr_c} +
  \Ipdp{0}{z}_{\Hscr_c}=0,
$$
and so $y=0$. Applying \eqref{eq:isomisom} again, we now obtain that for all
$\sbm{x_2\\u_2}\in\dom{\SmallSysNode_c}$:
$$
  0=\Ipdp{0}{y_2}_\Uscr - \Ipdp{0}{u_2}_\Yscr = \Ipdp{z}{x_2}_{\Hscr_c} + \Ipdp{0}{z_2}_{\Hscr_c},
    \quad \sbm{z_2\\y_2}=\SmallSysNode_c\sbm{x_2\\u_2},
$$
so that $z\perp x_2$ for all $\sbm{x_2\\y_2}\in\dom{\SmallSysNode_c}$.
In particular, for every $\lambda\in\cplus$ and $u\in\Uscr$ we have that $\sbm{x_2\\u_2}:=\sbm{e_c(\overline\lambda)^*u\\u}\in\dom{\SmallSysNode_c}$ and 
$$
  0=\Ipdp{z}{e_c(\overline\lambda)^*u}_{\Hscr_c} = \Ipdp{z(\overline\lambda)}{u}_\Uscr,\quad \lambda\in\cplus,\,u\in\Uscr,
$$
and therefore $z(\overline\lambda) = 0$ for all $\lambda\in\cplus$. We conclude that both $z$ and $y$ are zero as needed to complete the proof.
\end{proof}

From now on, we let $\SmallSysNode_c$ denote the \emph{closure} of the linear 
operator determined by \eqref{eq:SodDef}.

\begin{thm}\label{thm:firstmodel}
The operator $\SmallSysNode_c$ is an energy-preserving system node with input 
space $\Uscr$, state space $\Hscr_c$, and output space $\Yscr$. Denoting the 
main and control operators of $\SmallSysNode_c$ by $A_c$ and $B_c$, respectively, we obtain that
\begin{equation}\label{eq:isomeval}
	(\alpha-A_c|_{\Hscr_c})^{-1}B_c=e_c(\overline\alpha)^*,\quad \alpha \in\cplus.
\end{equation}
In addition, $\SmallSysNode_c$ is controllable:
$$
	\cspn\set{(\alpha-A_c|_{\Hscr_c})^{-1}B_cu\mid u\in\Uscr,\,\alpha \in\cplus}=\Hscr_c,
$$
and $\SmallSysNode_c$ realizes $\varphi$: 
$$
  \bbm{C_c\&D_c}\bbm{(\alpha-A_c|_{\Hscr_c})^{-1}B_c\\1}=\varphi(\alpha), \quad\alpha \in\cplus.
$$
\end{thm}
\begin{proof}
We use Lemma \ref{lem:dissmaxpass} to prove that $\SmallSysNode_c$ is a passive 
system node. By Lemma \ref{lem:contrclosable}, $\SmallSysNode_c$ is closed, and it suffices to show that the following subspace of the range of $\sbm{\sbm{1&0}- \sbm{A_c \& B_c} \\\sbm{0&1}}$ is 
dense in $\sbm{\Hscr_c\\\Uscr}$:
$$
\begin{aligned}
  \Rscr :&= \spn\set{\bbm{\bbm{1&0}-\bbm{A_c \& B_c}\\\bbm{0&\sqrt2}}\bbm{e_c(\overline\lambda)^*u\\u}\biggmid
  \lambda\in\cplus,\,u\in\Uscr} \\
  &= \spn\set{\bbm{(1-\lambda)e_c(\overline\lambda)^*u\\\sqrt2u}\biggmid \lambda\in\cplus,\,u\in\Uscr}.
\end{aligned}
$$
This space is indeed dense, since
$$
\begin{aligned}
	\bbm{x_2\\u_2}\in \bbm{\Hscr_c\\\Uscr}\ominus \Rscr\quad &\Longleftrightarrow\quad
	\forall\lambda\in\cplus,\,u\in\Uscr:\\ &\qquad\qquad \Ipdp{x_2}{(1-\lambda)e_c(\overline\lambda)^*u}_{\Hscr_c}+\Ipdp{u_2}{\sqrt2u}_\Uscr=0\\
	& \Longleftrightarrow\quad \forall\lambda\in\cplus:~(\overline\lambda-1)x_2(\overline\lambda)=\sqrt2u_2.
\end{aligned}
$$
Choosing $\lambda=1$ yields that $u_2=0$ and hence $x_2(\lambda)=0$ for all 
$\lambda\in\cplus\setminus\set1$. Since $x_2$ is holomorphic and thus continuous, 
also $x_2(1)=0$ as well, and hence $x_2$ is the zero function in $\cH_o$. We have established that $\SmallSysNode_c$ is a passive system node, which is moreover energy preserving due to Definition \ref{def:scattpass} and Lemma \ref{lem:contrclosable}.

By Lemma \ref{lem:cplusres}, we have $\cplus\subset\res{A|_{\Hscr_c}}$. Then \eqref{eq:isomeval} follows from \eqref{eq:SodDef} and Definition \ref{def:altsysnode}.3, since for every $\alpha \in\cplus$:
\begin{equation}\label{eq:isomBexp}
\begin{aligned}
  \bbm{A_c \& B_c} \bbm{e_c(\overline\alpha)^*u \\ u}&= \alpha e_c(\overline\alpha)^*u = 
  A_c|_{\Hscr_c} e_c(\overline\alpha)^*u+B_cu \\
  &\Longleftrightarrow\quad (\alpha-A_c|_{\Hscr_c} ) e_c(\overline\alpha)^*u=B_cu.
\end{aligned}
\end{equation}
In particular, $\SmallSysNode_c$ is controllable because 
$$\cspn\set{e_c(\overline\alpha)^*u \mid u\in\Uscr,\,\alpha\in\cplus}=\Hscr_c$$ 
by Remark \ref{rem:kernelsdense}.  Finally, 
the transfer function of $\SmallSysNode_c$ evaluated at $\alpha\in\cplus$ is
\begin{equation}\label{eq:isomCD}
  \bbm{C_c \& D_c}\bbm{(\alpha-A_c|_{\Hscr_c})^{-1}B_cu\\u} =
  \bbm{C_c \& D_c}\bbm{e_c(\overline\alpha)^*u\\u}=\varphi(\alpha)u,
\end{equation}
for all $u\in\Uscr$.
\end{proof}

The domain of the main operator $A_c$ of $\SmallSysNode_c$ is defined abstractly in \eqref{eq:Adef}, but we do not know how to characterize $\dom{A_c}$ explicitly. The observation operator $C_c$ is defined in \eqref{eq:Cdef}, but we have no explicit formula for the action of $C_c$ on generic elements of $\dom{A_c}$ either. These two shortcomings will cause us significant difficulties later.

\subsection{Uniqueness up to unitary similarity}

We now prove that every controllable energy-preserving realization of $\varphi$ 
is unitarily similar to $\SmallSysNode_c$;  this justifies the terminology {\em canonical functional-model} system node for $\SmallSysNode_c$.

\begin{thm}\label{thm:firstsimilar}
Let $\varphi\in\Sscr(\cplus;\Uscr,\Yscr)$ and let $\SmallSysNode$ be a controllable and energy preserving realization of 
$\varphi$ with state space $\Xscr$. Then the mapping $\Delta:\Hscr_c\to\Xscr$ defined by
\begin{equation}\label{eq:firstevalprime}
	\Delta e_c(\overline \lambda)^*u := (\lambda-A|_{\Xscr})^{-1}Bu,\quad 
	\lambda\in\cplus,\,u\in\Uscr,
\end{equation}
extends by linearity and limit-closure to a unitary operator $\Hscr_c\to\Xscr$.
Moreover, $\Delta$ \emph{intertwines} $\SmallSysNode$ with $\SmallSysNode_c$:
\begin{equation}\label{eq:contrintertw}
\begin{aligned}
  \dom\SmallSysNode &= \bbm{\Delta&0\\0&1_\Uscr}\dom{\SmallSysNode_c}\quad\text{and}\\
  \SysNode\bbm{\Delta&0\\0&1_\Uscr} &= \bbm{\Delta&0\\0&1_\Yscr}\SysNode_c,
\end{aligned}
\end{equation}
so that $\SmallSysNode$ and $\SmallSysNode_c$ are unitarily similar.
\end{thm}
\begin{proof}
The key to the proof is the following consequence of \eqref{eq:SSplitResDom} and
\eqref{eq:SSplitRes}: for all $\lambda\in\res {A}$ we have that $\sbm{(\lambda-A|_\Xscr)^{-1}Bu\\u}\in\dom{\SmallSysNode}$ and
\begin{equation}\label{eq:Sisom}
  \SysNode\bbm{(\lambda-A|_\Xscr)^{-1}Bu\\u} = 
    \bbm{\lambda(\lambda-A|_\Xscr)^{-1}Bu\\\varphi(\lambda)u}.
\end{equation}

Since $\SmallSysNode$ is assumed to be energy preserving, it follows from 
Lemma \ref{lem:cplusres} that $\cplus\subset\res{A}$ and that \eqref{eq:scattdissoppolar} 
is satisfied. According to \eqref{eq:firstevalprime}, we have for all $\lambda,\mu\in\cplus$ and all $u,v\in\Uscr$ that
\begin{equation*}\label{eq:Deltaccontr1}
\begin{aligned}
  (\mu+\lambda)\Ipdp{\Delta e_c(\overline\lambda)^*u}{\Delta e_c(\mu)^*v}_{\Xscr} \qquad &= \\ 
  (\mu+\lambda) \Ipdp{(\lambda-A|_\Xscr)^{-1}Bu}
  {(\overline\mu-A|_\Xscr)^{-1}Bv}_{\Xscr} \qquad &= \\
  \Ipdp{\lambda(\lambda-A|_\Xscr)^{-1}Bu}
  {(\overline\mu-A|_\Xscr)^{-1}Bv}_{\Xscr} \qquad &\\
    \qquad + \Ipdp{(\lambda-A|_\Xscr)^{-1}Bu} 
    {\overline\mu(\overline\mu-A|_\Xscr)^{-1}Bv}_{\Xscr}&.
\end{aligned}
\end{equation*}
This is by \eqref{eq:scattdissoppolar} and \eqref{eq:Sisom} equal to
\begin{equation*}\label{eq:Deltaccontr2}
  \Ipdp{u}{v}_\Uscr - \Ipdp{\varphi(\lambda)u}{\varphi(\overline\mu)v}_\Yscr
  = (\mu+\lambda)\Ipdp{e_c(\overline\lambda)^*u}{e_c(\mu)^*v}_{\Hscr_c},
\end{equation*}
where we used \eqref{eq:isomkernelisom} in the last step. We can conclude that
\begin{equation}\label{eq:DeltaUnitPrel}
 \Ipdp{\Delta e_c(\overline\lambda)^*u}{\Delta e_c(\mu)^*v}_{\Xscr} = 
    \Ipdp{ e_c(\overline\lambda)^*u}{ e_c(\mu)^*v}_{\Hscr_c}, \quad \overline\lambda,\mu\in\cplus,\,u,v\in\Uscr.
\end{equation}

Taking linear combinations, we obtain from \eqref{eq:DeltaUnitPrel} that for all $\overline\lambda_k\in\cplus$ and $u_k\in\Uscr$:
\begin{equation}\label{eq:DeltacIsom}
  \left\|\Delta \sum_{k=1}^n e_c(\overline{\lambda_k})^*u_k\right\|_{\Xscr}^2 =
  \left\|\sum_{k=1}^n e_c(\overline{\lambda_k})^*u_k\right\|_{\Hscr_c}^2.
\end{equation}

Denote $\Escr_0 := \spn\set{e_c(\overline\lambda)^*u\mid \overline\lambda \in \cplus,\, u \in \Uscr}$, equipped with the norm of $\Hscr_c$.  Then each $x \in \Escr_0$ can be written as a sum $x
= \sum_{k=1}^n e_c(\overline{\lambda_k})^*u_k$, and \eqref{eq:DeltacIsom} shows that the
value of $\Delta \sum_{k=1}^n e_c(\overline{\lambda_k})^*u_k$ is independent of
the particular linear combination $\sum_{k=1}^n e_c(\overline{\lambda_k})^*u_k$ that is
used to represent $x$.  Thus, $\Delta$, which was originally defined
only for kernel functions $e(\overline\lambda)^*u$ with $\overline\lambda \in
\cplus$ and $u \in \Uscr$, has a unique extension to a linear operator
$\Escr_0 \to \Xscr$, which we still denote by $\Delta$. Due to \eqref{eq:DeltacIsom}, this operator is isometric,
and by \eqref{eq:firstevalprime} the image of $\Escr_0$ under this operator is dense in $\Xscr$, since
$\SmallSysNode$ is assumed to be controllable.  As $\Escr_0$ is dense
in $\Hscr_c$, we may further extend $\Delta$ to a unitary operator
$\Hscr_c \to \Xscr$, which we still denote by $\Delta$.

Now we prove that $\Delta$ intertwines $\SmallSysNode$ with $\SmallSysNode_c$. It follows from \eqref{eq:SSplitResDom} that $\sbm{\Delta&0\\0&1}$ maps $\Dscr_0$, defined in \eqref{eq:SodDef}, into $\dom{\SmallSysNode}$. By \eqref{eq:firstevalprime}, \eqref{eq:Sisom}, and \eqref{eq:SodDefprel}, the following equality holds 
for all $\lambda\in\cplus$ and $u\in\Uscr$:
\begin{equation}\label{eq:Sisom2}
\begin{aligned}
  \SysNode\bbm{\Delta e_c(\overline\lambda)^*u\\u} &= \SysNode\bbm{(\lambda-A|_\Xscr)^{-1}Bu\\u} 
  \\  &= \bbm{\lambda(\lambda-A|_\Xscr)^{-1}Bu\\\varphi(\lambda)u} \\
  &= \bbm{ \lambda \Delta e_c(\overline\lambda)^* u\\\varphi(\lambda)u} = 
  \bbm{\Delta \bbm{A_c \& B_c} \\ C_c \& D_c} \bbm{e_c(\overline\lambda)^* u\\u},
\end{aligned}
\end{equation}
which shows that $\SmallSysNode\sbm{\Delta&0\\0&1_\Uscr}$ and 
$\sbm{\Delta&0\\0&1_\Yscr}\SmallSysNode_c$ coincide on $\Dscr_0$. Furthermore, $\Dscr_0$ 
is dense in $\dom{\SmallSysNode_c}$, equipped with the graph norm, by Lemma \ref{lem:contrdense}.

We next show how it follows from the above that
\begin{equation}\label{eq:enpresintertproof}
\begin{aligned}
  \bbm{\Delta x\\u} &\in \dom{\SmallSysNode}~\text{for all}~\bbm{x\\u}\in\dom{\SmallSysNode_c},\quad\text{and}\\
  &\SysNode\bbm{\Delta x\\u} = \bbm{\Delta \bbm{A_c \& B_c} \\ C_c \& D_c }\bbm{x\\u},\quad \bbm{x\\u}\in\dom{\SmallSysNode_c}.
\end{aligned}
\end{equation}
For every $\sbm{x\\u}\in\dom{\SmallSysNode_c}$ there by the definition 
of $\SmallSysNode_c$ exists a sequence $\sbm{x_n\\u_n}\in \Dscr_0$, such that 
$\sbm{x_n\\u_n}\to\sbm{x\\u}$ in $\sbm{\Hscr_c\\\Uscr}$ and $\SmallSysNode_c\sbm{x_n\\u_n}\to\SmallSysNode_c\sbm{x\\u}$ in $\sbm{\Hscr_c\\\Yscr}$. 
By the continuity of $\Delta$ and the fact that $\SmallSysNode\sbm{\Delta&0\\0&1_\Uscr}$ and 
$\sbm{\Delta&0\\0&1_\Yscr}\SmallSysNode_c$ coincide on $\Dscr_0$, this implies that 
$$
  \bbm{\Delta \bbm{A_c \& B_c} \\ C_c \& D_c }\bbm{x_n\\u_n} =
  \SysNode\bbm{\Delta x_n\\u_n} \to \bbm{\Delta \bbm{A_c \& B_c} \\ C_c \& D_c }\bbm{x\\u}.
$$
Using the continuity of $\Delta$ again, we obtain that $\sbm{\Delta x_n\\u_n}\in\dom 
{\SmallSysNode}$ converges to $\sbm{\Delta x\\u}$ in $\sbm{\Xscr\\\Uscr}$, and so 
by the closedness of $\SmallSysNode$, we have \eqref{eq:enpresintertproof}.

It remains to prove that $\sbm{\Delta&0\\0&1}$ maps $\dom{\SmallSysNode_c}$ \emph{onto} $\dom{\SmallSysNode}$. As a consequence of \eqref{eq:firstevalprime} and Lemma \ref{lem:contrdense}, $\sbm{\Delta&0\\0&1}\Dscr_0$ is dense in $\dom{\SmallSysNode}$ with the graph norm. Hence, for every $\sbm{w\\u}\in\dom{\SmallSysNode}$, we can find a sequence $\sbm{w_n\\u_n}\in\sbm{\Delta&0\\0&1}\Dscr_0$ that converges to $\sbm{w\\u}$ in the graph norm of $\SmallSysNode$. Writing $x_n:=\Delta^{-1}w_n$, we obtain from \eqref{eq:enpresintertproof} and the closedness of $\SmallSysNode_c$ that $\sbm{\Delta^{-1}w_n\\u_n}\to\sbm{\Delta^{-1}w\\u}$ in the graph norm of $\SmallSysNode_c$. (The details are very similar to the preceding paragraph.) Thus, for every $\sbm{w\\u}\in\dom{\SmallSysNode}$, we have $\sbm{x\\u}:=\sbm{\Delta^{-1}w\\u}\in\dom{\SmallSysNode_c}$ and $\sbm{w\\u}=\sbm{\Delta x\\u}$.
\end{proof}

We would like to obtain explicit formulas for the main, control, and observation operators of $\SmallSysNode_c$ acting on generic elements of $\Hscr_c$, and similarly for the adjoint. It turns out that this task is much easier for $\SmallSysNode_c^*$ than for $\SmallSysNode_c$, so we start with the dual. This choice is also motivated by Theorem \ref{thm:msw}.

\subsection{Explicit formulas for the system-node operators of the dual}  \label{S:contexp}

In reproducing-kernel Hilbert spaces, the existence of an explicit formula 
for the action of a given operator on kernel functions usually means that there is an equally explicit formula for the action of the adjoint on a generic functional element of the reproducing kernel Hilbert space.    This phenomenon continues to hold in the present unbounded setting, as illustrated in the following proposition. We refer the reader back to Subsection \ref{sec:dualnode} for the definition of dual system node.

The reader will observe 
that many of the formulas in this section have apparent singularities at some points of the form $0/0$.
Since the functions are holomorphic (or conjugate holomorphic), the singularities are in fact removable
and the formulas continue to hold when one applies holomorphic continuation to evaluate at such exceptional points.

By the general principles explained in Subsection \ref{sec:dualnode}, we know that $\sbm{ A \& B \\ C \& D}_c^*$ is a system node on 
$(\cY, \cX, \cU)$. We now compute this dual system node.

\begin{thm}\label{thm:contrdual}
The dual system node of $\SmallSysNode_c$ is the operator 
$$ 
\bbm{A \& B \\ C \& D}_c^* : \bbm{\Hscr_o\\\Yscr}\supset\dom{\SmallSysNode_c^*}\to\bbm{\Hscr_o\\\Uscr}
$$ 
given by
\begin{align}
  \SysNode_c^* &: \bbm{x\\y}\mapsto\bbm{z\\u}, \quad\text {where}\label{eq:isomDual}\\
   z(\mu) &:= \mu x(\mu)+\wtsmash\varphi(\mu)y-u,\quad \mu\in\cplus,\quad \text{and}\quad~ \label{defz} \\
   u &:= \lim_{\re\eta\to\infty} \eta x(\eta)+\wtsmash\varphi(\eta)y,\quad\text{with domain}\label{defu} \\
   \dom{\SysNode_c^*} &:= \set{\bbm{x\\y}\in\bbm{\Hscr_c\\\Yscr}\bigmid 
    \exists u\in\Uscr:~z\in\Hscr_c~\text{in \eqref{defz}}}. \label{defdom}
\end{align}
For every $\sbm{x\\y}\in\dom{\SmallSysNode_c^*}$, the $u\in\Uscr$ such that 
$z$ in \eqref{defz} lies in $\Hscr_c$ 
is unique and it is given by \eqref{defu}.   
\end{thm}

\begin{proof}
We combine the graph characterization
$$
  \bbm{\Hscr_c\\\Uscr\\\Hscr_c\\\Yscr}\ominus \left(\bbm{-\bbm{1&0\\0&1}\\
  \SysNode_c}\dom{\SysNode_c}\right) =
  \bbm{\SysNode_c^*\\\bbm{1&0\\0&1}}\dom{\SysNode_c^*}
$$
of the adjoint of $\SmallSysNode_c$ with the construction of $\SmallSysNode_c$ and thus obtain that $\sbm{x\\y}\in\dom{\SmallSysNode_c^*}$ and $\sbm{z\\u}=\SmallSysNode_c^*\sbm{x\\y}$ if and only if
$$
\begin{aligned}
  \bbm{z\\u\\x\\y} \in &\bbm{\Hscr_c\\\Uscr\\\Hscr_c\\\Yscr}\ominus 
    \cspn\set{\bbm{-e_c(\mu)^*v\\-v\\\overline\mu e_c(\mu)^*v\\\varphi(\overline\mu)v}, \quad v\in\Uscr,\,\mu\in\cplus}\\
  &\quad=\bigcap_{\mu\in\cplus} \Ker{\bbm{-e_c(\mu)&-1&\mu e_c(\mu)&\varphi(\overline\mu)^*}}.
\end{aligned}
$$
Thus a pair $\sbm{x\\y}\in\sbm{\Hscr_c\\\Yscr}$ lies in $\dom{\SmallSysNode_c^*}$ if and only if there exist 
$z\in\Hscr_c$ and $u\in\Uscr$ such that
$-z(\mu)-u+\mu x(\mu)+\wtsmash\varphi(\mu)y=0$ for all $\mu\in\cplus$,
i.e., $z$ is given by \eqref{defz}.
When such a $u$ exists, we have
$$
  \lim_{\re\eta\to\infty} z(\eta) = \eta x(\eta)+\varphi(\overline\eta)^*y-u=0
$$ 
by Proposition \ref{prop:xtendstozero}, and hence $u$ is given by \eqref{defu}.
\end{proof}

With the formulas in Theorem \ref{thm:contrdual} as a starting point, it is possible to compute the main operator $A_c^d = A_c^*$, control operator $B_c^d = C_c^* \in \cB(\cY, \cH^d_{c,-1})$, and observation operator 
 $C_c^d = B_c^* \in \cB(\cH^d_{c,1}, \cU)$, of $\SmallSysNode_c^*$ explicitly.

\begin{prop}   \label{P:dualAC}
The domain $\dom{A_c^*}=\cH^d_{c,1}$ of $A_c^*=A_c^d$ is given by
\begin{equation}   \label{domAc*}
  \dom{A_c^*} = \{ x \in \cH_c \mid \exists u \in \cU:~ \mu \mapsto \mu x(\mu) - u \in \cH_c \}.
\end{equation}
Moreover, when $x \in \dom{A_c^*}$, the associated vector $u$ can be recovered from $x$ using the formula $u = \lim_{\re\eta \to \infty} \eta x(\eta)$, and
\begin{equation}  \label{AC}
  (A_c^* x)(\mu) = \mu x(\mu) - \lim_{\re\eta \to \infty}  \eta x(\eta), \quad
 C_c^d x = B_c^* x = \lim_{\re\eta \to \infty}   \eta x(\eta).
\end{equation}
\end{prop}
\begin{proof}
By Definition \ref{def:dualsysnode}, \eqref{eq:Adef}, and \eqref{eq:Cdef}, a function $x\in\Hscr_c$ lies in $\dom{A_c^d}$ if and only if $\sbm{x\\0}\in\dom{\SmallSysNode_c^*}$, and in this case 
$$
  \bbm{A_c^dx\\C_c^dx}=\SysNode_c^*\bbm{x\\0}.
$$
Comparing this to Theorem \ref{thm:contrdual} with $y=0$ gives the result.
\end{proof}

To get an explicit description of the $(-1)$-scaled rigged space (also called ``extrapolation space'') $\cH^d_{c,-1}$ we first need a formula for the resolvent of $A_c^*$.

\begin{prop}   \label{P:resolAc*}
Let $A_c^d = A_c^*$ be the main operator and $B_c^d = C_c^*$  be the control operator for  the dual system node 
$\sbm{A_c^d \& B_c^d \\ C_c^d \& D_c^d } = \sbm{ A \& B  \\ C \& D}_c^*$.  Then the resolvent of $A_c^*$ is given by 
\begin{equation}   \label{Ac*res}
\left( ( \overline\alpha - A_c^*)^{-1} x \right)(\mu) = \frac{ x(\mu) - x(\overline\alpha)}{\overline\alpha - \mu}, \quad \alpha, \mu \in {\mathbb C}^+,\, x \in \cH_c.
\end{equation}
Moreover, the following formulas hold:
\begin{align}  
\left(A_c^* (\overline\alpha - A_c^*)^{-1}x\right)(\mu) & = \frac{ \mu x(\mu) - \overline\alpha x(\overline\alpha)}{ \overline\alpha - \mu}, 
  \quad \alpha, \mu \in {\mathbb C}^+, \, x \in \cH_c,  \label{Ac*res1} \\
\left( (\overline\alpha - A_c^*|_{\cH_c})^{-1} C_c^* y \right) (\mu) & = \frac{\wtsmash\varphi(\mu) - \wtsmash\varphi(\overline\alpha)}{\overline\alpha - \mu} y,
  \quad \alpha, \mu \in {\mathbb C}^+,\, y\in\Yscr, \label{Ac*res2} \\
  B_c^* (\overline\alpha - A_c^*)^{-1}&= e_c(\overline\alpha),\quad \alpha\in\cplus.\label{Ac*res3}
\end{align}
\end{prop}

\begin{proof}
For $\xi \in \dom{A_c^*}$, set  $x = (\overline\alpha - A_c^*) \xi$.  From the formulas \eqref{AC} we see that
\begin{equation*}
x(\mu) = (\overline\alpha - \mu) \xi(\mu) + B_c^* \xi.
\end{equation*}
We conclude that $B_c^* \xi = x(\overline\alpha)$ and $x(\mu) = (\overline\alpha - \mu) \xi(\mu) + x(\overline\alpha)$.  Solving for $\xi$ gives
$\xi(\mu) = \frac{x(\mu) - x(\overline\alpha)}{\overline\alpha - \mu}$ and formula \eqref{Ac*res} follows.

From $A_c^* (\overline\alpha - A_c^*)^{-1} = \overline\alpha (\overline\alpha - A_c^*)^{-1} - 1$, we get
\begin{align*}
  \left(A_c^* (\overline\alpha - A_c^*)^{-1} x \right) (\mu) & = \frac{ \overline\alpha x(\mu) - \overline\alpha x(\overline\alpha)}{\overline\alpha - \mu} - x(\mu) 
   = \frac{\mu x(\mu) - \overline\alpha x(\overline\alpha)}{\overline\alpha - \mu}
 \end{align*}
 and formula \eqref{Ac*res1} follows.
 
 By \eqref{eq:SSplitResDom}, for an arbitrary $y \in \cY$ we can set $x: =  \left(\overline\alpha - A_c^*|_{\cH_c}\right)^{-1} C_c^*y$ in order to get $\sbm{x \\ y} \in \dom{\sbm{ A \& B \\ C \& D}_c^*}$.  If we further set $\sbm{z \\ u}:= \sbm{A \& B \\ C \& D}_c^* \sbm{x \\ y }$, then we obtain from \eqref{eq:SSplitRes} that 
 $$
 z = \overline\alpha x,   \quad u = \wtsmash\varphi(\overline\alpha) y.
 $$
 (Here we use the fact that the transfer function of $\sbm{ A \& B \\ C \& D}_c^*$ is $\wtsmash\varphi$ since the transfer function of $\sbm{A \& B \\ C \& D}_c$ is $\varphi$, as observed in Proposition \ref{prop:dualsysnode}.)  On the other hand, from Theorem \ref{thm:contrdual} we know that $z(\mu) = \mu x(\mu) + \wtsmash\varphi(\mu) y - u$. Combining these, we have for all $y \in \cY$ and $\alpha\in {\mathbb C}^+$ that
$$  
  \overline\alpha x(\mu) = \mu x(\mu) + \wtsmash\varphi(\mu) y - \wtsmash\varphi(\overline\alpha) y 
    \quad \Longrightarrow\quad x(\mu) = \frac{ \wtsmash\varphi(\mu) - \wtsmash\varphi(\overline\alpha)}{ \overline\alpha - \mu} y
$$
and formula \eqref{Ac*res2} follows. Formula \eqref{Ac*res3} follows directly from \eqref{eq:isomeval}.
\end{proof}

We recall that the (-1)-scaled rigged space is defined as the completion of the space $\cH_c$ in the norm
\begin{equation}\label{eq:Hcminus1dnorm}
\| x \| = \left\| (\overline\beta - A_c^*)^{-1} x \right\|_{\cH_c}
\end{equation}
and that $(\overline\beta - A_c^*)^{-1}$ has an extension to a unitary operator from this rigged space onto $\Hscr_c$. For $x \in \cH_c$, we now have the formula \eqref{Ac*res} for the action of the resolvent $(\overline\beta - A_c^*)^{-1}$ on $x$. This suggests that we identify the $(-1)$-rigged space concretely as follows:
\begin{equation}  \label{Hdc-1con}
\cH^d_{c,-1} :=\left\{ x : {\mathbb C}^+ \to \cU \bigmid \mu \mapsto  \frac{ x(\mu) - x(\overline\beta)}{\overline\beta - \mu} \in \cH_c \right\},
\end{equation}
with norm given by
\begin{equation}  \label{Hdc-1con-norm}
\|x\|_{\cH^d_{c,-1}}  := \left\| \mu \mapsto \frac{ x(\mu) - x(\overline\beta)}{\overline\beta - \mu} \right\|_{\cH_c},
\end{equation}
as was also done in \cite{JaZw02, JaZw02b} in a closely related setting.

Once again the norm depends on the choice of $\overline\beta$ but all norms arising in this way are equivalent.
Note that constant functions $x(\mu) = u$ are in $\cH^d_{c,-1}$ with norm zero;  therefore we view the space as consisting of equivalence classes, where two representatives $x$ and $\xi$ of the same equivalence class differ by a constant:  $x(\mu) - \xi(\mu) = v$ for all $\mu \in {\mathbb C}^+$ for some $v \in \cU$.  We denote the equivalence class of $x$ in $\cH^d_{c,-1}$ by $[x]$; and if $[x] = [\xi]$ then we write $x \cong \xi$.   
 Next some properties of this space $\cH^d_{c,-1}$ are summarized:

\begin{thm}   \label{thm:hdc-1concrete}
The space $\cH^d_{c,-1}$ defined in \eqref{Hdc-1con} and \eqref{Hdc-1con-norm} is complete.
\begin{enumerate}
\item The map $\iota : x \mapsto [x]$ embeds $\cH_c$ into  $\cH^d_{c,-1}$ as a dense subspace.
A given element $[x] \in \cH^d_{c,-1}$
is of the form $\iota(z)$ for some $z \in \cH_c$ if and only if the function $\displaystyle{\mu \mapsto \frac{x(\mu) - x(\overline\alpha)}{\overline\alpha - \mu}}$, $\mu\in\cplus$, is not only in $\cH_c$ but in fact is in $\dom{A_c^*} = \cH^d_{c,1} \subset \cH_c$ for some, or equivalently for all, $\alpha\in\cplus$.  When this is the case, the equivalence class representative $z$ for $[x]$ that lies in $\cH_c$, is uniquely determined by the decay condition at infinity:  
\begin{equation}  \label{decay}
\lim_{\re \eta \to \infty}  z(\eta) = 0.
\end{equation}

\item Define an operator $A_c^*|_{\cH_c} : \cH_c \to \cH^d_{c, -1}$ by
\begin{equation}   \label{Ac*-ext}
  A_c^*|_{\cH_c} x:=[\mu \mapsto \mu x(\mu)],\quad x\in\cH_c,\, \mu\in\cplus.
\end{equation}
When $\cH_c$ is identified as a linear sub-manifold of $\cH^d_{c,-1}$, then $A_c^*|_{\cH_c}$ is the continuous extension of $A_c^* : \dom{A_c^*} \to \cH_c$ to an operator $\Hscr_c\to\Hscr_{c,-1}^d$. Its resolvent is given by
\begin{equation}   \label{resA*ext}
\left((\overline\alpha - A^*_c|_{\cH_c})^{-1}[x]\right)(\mu) = \frac{x(\mu) - x(\overline\alpha)}{\overline\alpha - \mu},
  \quad \alpha,\mu\in\cplus,\,[x]\in\Hscr_{c,-1}^d,
\end{equation}
and for $\overline\alpha=\overline\beta$ this is a unitary map from $\cH^d_{c,-1}$ to $\cH_c$. 

\item With $\cH^d_{c, -1}$ identified concretely as in \eqref{Hdc-1con}, the action of $C_c^* : \cY \to \cH^d_{c,-1}$ is given by
\begin{equation}   \label{Cc*con}
C_c^* y:= [ \mu \mapsto \wtsmash\varphi(\mu) y ], \quad y\in\cY,\, \mu\in\cplus.
\end{equation} 
\end{enumerate}
\end{thm}
\begin{proof}  We first check that $\cH^d_{c,-1}$ defined as in \eqref{Hdc-1con} and \eqref{Hdc-1con-norm} is complete.
Suppose that  $[x_n]$ is a Cauchy sequence in $\cH^d_{c,-1}$.  Then the sequence $z_n := \mu\mapsto\frac{x_n(\mu) - x_n(\overline\beta)}{\overline\beta - \mu}$, $\mu\in\cplus$, is Cauchy in $\cH_c$ and it converges to some $z$ in $\cH_c$.  We solve the equation $\frac{x(\mu) - x(\overline\beta)}{\overline\beta - \mu} = z(\mu)$ to come up with
\begin{equation}\label{eq:alphaAdcalc}
\begin{aligned}
  \frac{x(\mu) - x(\overline\beta)}{\overline\beta - \mu} = z(\mu) \quad &\Longleftrightarrow \quad
    x(\mu) = x(\overline\beta) + (\overline\beta - \mu) z(\mu) \\
  &\Longleftrightarrow\quad x \cong \mu\mapsto(\overline\beta - \mu) z(\mu),\quad\mu\in\cplus;
\end{aligned}
\end{equation}
note that the solution is determined only up to an additive constant. By \eqref{eq:alphaAdcalc}, $[\mu\mapsto(\overline\beta - \mu) z(\mu)]\in\cH^d_{c, -1}$ since $z\in\Hscr_c$, and $[x_n] \to [x]$ in $\Hscr_{c,-1}^d$ by \eqref{Hdc-1con-norm} and the fact that $z_n\to z$ in $\Hscr_c$.

\begin{enumerate}
\item If $z\in\Hscr_c$ then by \eqref{Hdc-1con}, $\mu\to z(\mu)+v\in\Hscr_{c,-1}^d$ for all $v\in\Uscr$, since
\begin{equation}\label{eq:classisHc}
  \mu\mapsto\frac{z(\mu)+v-z(\overline\beta)-v}{\overline\beta-\mu} = (\overline\beta-A_c^*)^{-1}z\in\dom{A_c^*},\,\mu\in\cplus,
\end{equation}
see \eqref{Ac*res}, and hence $\iota(\Hscr_c)\subset \Hscr_{c,-1}^d$. From \eqref{eq:classisHc} it also follows that if $[x]=\iota(z)$ for some $z\in\Hscr$ then $\mu\mapsto\frac{x(\mu)-x(\overline\alpha)}{\overline\alpha-\mu} \in\dom{A_c^*}$ for all $\alpha\in\cplus$. Conversely, if $w:=\mu\mapsto\frac{x(\mu)-x(\overline\alpha)}{\overline\alpha-\mu} \in\dom{A_c^*}$ for some $\alpha\in\cplus$, then by \eqref{AC}:
$$
\begin{aligned}
  \big((\overline\alpha-A_c^*)w\big)(\mu)-x(\mu) 
  &= (\overline\alpha-\mu)w(\mu)+\lim_{\re\eta\to\infty} \eta w(\eta)-x(\mu) \\
  &= x(\alpha)-\lim_{\re\eta\to\infty} \eta w(\eta),
\end{aligned}
$$
which is constant, so that $[x]=\iota\big((\overline\alpha-A_c^*)w\big)$.

It is a consequence of the estimate \eqref{eq:HcZspeed} that functions in $\cH_c$ satisfy the decay condition \eqref{decay}.  As two representatives of the same equivalence class differ by a constant, it is clear that there can be at most one representative of a given equivalence class which satisfies \eqref{decay}.  Thus the decay condition \eqref{decay} picks out the unique representative which is in $\cH_c$ (assuming that the equivalence class is in the image of $\iota$). Apart from the claim that $\iota(\Hscr_c)$ is dense in $\Hscr_{c,-1}^d$, Assertion 1 is proved.
 
\item We next suppose that $x \in \cH_c$ and we wish to verify that $[\mu x(\mu)]$ is in $\cH^d_{c,-1}$.  Thus we must check that the function
$z:\mu \mapsto \frac{ \mu x(\mu) - \overline\beta x(\overline\beta)}{\overline\beta - \mu}$ is in $\cH_c$. But we have already verified that this expression is just the formula for $A_c^*(\overline\beta - A_c^*)^{-1} x$, see formula \eqref{Ac*res1}, and hence $z$ is in $\cH_c$ as wanted. Thus $A_c^*|_{\Hscr_c}$ maps $\Hscr_c$ into $\Hscr_{c,-1}$ and it follows from \eqref{AC} that $A_c^*|_{\Hscr_c}$ is an extension of $A_c^*$. It is straightforward to verify that the formula \eqref{Ac*res} for the resolvent of $A_c^*$ extends to \eqref{resA*ext}; use \eqref{Ac*-ext} and read \eqref{eq:alphaAdcalc} backwards. Combining \eqref{resA*ext} with \eqref{Hdc-1con-norm}, we obtain that $(\overline\beta-A_c|_{\Hscr_c})^{-1}$ is isometric from all of $\Hscr_{c,-1}^d$ into $\Hscr_c$. On the other hand, $(\overline\beta-A_c|_{\Hscr_c})^{-1}$ is onto $\Hscr_c$ by \eqref{eq:alphaAdcalc}: for every $z\in\Hscr_c$, $[\mu\mapsto(\overline\beta-\mu)z(\mu)]\in\Hscr_{c,-1}$ is mapped into $z$ by $(\overline\beta-A_c|_{\Hscr_c})^{-1}$.

By construction, once we fix our choice of $\overline\beta \in {\mathbb C}^+$ to define the norm on the $(-1)$-rigged space, the operator $(\overline\beta - A|_{\cH_c})^{-1}$ is unitary from this rigged space onto $\cH_c$. This makes precise the identification of the completion of $\cH_c$ in the norm \eqref{eq:Hcminus1dnorm} with the concrete version of $\Hscr_{c,-1}^d$ given by \eqref{Hdc-1con}--\eqref{Hdc-1con-norm}. Moreover, $A_c^*|_{\Hscr_c}$ maps $\Hscr_c$ continuously into $\Hscr_{c,-1}^d$, since for all $x\in\Hscr_c$:
$$
\begin{aligned}
  \|A_c^*|_{\Hscr_c}x\|_{\Hscr_{c,-1}^d} &= \|(\overline\beta-A_c^*|_{\Hscr_c})^{-1}A_c^*|_{\Hscr_c}x\|_{\Hscr_c}
  = \|\overline\beta(\overline\beta-A_c^*)^{-1}x-x\|_{\Hscr_c} \\
  &\leq \Big||\overline\beta|\,\|(\overline\beta-A_c^*)^{-1}\|+1\Big|\,\|x\|_{\Hscr_c}.
\end{aligned}
$$
Since $\Hscr_{c,-1}^d$ is a completion of $\iota(\Hscr_c)$, it is clear that $\iota(\Hscr_c)$ is dense in $\Hscr_{c,-1}^d$. Now Assertions 1 and 2 are proved completely.

\item By Definition \ref{def:altsysnode}.3, $C_c^* y = A_c^* \& C_c^* \sbm{x \\ y} - A_c^*|_{\cH_c}x$, where $x$ is any choice of function in $\cH_c$ for which $\sbm{x \\ y}$ is in $\dom{ \sbm{A\& B \\ C \& D}_c^*}$;
here we use the fact that such an $x$ exists for every $y \in \cY$ by \eqref{eq:SSplitResDom}. Using \eqref{defz} together with  \eqref{Ac*-ext}, we see that
\begin{align*}
C_c^* y & =  [\mu\mapsto \mu x(\mu) + \wtsmash\varphi(\mu) y - u - 
  \mu x(\mu)]  = [ \mu\mapsto \wtsmash\varphi(\mu) y ],\quad\mu\in\cplus,
\end{align*}
\end{enumerate}
 
\,and \eqref{Cc*con} follows.
\end{proof}

We point out that $\mu\mapsto(\overline\beta - \mu) z(\mu)$ in \eqref{eq:alphaAdcalc} is the unique representative of $[x]$ with value zero at $\overline\beta$; this will be useful in Subsection \ref{sec:repkerncontr} below. Furthermore, it follows from \eqref{AC} that for all $x\in\dom{A_c^*}$:
$$
  \big((\overline\alpha-A_c^*) x\big)(\mu) = (\overline\alpha-\mu) x(\mu) + \lim_{\re\eta \to \infty} \eta x(\eta),
$$
so that for arbitrary $\alpha\in\cplus$:
$$
  \lim_{\re\eta \to \infty} \eta x(\eta) = \big((\overline\alpha-A_c^*) x\big)(\overline\alpha),\quad x\in\dom{A_c^*}.
$$
We next give another interpretation of this limit.

\begin{rem}\label{rem:leftshift}
Assume that $f$ and its distribution derivative $f'$ both lie in $L^2(\rplus;\Yscr)$. Then their Laplace transforms $\whsmash f$ and $\widehat{f'}$ both lie in $H^2(\cplus;\Yscr)$, and upon combining the general Laplace-transform formula $\widehat{f'}(\mu)=\mu\whsmash f(\mu)-f(0)$ with \eqref{eq:H2Zspeed}, it follows that
\begin{equation}\label{eq:lapliminf}
  \lim_{\re \mu \to \infty} \mu \whsmash f(\mu)=f(0).
\end{equation}
Hence, the limit $\lim_{\re\eta \to \infty} \eta x(\eta)$ equals $\check x(0)$, where $\check x$ is the inverse Laplace transform of $x$. Comparing this to \eqref{AC}, we see that $A_c^d$ is the frequency-domain analogue of spatial derivative, the generator of an incoming left shift.

In fact, $\SmallSysNode_c^*$ is a frequency-domain analogue of the standard output-normalized shift realization of $\wtsmash \varphi$, but with $\Hscr_c$ as state space rather than an isometrically contained subspace of $H^2(\cplus;\Uscr)$; see e.g.\ \cite[Def.\ 9.5.1]{StafBook}. The reason for choosing the state space $\Hscr_c$ is that it makes $\SmallSysNode_c$ energy preserving, provided that we choose the appropriate norm in $\Hscr_c$. The realization $\SmallSysNode_c$ is in general not energy preserving if we use the norm of $H^2(\cplus;\Uscr)$ on the state space; see \cite{ArStKYP} for more details on this. These statements are included only in order to provide the reader with some intuition; we make no use of them here and we give no proofs.
\end{rem}

As a corollary of \eqref{Cc*con} and \eqref{Hdc-1con}, we see that
\begin{equation} \label{phidifquotin}
 \mu \mapsto \frac{\wtsmash\varphi(\mu) - \wtsmash\varphi(\overline\alpha)}{\overline\alpha - \mu}y  \in \cH_c\quad \text{for all}~ \alpha\in {\mathbb C}^+
 \text{ and } y \in \cY.
 \end{equation}
 In fact the formula \eqref{Ac*res2} identifying this expression with $(\overline\alpha - A_c^*|_{\cH_c})^{-1}C_c^* y$ can now be seen as a consequence  of the formula \eqref{Cc*con} combined with \eqref{resA*ext}.

Finally, note that we can recover $\SmallSysNode_c^*$ from $A_c^*$, $C_c^*$, $B_c^*$, and $\wtsmash\varphi$ evaluated at one arbitrary point in $\cplus$, as described in Remark \ref{rem:reconstruct}. Now we return to studying the system node $\SmallSysNode_c$ rather than its dual.

\subsection{More explicit formulas for the controllable model}

In this subsection we obtain more explicit formulas (to the extent possible) for the action of the operators $A_c$, $B_c$, and $C_c$ in $\SmallSysNode_c$. 

Let us say that an expression of the form $(\alpha - A_{c})^{-1} 
e_{c}(\lambda)^{*} u$, where $\alpha,\lambda\in\cplus$ and $u\in\Uscr$, is a {\em regularized kernel function}.  While 
a kernel function $e_{c}(\lambda)^{*} u$ itself may not be in 
$\operatorname{dom} (A_{c})$, it turns out that a regularized kernel function is always 
in $\operatorname{dom} (A_{c})$. More precisely, by the first assertion in the following proposition
\begin{equation}\label{eq:resrangespan}
\begin{aligned}
  \Fscr_0 :&=\spn \set{\frac{e_c(\lambda)^*u-e_c(\overline\alpha)^*u}{\alpha-\overline\lambda}\bigmid
  \lambda\in\cplus,\,u\in\Uscr} \\
  &= (\alpha-A_c)^{-1}\spn \set{e_c(\lambda)^*u\mid \lambda\in\cplus,\,u\in\Uscr},
\end{aligned}
\end{equation}
and by Lemma \ref{lem:contrdense} and \eqref{eq:isomeval} this linear span is a dense subspace of $\dom{A_c}$ in the graph norm for all fixed $\alpha\in\cplus$. In particular, the difference of two kernel functions is in $\dom{A_c}$. (The linear span $\Fscr_0$ can also be viewed as a dense subspace of $\Hscr_c$.)

The first result gives more explicit formulas for the actions of $A_c$ and $C_c$ on regularized kernel functions:

\begin{prop}\label{prop:isomops}
The following statements hold for the system node $\SmallSysNode_c$ in Theorem \ref{thm:firstmodel}:
\begin{enumerate}

\item The action of the resolvent of $A_c$ on kernel functions is given by
\begin{equation}\label{eq:isomres}
  (\alpha-A_c)^{-1}\big(e_c(\lambda)^*u\big) = 
  \frac{e_c(\lambda)^*-e_c(\overline\alpha)^*}{\alpha-\overline\lambda}u,
  \quad \alpha,\lambda\in\cplus,\,u\in\Uscr.
\end{equation}
The formula \eqref{eq:isomres} uniquely determines the action of $(\alpha - A_c)^{-1}$ on the whole space $\cH_c$ by linearity and continuity.

\item The main operator $A_c$ and observation operator $C_c$ of $\SmallSysNode_c$ have 
the following actions on regularized kernel functions:
\begin{align}
  A_c\big( (\alpha-A_c)^{-1} e_c(\lambda)^*u\big) &=
  \frac{\overline\lambda e_c(\lambda)^*-\alpha e_c(\overline\alpha)^*}{\alpha-\overline\lambda}u,
  \quad\text{and} \label{eq:isomAdense}\\
  C_c\big( (\alpha-A_c)^{-1} e_c(\lambda)^*u\big) &=
  \frac{\varphi(\overline\lambda)-\varphi(\alpha)}{\alpha-\overline\lambda}u,
  \quad \alpha,\lambda\in\cplus,\, u\in\Uscr.\label{eq:isomCdense}
\end{align}
Moreover these formulas uniquely determine $A_c$ and $C_c$ on the whole space $\cH_{c,1}=\dom{A_c}$.
\end{enumerate}
\end{prop}

\noindent\emph{Proof.}
\begin{enumerate}
\item We first prove that the resolvent of $A_c$ satisfies \eqref{eq:isomres}. 
For arbitrary $\alpha,\lambda\in\cplus$ and $u\in\Uscr$ we by the definition \eqref{eq:SodDef} 
of $\SmallSysNode_c$ have that $\sbm{e_c(\lambda)^*u\\u},
\sbm{e_c(\overline\alpha)^*u\\u}\in\dom{\SmallSysNode_c}$ and that
\begin{equation}\label{eq:isomACcomp}
\begin{aligned}
  \bbm{e_c(\lambda)^*u\\u}-\bbm{e_c(\overline\alpha)^*u\\u} &=
  \bbm{e_c(\lambda)^*u-e_c(\overline\alpha)^*u\\0}\in\dom{\SmallSysNode_c},\\
  \bbm{A_c\\C_c}\big(e_c(\lambda)^*-e_c(\overline\alpha)^*\big)u &= 
  \SysNode_c\left(\bbm{e_c(\lambda)^*u\\u}-\bbm{e_c(\overline\alpha)^*u\\u}
  \right)\\
  &= \bbm{\overline\lambda e_c(\lambda)^*-\alpha e_c(\overline\alpha)^*\\\varphi(\overline\lambda)
  -\varphi(\alpha)}u,
\end{aligned}
\end{equation}
so that in particular $e_c(\lambda)^*u-e_c(\overline\alpha)^*u\in\dom{A_c}$ and 
$$
  A_c\big(e_c(\lambda)^*u-e_c(\overline\alpha)^*u\big)=\overline\lambda 
  e_c(\lambda)^*u-\alpha e_c(\overline\alpha)^*u.
$$
Now clearly
$$
\begin{aligned}
  (\alpha-A_c)\big(e_c(\lambda)^*u-e_c(\overline\alpha)^*u\big) 
    &= \\ \alpha e_c(\lambda)^*u - \alpha e_c(\overline\alpha)^*u - 
    \overline \lambda e_c(\lambda)^*u + \alpha e_c(\overline\alpha)^*u &= \\
    (\alpha -\overline \lambda) e_c(\lambda)^*u,
\end{aligned}
$$
which shows that the resolvent satisfies \eqref{eq:isomres}.  As the span of kernel functions is dense in $\cH_c$ and $(\alpha - A_c)^{-1}$ is a bounded linear operator on $\cH_o$, we see that \eqref{eq:isomres} uniquely determines $(\alpha - A_c)^{-1}$ on the whole space $\cH_c$.

\item It follows from \eqref{eq:isomACcomp} and \eqref{eq:isomres} that $A_c$ and $C_c$ 
satisfy \eqref{eq:isomAdense} and \eqref{eq:isomCdense}, and therefore the maps $A_c$ and $C_c$ are 
determined by these formulas on the dense subspace $\Fscr_0$ of their domain $\dom{A_c}$. Since $A_c$ and $C_c$ are bounded from $\dom{A_c}$ to $\Hscr_c$ and $\Yscr$, respectively, the claim follows. $\square$
\end{enumerate}

Proposition \ref{prop:isomops} is of a preliminary nature, and now we proceed to search for the actions of operators on generic elements of their domains. The following important result is a consequence of Theorem \ref{thm:msw}.

\begin{thm}\label{thm:enprescontainedadjoint}
The following claims are true:

\begin{enumerate}
\item For arbitrary $\sbm{x\\u}\in\dom{\SmallSysNode_c}$, we can set $y:=\bbm{C_c\&D_c}\sbm{x\\u}$ and obtain
\begin{equation}\label{eq:AcBcAction}
  \bbm{A_c\&B_c}\bbm{x\\u}=\mu\mapsto -\mu x(\mu)-\wtsmash\varphi(\mu)y+u,\quad\mu\in\cplus.
\end{equation}
Here $u$ can be recovered from $\sbm{x\\y}$ using \eqref{defu}. Moreover,
\begin{equation}   \label{Ac-id}
(A_cx)(\mu) = -\mu x(\mu) - \wtsmash \varphi(\mu) C_c x,\quad x\in\dom{A_c},\,\mu\in\cplus.
\end{equation}

\item We have
\begin{equation}\label{eq:contrdomcont}
\begin{aligned}
  \dom{\SmallSysNode_c} &\subset\left\{\bbm{x\\u}\in\bbm{\Hscr_c\\\Uscr} \bigmid \right. \\
   &\qquad \exists y\in\Yscr:~ \mu\mapsto -\mu x(\mu)-\wtsmash\varphi(\mu)y+u\in\Hscr_c\bigg\},
\end{aligned}
\end{equation}
and in particular
\begin{equation}\label{eq:contrdomAcont}
  \dom{A_c}\subset\set{x\in\Hscr_c \bigmid \exists y\in\Yscr:~ \mu\mapsto \mu x(\mu)+\wtsmash\varphi(\mu)y\in\Hscr_c}.
\end{equation}
\end{enumerate}
\end{thm}
\begin{proof}
The equation \eqref{eq:AcBcAction} follows from the energy-preserving property of $\SmallSysNode_c$, Theorem \ref{thm:msw}.1, and Theorem \ref{thm:contrdual}. Taking $u=0$ and using the definitions \eqref{eq:Adef} and \eqref{eq:Cdef} of $A_c$ and $C_c$, we obtain \eqref{Ac-id}. By Definition \ref{def:altsysnode}, $\bbm{A_c\&B_c}$ maps $\dom{\SmallSysNode_c}$ into $\Hscr_c$, and combining this with \eqref{eq:AcBcAction}, we get \eqref{eq:contrdomcont}. Taking $u=0$ in \eqref{eq:contrdomcont} together with \eqref{eq:Adef} of $A_c$, we arrive at \eqref{eq:contrdomAcont}.
\end{proof}

\begin{cor}\label{cor:projUY}
For every $u\in\Uscr$, there exist $x\in\Hscr_c$ and $y\in\Yscr$, such that
\begin{equation}\label{eq:funinHc}
  \mu\mapsto -\mu x(\mu)-\wtsmash\varphi(\mu)y+u\in\Hscr_c.
\end{equation}
Also, for every $y\in\Yscr$, there exist $x\in\Hscr_c$ and $u\in\Uscr$, such that \eqref{eq:funinHc} holds. The set of $x\in\Hscr_c$, for which there exist $u\in\Uscr$ and $y\in\Yscr$, such that \eqref{eq:funinHc} holds, is dense in $\Hscr_c$.

Moreover, \eqref{eq:funinHc} holds if and only if $\sbm{x\\y}\in\dom{\SmallSysNode_c^*}$ and $u=\bbm{B_c^*\&D_c^*}\sbm{x\\y}$.
\end{cor}
\begin{proof}
By \eqref{eq:SodDefprel}, for all $u\in\Uscr$, we have $\sbm{e_c(\overline\lambda)^*u\\u}\in\dom{\SmallSysNode_c}$ for every $\lambda\in\cplus$, and by Theorem \ref{thm:enprescontainedadjoint}.2, this pair $\sbm{x\\u}$ satisfies \eqref{eq:funinHc} with $y:=\bbm{C_c\&D_c}\sbm{x\\u}=\varphi(\lambda)u$. The condition \eqref{eq:funinHc} is equivalent to $\sbm{x\\y}\in\dom{\SmallSysNode_c^*}$ and $u=\bbm{B_c^*\&D_c^*}\sbm{x\\y}$ by Theorem \ref{thm:contrdual}. Now the proof is completed by combining \eqref{eq:SSplitResDom} with the fact that $\SmallSysNode_c^*$ is a system node with input space $\Yscr$. (Recall that $\dom{A^*}$ is dense in $\Xscr$ for a system node with main operator $A^*$ and state space $\Xscr$.)
\end{proof}

The inclusion \eqref{eq:contrdomcont} is in general not an equality and hence \eqref{eq:funinHc} does not imply $\sbm{x\\u}\in\dom{\SmallSysNode_c}$; this brings us some difficulties later. If we know $A_c$, including its domain, then it will soon turn out that we can calculate $C_cx$ for generic $x\in\dom{A_c}$. Then the following continuous-time version of Theorem \ref{T:iso-real} gives a description of $\SmallSysNode_c$, including its domain:

\begin{thm}\label{thm:contrsemiexplicit}
A pair $\sbm{x\\u}\in\sbm{\Hscr_c\\\Uscr}$ lies in $\dom{\SmallSysNode_c}$ if and only if for some, or equivalently for all, $\lambda\in\cplus$, the function $x-e_c(\lambda)^*u$ lies in $\dom{A_c}$. When this is the case, for an arbitrary fixed $\lambda\in\cplus$, the action of $\SmallSysNode_c$ is
\begin{equation}\label{eq:contrrealgen}
\begin{aligned}
  \SysNode_c\bbm{x\\u} &= \bbm{\mu\mapsto-\mu x(\mu)
    -\varphi(\overline\mu)^*\gamma_\lambda+\big(1-\varphi(\overline\mu)^*\varphi(\overline\lambda)\big)u
    \\ \gamma_\lambda + \varphi(\overline\lambda)u},\\
  \gamma_\lambda &= C_c\big(x-e_c(\lambda)^*u\big).
\end{aligned}
\end{equation}
\end{thm}
\begin{proof}
By \eqref{eq:SSplitResDom} we have that $x-(\overline\lambda-A_c|_{\Hscr_c})^{-1}B_cu\in\dom{A_c}$ if and only if $\sbm{x\\u}\in\dom{\SmallSysNode_c}$ for arbitrary $\lambda\in\cplus$, and this should be combined with \eqref{eq:isomeval} to verify the claim on the domains. It follows from \eqref{eq:CDreconstr} and \eqref{eq:isomeval} that $\bbm{C_c\&D_c}\sbm{x\\u}=\gamma_\lambda+\varphi(\overline\lambda)u$, and together with Theorem \ref{thm:enprescontainedadjoint}.1, this implies \eqref{eq:contrrealgen}.
\end{proof}

In the sequel, the following familiar bounded operator from $\Hscr_c$ into $\Yscr$ plays such an important role that we give it a special notation:

\begin{defn}
We denote
\begin{equation}\label{eq:taucdef}
  \tau_{c,\alpha}:=C_c(\alpha-A_c)^{-1},\quad\alpha\in\cplus.
\end{equation}
\end{defn}

 From \eqref{Ac*res2} it then follows that for all $y\in\Yscr$:
 \begin{equation}\label{taualphacadj}
  (\tau_{c,\alpha} )^*  y = (\overline\alpha - A_c^*|_{\cH_c})^{-1}C_c^*y = \mu \mapsto \frac{ \wtsmash\varphi(\mu) - \wtsmash\varphi(\overline\alpha)}{\overline\alpha - \mu},\quad\mu\in\cplus.
  \end{equation}
Hence, if $x$ happens to lie in $\dom A$, then taking $u=0$ in \eqref{eq:contrrealgen} yields (note the absence of $\lambda$ from the right-hand side)
$$
  \gamma_\lambda = C_c x= \tau_{c,\alpha}(\alpha-A_c)x
$$
for $\alpha\in\cplus$, and moreover, $\gamma_\lambda$ is the unique element in $\Yscr$, such that
$$
     \langle \gamma_\lambda, y \rangle_{\cY} = \left\langle (\alpha-A_c)x,
      \mu \mapsto \frac{ \wtsmash\varphi(\mu) - \wtsmash\varphi(\overline\alpha)}{\overline\alpha - \mu}y 
     \right\rangle_{\Hscr_c} \text{ for all } y \in \cY.
$$
This should be compared to the definition of $\wtsmash g(0)$ in \eqref{tildeg(0)}.

In fact, the arbitrary parameter $\lambda$ in Theorem \ref{thm:contrsemiexplicit} is only used for calculating $y=\bbm{C_c\&D_c}\sbm{x\\u}$, which is obviously independent of $\lambda$, and the formula for $\bbm{A_c\&B_c}\sbm{x\\u}$ only uses $y$; compare \eqref{eq:contrrealgen} to \eqref{eq:AcBcAction}. In Theorem \ref{T:iso-real} there is no need for any arbitrary parameter $w$, because there $\sbm{x\\0}\in\dom{\sbm{\mathrm A_c&\mathrm B_c\\\mathrm C_c&\mathrm D_c}}$ for all $x\in\mathrm H_c$, since $\sbm{\mathrm A_c&\mathrm B_c\\\mathrm C_c&\mathrm D_c}$ is bounded.

\begin{rem}\label{rem:removecircle}
The use of $\tau_{c,\alpha}$ thus allows us to calculate $C_c$ on generic elements of $\dom{A_c}$ using $C_c=\tau_{c,\alpha}(\alpha-A_c)$ and \eqref{taualphacadj}, assuming that we have an explicit formula for $A_c$, but note that the formula \eqref{Ac-id} gives $A_c$ in terms of $C_c$. This circle definition can be corrected if the condition $\cH_c \cap \set{\wtsmash \varphi (\cdot)y\mid y\in \cY} = \{0\}$ holds. Indeed, under this assumption, the function $\wtsmash\varphi(\cdot)y$, such that $\mu\mapsto \mu x(\mu)+\wtsmash\varphi(\mu)y\in\Hscr_c$, is uniquely determined by $x$. (Such a $y$ exists for every $x\in\dom{A_c}$ by \eqref{eq:contrdomAcont}.) Note that $y$ will in general not be uniquely determined, only the function $\wtsmash\varphi(\cdot)y$; see Lemma \ref{lem:domequivdeluxe} and Theorem \ref{thm:conschar} below for more details on this.
\end{rem}

With the help of $\tau_{c,\alpha} $, we have an explicit formula for the resolvent operator $(\alpha - A_c)^{-1}$, $\alpha \in {\mathbb C}^+$, on generic elements of $\Hscr_c$:

\begin{cor}\label{cor:ContrRes}
The resolvent operator $(\alpha - A_c)^{-1}$ acting on arbitrary functions in $\cH_c$ is given explicitly by
  \begin{equation}   \label{Acresol}
  \left(  (\alpha - A_c)^{-1} x \right)(\mu) = \frac{ x(\mu) - \wtsmash\varphi(\mu) \tau_{c,\alpha} x}{ \alpha + \mu},
    \quad x\in\Hscr_c,\,\alpha,\mu\in\cplus.
\end{equation}
\end{cor}
\begin{proof}
Setting $z:=(\alpha-A_c)^{-1}x$ in \eqref{Acresol} and using \eqref{eq:taucdef}, we obtain the equivalent condition
$$
  (\alpha + \mu)z(\mu) = \big((\alpha-A_c)z\big)(\mu) - \wtsmash\varphi(\mu) C_cz,
    \quad z\in\dom{A_c},\,\mu\in\cplus,
$$
which is true by \eqref{Ac-id}.
\end{proof}

Using \eqref{Acresol} involves calculating $\tau_{c,\alpha} x$ for generic $x\in\Hscr_c$, but we have no formula for this except for in the case when $x$ is a kernel function. One way to calculate $\tau_{c,\alpha} x$ is to use \eqref{taualphacadj} and calculate
$$
  \Ipdp{\tau_{c,\alpha} x}{\gamma}_\Yscr=\Ipdp{x}{\mu\mapsto\frac{\wtsmash\varphi(\mu)-\wtsmash\varphi(\overline\alpha)}
    {\overline\alpha-\mu}\gamma}_{\Hscr_c}, \quad \gamma\in\Yscr.
$$

From the domain of a system node $\SmallSysNode$, the domain of $A_c$ is constructed using \eqref{eq:Adef}. Conversely, if we know $\dom{A_c}$, then we can recover $\dom{\SmallSysNode}$ using \eqref{eq:SSplitResDom} as in the proof of Theorem \ref{thm:contrsemiexplicit}. In particular, the following result shows that we have equality in \eqref{eq:contrdomcont} if and only if we have equality in \eqref{eq:contrdomAcont}.

\begin{lem}\label{lem:domequivdeluxe}
The following claims are true:
\begin{enumerate}
\item The condition
\begin{equation}\label{eq:contrdomconteq}
\begin{aligned}
  \dom{\SmallSysNode_c} &= \bigg\{\bbm{x\\u}\in\bbm{\Hscr_c\\\Uscr} \bigmid \\
    &\quad\exists y\in\Yscr:~ \mu\mapsto -\mu x(\mu)-\wtsmash\varphi(\mu)y+u\in\Hscr_c\bigg\},
\end{aligned}
\end{equation}
holds if and only if
\begin{equation}\label{eq:contrdomAconteq}
  \dom{A_c}=\set{x\in\Hscr_c \mid \exists y\in\Yscr:~ \mu\mapsto \mu x(\mu)+\wtsmash\varphi(\mu)y\in\Hscr_c}.
\end{equation}

\item It holds that
 \begin{equation}   \label{assume1c}
 \cH_c \cap \set{\wtsmash \varphi (\cdot)y\mid y\in \cY} = \{0\}
 \end{equation}
if and only if for some (or equivalently for all) $\alpha\in\cplus$:
\begin{equation}\label{eq:contrconsconseq}
  \dom{A_c}\cap\set{\mu\mapsto\frac{\wtsmash\varphi(\mu)y}{\alpha+\mu}\bigmid y\in\Yscr} = \zero.
\end{equation}

\item The conditions \eqref{eq:contrdomconteq}--\eqref{eq:contrconsconseq} hold if and only if for some (or equivalently for all) $\alpha\in\cplus$:
\begin{equation}\label{eq:contrconscond}
  \Hscr_c\cap\set{\mu\mapsto\frac{\wtsmash\varphi(\mu)y}{\alpha+\mu} \bigmid y\in\Yscr} = \zero.
\end{equation}
\end{enumerate}
\end{lem}
\begin{proof}
\begin{enumerate}
\item By Theorem \ref{thm:enprescontainedadjoint}, the domains are included in the sets on the right-hand sides of \eqref{eq:contrdomconteq} and \eqref{eq:contrdomAconteq}, so we only need to show that the converse inclusions are equivalent. First assume \eqref{eq:contrdomconteq} and let $x\in\Hscr_c$ and $y\in\Yscr$ be such that
\begin{equation}\label{eq:funindomAc}
  \mu\mapsto \mu x(\mu)+\wtsmash\varphi(\mu)y\in\Hscr_c.
\end{equation}
The working assumption \eqref{eq:contrdomconteq} then implies that $\sbm{x\\0}\in\dom{\SmallSysNode_c}$, and according to \eqref{eq:Adef}, this precisely means that $x\in\dom{A_c}$.

Now assume \eqref{eq:contrdomAconteq} and \eqref{eq:funinHc}. Since $\alpha K_c(\cdot,\overline\alpha)u\in\Hscr_c$, we also have
$$
\begin{aligned}
  -\mu x(\mu)-\wtsmash\varphi(\mu)y+u-\alpha K_c(\mu,\overline\alpha)u =\qquad \\
  -\mu x(\mu) + \mu\frac{1-\wtsmash\varphi(\mu)\varphi(\alpha)}{\mu+\alpha}u-\wtsmash\varphi(\mu)\big( y-\varphi(\alpha)u \big)
\end{aligned}
$$
in $\Hscr_c$ as a function of $\mu\in\cplus$. Hence,
\begin{equation}\label{eq:domAchartemp}
  \exists \gamma\in\Yscr:~ \mu\mapsto -\mu x(\mu)
    +\mu\frac{1-\wtsmash\varphi(\mu)\varphi(\alpha)}{\mu+\alpha}u-\wtsmash\varphi(\mu)\gamma\in\Hscr_c,
\end{equation}
since one can simply take $\gamma:=y-\varphi(\alpha)u$. The statement \eqref{eq:domAchartemp} is by \eqref{eq:contrdomAconteq} equivalent to
$x-e_c(\overline\alpha)^*u\in\dom{A_c}$, and according to the first assertion in Theorem \ref{thm:contrsemiexplicit}, this is equivalent to $\sbm{x\\u}\in\dom{\SmallSysNode_c}$.

\item First assume \eqref{assume1c} and suppose that $z:=\mu\mapsto\frac{\wtsmash\varphi(\mu)y}{\alpha+\mu}\in\dom{A_c}$ for some arbitrary $y\in\Yscr$ and $\alpha\in\cplus$. Use \eqref{Ac-id} and\eqref{eq:taucdef} to calculate
$$
\begin{aligned}
  \big((\alpha-A_c)z\big)(\mu) &= (\alpha+\mu) z(\mu) + \wtsmash \varphi(\mu)C_c z \\
  &= \wtsmash \varphi(\mu)y + \wtsmash \varphi(\mu) C_c z \in \Hscr_c\cap\wtsmash\varphi(\cdot)\Yscr.
\end{aligned}
$$
By \eqref{assume1c} this quantity is 0, and since $\alpha-A_c$ is injective, we have $z=0$, and it follows that \eqref{eq:contrconsconseq} holds for all $\alpha\in\cplus$. 

Conversely, assume that \eqref{eq:contrconsconseq} holds for some $\alpha\in\cplus$ and suppose that $x:=\wtsmash\varphi(\cdot)y$ is in $\Hscr_c$ for some $y\in\Yscr$. Use \eqref{Acresol} to calculate
$$
\begin{aligned}
  \left(  (\alpha - A_c)^{-1} x \right)(\mu) &= \frac{ x(\mu) - \wtsmash\varphi(\mu) \tau_{c,\alpha}\,x}{ \alpha + \mu} \\
  &= \frac{ \wtsmash\varphi(\mu)}{\alpha + \mu} \big(y- \tau_{c,\alpha}\,x\big)
  \in \dom{A_c}\cap\frac{\wtsmash\varphi(\cdot)}{\alpha+\cdot}\Yscr.
\end{aligned}
$$
By \eqref{eq:contrconsconseq} this quantity is 0, and hence also $x=0$, which proves \eqref{assume1c}.

\item First assume that \eqref{eq:contrconscond} is satisfied. Then trivially \eqref{eq:contrconsconseq} holds, since $\dom{A_c}\subset\Hscr_c$, and we next prove that \eqref{eq:contrdomAconteq} is satisfied too.  Suppose that $x\in\Hscr_c$ and $y\in\Yscr$ are such that \eqref{eq:funindomAc} holds. Then for every $\alpha\in\cplus$ it also holds that
$$
\begin{aligned}
  z: &= \mu\mapsto (\alpha+\mu) x(\mu)+\wtsmash\varphi(\mu)y\in\Hscr_c \quad\text{and thus}\\
  x &= \mu\mapsto\frac{z(\mu)-\wtsmash\varphi(\mu)y}{\alpha+\mu}\in\Hscr_c.
\end{aligned}
$$
On the other hand we have
$$
  (\alpha-A_c)^{-1}z = \mu\mapsto\frac{z(\mu)-\wtsmash\varphi(\mu)\tau_{c,\alpha} z}{\alpha+\mu}\in\dom{A_c}\subset\Hscr_c,
$$
and these two together imply that $\mu\mapsto \frac{\wtsmash\varphi(\mu)(y-\tau_{c,\alpha} z)}{\alpha+\mu}\in\Hscr_c$. The working assumption \eqref{eq:contrconscond} then gives that $\frac{\wtsmash\varphi(\mu)(y-\tau_{c,\alpha} z)}{\alpha+\mu}=0$ for all $\mu\in\cplus$, i.e., $\wtsmash\varphi(\cdot)y=\wtsmash\varphi(\cdot)\tau_{c,\alpha} z$, and this implies that $x=(\alpha-A_c)^{-1}z\in\dom{A_c}$.

Finally, with the objective of showing \eqref{eq:contrconscond}, we assume that $\alpha\in\cplus$ is such that \eqref{eq:contrdomAconteq} and \eqref{eq:contrconsconseq} hold. Then we pick an $x:=\mu\mapsto \frac{\wtsmash\varphi(\mu)\eta}{\alpha+\mu}\in\Hscr_c$, so that also $\alpha x\in\Hscr_c$, and thus by \eqref{eq:contrdomAconteq}:
$$
\begin{aligned}
  x\in\dom{A_c} \quad&\Longleftrightarrow\quad \exists \gamma\in\Yscr:~
  \mu\mapsto \mu\frac{\wtsmash\varphi(\mu)\eta}{\alpha+\mu}+\wtsmash\varphi(\mu)\gamma\in\Hscr_c \\
  &\Longleftrightarrow\quad \exists \gamma\in\Yscr:~
  \mu\mapsto \mu\frac{\wtsmash\varphi(\mu)\eta}{\alpha+\mu}+\alpha\frac{\wtsmash\varphi(\mu)\eta}{\alpha+\mu}+\wtsmash\varphi(\mu)\gamma\in\Hscr_c.
\end{aligned}
$$
\end{enumerate}

This is seen to be true by simply choosing $\gamma:=-\eta$. Thus every $x\in\Hscr_c$ 

of the form $x(\mu)=\frac{\wtsmash\varphi(\mu)y}{\alpha+\mu}$ also lies in $\dom{A_c}$, and \eqref{eq:contrconsconseq} finally gives 

the desired result.
\end{proof}

Theorem \ref{T:iso-real} includes a characterization of the case where controllable isometric realization is in fact unitary. The corresponding situation in the present paper is that $\SmallSysNode_c$ is not only energy preserving, but even conservative, so that also $\SmallSysNode_c^*$ is also energy preserving. We have the following characterizations of this case:

\begin{thm}\label{thm:conschar}
The following conditions are equivalent:
\begin{enumerate}
\item The system node $\SmallSysNode_c$ is conservative.
\item The condition \eqref{eq:contrdomconteq} holds together with the following implication:
\begin{equation}\label{eq:transfnotinHc}
  \wtsmash\varphi(\cdot)y\in\Hscr_c\quad\Longrightarrow\quad y=0.
\end{equation}
\item The condition \eqref{eq:contrconscond} holds for some (or equivalently for all) $\alpha\in\cplus$ and
\begin{equation}\label{eq:transfnotinHc2}
  \wtsmash\varphi(\mu)y=0~\text{for all}~\mu\in\cplus\quad\Longrightarrow\quad y=0.
\end{equation}
\item The function $1 - \widetilde \varphi(\cdot)^* \widetilde \varphi(\cdot)$ has maximal factorable minorant in the right half-plane sense equal to 0, i.e.,
if $a : {\mathbb C}^+ \to \cB(\cY, \cY')$ is holomorphic with
$a(\mu)^* a(\mu) \le 1 - \widetilde \varphi(\mu)^* \widetilde \varphi(\mu)$ for almost all $\mu$ on the imaginary line $i {\mathbb R}$, then $a = 0$.
\end{enumerate}
\end{thm}
\begin{proof}
This proof is heavily based on Theorem \ref{thm:msw}. We first prove that \emph{3.\ implies 2.}: Assume \eqref{eq:contrconscond} and \eqref{eq:transfnotinHc2}. Then \eqref{eq:contrdomconteq} and \eqref{assume1c} hold by Lemma \ref{lem:domequivdeluxe}.3. Hence, if $\wtsmash\varphi(\cdot)y\in\Hscr_c$, then $\wtsmash\varphi(\mu)y=0$, for all $\mu\in\cplus$, which by \eqref{eq:transfnotinHc2} implies $y=0$.

We next show that \emph{2.\ implies 1.} Assume that \eqref{eq:contrdomconteq} and \eqref{eq:transfnotinHc} hold. By Theorem \ref{thm:msw}.2, we need to show that $\sbm{\sbm{1&0}\\\sbm{C_c\&D_c}}$ maps $\dom{\SmallSysNode_c}$ onto $\dom{\SmallSysNode_c^*}$. Hence fix $\sbm{x\\y}\in\dom{\SmallSysNode_c^*}$ arbitrarily and let $u$ be the unique element in $\Uscr$ for which \eqref{eq:funinHc} holds; see \eqref{defdom}. Then $\sbm{x\\u}\in\dom{\SmallSysNode_c}$ by \eqref{eq:contrdomconteq} and we may define $\eta:=\bbm{C_c\&D_c}\sbm{x\\u}$. It follows from Theorem \ref{thm:enprescontainedadjoint}.1 that \eqref{eq:funinHc} holds also with $y$ replaced by $\eta$, and hence $\wtsmash\varphi(\cdot)(y-\eta)\in\Hscr_c$. Now the implication \eqref{eq:transfnotinHc} gives that $y=\eta=\bbm{C_c\&D_c}\sbm{x\\u}$.

Finally, we prove that \emph{1.\ implies 3.}: By Theorem \ref{thm:msw}.2, conservativity of $\SmallSysNode_c$ implies that $\Rscr:=\range{\SmallSysNode_c+\sbm{\overline\alpha&0\\0&0}}$ is dense in $\sbm{\Hscr_c\\\Yscr}$ for some, or equivalently for all, $\alpha\in\cplus$. From the construction of $\SmallSysNode_c$ it follows that
$$
  \Gscr:=\spn\set{\bbm{(\overline\alpha+\mu) e_c(\overline\mu)^*u\\\varphi(\mu)u}\bigmid \mu\in\cplus,\,u\in\Uscr}
$$ 
is dense in $\Rscr$, because by the construction of $\SmallSysNode_c$, the graph of $\SmallSysNode_c+\sbm{\overline\alpha&0\\0&0}$ is the closure of 
$$
  \spn\set{\bbm{(\overline\alpha+\mu) e_c(\overline\mu)^*u\\\varphi(\mu)u\\e_c(\overline\mu)^*u\\u}\bigmid \mu\in\cplus,\,u\in\Uscr}
  \quad\text{in}\quad \bbm{\Hscr_c\\\Yscr\\\Hscr_c\\\Uscr}.
$$
Therefore $\Gscr$ is also dense in $\sbm{\Hscr_c\\\Yscr}$. 

Now note that we have 
$$
\begin{aligned}
  \bbm{z\\y}\in\bbm{\Hscr_c\\\Yscr}\ominus \Gscr \quad &\Longleftrightarrow\\ 
  \Ipdp{\bbm{(\overline\alpha+\mu)e_c(\overline\mu)^*u\\\varphi(\mu)u}}{\bbm{z\\y}}_{\sbm{\Hscr_c\\\Yscr}}=0\quad \text{for all}~\mu\in\cplus,\,u\in\Uscr \quad
  &\Longleftrightarrow\\
  \quad z\in\Hscr_c\quad\text{and}\quad z(\overline\mu)=-\frac{\wtsmash\varphi(\overline\mu)y}{\alpha+\overline\mu},\quad\mu\in\cplus.\quad&
\end{aligned}
$$
Thus, if $z:=\mu\mapsto\frac{\wtsmash\varphi(\mu)y}{\alpha+\mu}\in\Hscr_c$ then we can denote this function by $-z$ and get that $\sbm{z\\y}\in\sbm{\Hscr_c\\\Yscr}\ominus\Gscr$, which then by the above implies that $z=0$ and $y=0$.

The equality of conservativity and assertion four is reduced to the corresponding discrete result in Theorem \ref{thm:discrconschar} using the Cayley transform described in Section \ref{sec:recovering}.
\end{proof}

\begin{rem}\label{rem:conservativity}
We can make the following interesting observations:
\begin{enumerate}
\item The condition \eqref{eq:transfnotinHc} implies \eqref{assume1c}. Together with \eqref{eq:transfnotinHc2},  \eqref{assume1c} also implies \eqref{eq:transfnotinHc}. Note that \eqref{eq:transfnotinHc2} can also be written  $\cap_{\mu\in\cplus}\Ker{\wtsmash\varphi(\mu)}=\zero$. Implication \eqref{eq:transfnotinHc} holds, e.g., if $\wtsmash\varphi$ is bounded away from zero on $i\R$. Indeed, in this case $\wtsmash\varphi(\cdot) y$ is not even in $L^2(i\mathbb R, \cU)$ for any nonzero $y \in \cY$, and {\it a fortiori}, $\wtsmash\varphi(\cdot) y\not\in\Hscr_c$. In Example \ref{ex:constantc} below the implication \eqref{eq:transfnotinHc} is false but \eqref{assume1c} is true.

\item Note that \eqref{eq:contrconscond} is true if $\wtsmash \varphi$ is inner: In this case $\Hscr_c=H^2({\mathbb C}^+ ; \cU ) \ominus M_{\wtsmash \varphi} H^2({\mathbb C}^+; \cY)$ isometrically by Corollary \ref{cor:inner} and the function $\mu\mapsto\frac{\wtsmash\varphi(\mu)y}{\alpha+\mu}$, $\mu\in\cplus$, is $M_{\wtsmash\varphi}$ applied to the kernel function $e(\overline\alpha)^*y$ in $H^2(\cplus;\Yscr)$, cf.\ Lemma \ref{lem:MphiProps}.2. Hence, for every co-inner $\varphi$, the model $\SmallSysNode_c$ is conservative.

\item On the other extreme of the situation in 2., if $\|\varphi\|_{H^\infty}=\delta<1$, then $\Hscr_c$ is simply a re-normed version of $H^2(\cplus;\Uscr)$ by \eqref{lifted norm}, since $\sqrt{1-\delta^2}\leq (1-M_{\wtsmash\varphi}M_{\wtsmash\varphi}^*)^{1/2}\leq 1$ and $Qx=x$, because $1-M_{\wtsmash\varphi}M_{\wtsmash\varphi}^*$ is injective. Then the intersection in \eqref{eq:contrconscond} is all of $\set{\mu\mapsto\frac{\wtsmash\varphi(\mu)y}{\alpha+\mu} \bigmid y\in\Yscr,\, \mu\in\cplus}$ and $\Hscr_c$ is not conservative unless $\varphi(\mu)=0$ for all $\mu\in\cplus$. If $\varphi$ is identically zero, then condition \eqref{eq:transfnotinHc2} is violated and hence $\SmallSysNode_c$ is not conservative in this case either. Thus, $\SmallSysNode_c$ can be conservative only if $\|\varphi\|_{H^\infty(\cplus;\Uscr,\Yscr)}=1$.
\end{enumerate}
\end{rem}

In the rest of this subsection, we assume that \eqref{eq:contrconscond} holds, which is true e.g.\ if $\SmallSysNode_c$ is conservative. In this case we can proceed to identify $\Hscr_{c,-1}$ concretely and calculate $A_c|_{\Hscr_c}$ and $B_c$ explicitly. In addition to \eqref{eq:contrconscond}, we make critical use of the characterization \eqref{eq:contrdomAconteq} of $\dom{A_c}$.

Since $\dom{A_c}$ is given by \eqref{eq:contrdomAconteq} and $\Hscr_{c,-1}=(\beta-A_c|_{\Hscr_c})\Hscr_c$, the formula \eqref{Acresol} for the resolvent of $A_c$ suggests the following concrete identification of the extrapolation space:
\begin{equation}   \label{Hc-1}
 \cH_{c,-1}  = \left\{ x : {\mathbb C}^+ \to \cY \bigmid \exists y\in \cY :~ \mu \mapsto \frac{x(\mu) +  \wtsmash \varphi(\mu) y}{ \beta + \mu} \in \cH_c\right\}.
\end{equation}
The assumption \eqref{assume1c} guarantees us that the function $\wtsmash\varphi(\cdot)y$ in \eqref{Hc-1} is uniquely determined by $x$ (whenever at least one such function exists). Note that the choice of $y$ is is in general not unique. Now we set
\begin{equation}   \label{Hc-1norm}
  \| x \|_{\cH_{c,-1}} = \left\| \mu\mapsto \frac{ x(\mu) + \wtsmash \varphi(\mu) y}{\beta + \mu} \right\|_{\cH_c},\quad x\in\cH_{c,-1}.
\end{equation}

We have $\wtsmash\varphi(\cdot)\gamma\in\Hscr_{c,-1}$ with zero norm for all $\gamma\in\Yscr$; simply choose $y:=-\gamma$ in \eqref{Hc-1} and \eqref{Hc-1norm}. Conversely, if $\|x\|_{\Hscr_{c,-1}}=0$, then $\mu\mapsto\frac{x(\mu)+\wtsmash\varphi(\mu)y}{\beta+\mu}=0$ for all $\mu\in\cplus$, i.e., $x(\mu)=-\wtsmash\varphi(\mu)y$ for all $\mu\in\cplus$. Hence, the elements of $\cH_{c,-1}$ are equivalence classes of functions modulo the subspace $\wtsmash \varphi (\cdot) \cY$. These equivalence classes are denoted as $[x]$, where $x$ is any particular representative of the equivalence class. We summarize the properties of the space $\cH_{c,-1}$ as follows:

\begin{thm}\label{T:Hc-1concrete}  
Assume that \eqref{eq:contrconscond} holds. Then the space $\cH_{c,-1}$  given by \eqref{Hc-1} and \eqref{Hc-1norm} is complete and the following claims are true:
\begin{enumerate}
\item The map $\iota  : x \mapsto [x]$ embeds $\cH_c$ into $\cH_{c,-1}$ as a dense subspace. A given element $[z] \in \cH_{c,-1}$ is of the form $\iota (x)$ for some $x \in \cH_c$ if and only if there is a choice of $y$ in $\Yscr$ so that the function $ \displaystyle{\mu \mapsto \frac{z(\mu) + \wtsmash\varphi(\mu) y}{ \alpha + \mu}}$ is not only in $\cH_c$ but also in $\cH_{c,1} = \dom{A_c}$ for some (or equivalently for every) $\alpha\in\cplus$. 

\item Define an operator $A_c|_{\cX_c} : \cH_c \to \cH_{c,-1}$ by
\begin{equation}\label{eq:AcExt}
  A_c|_{\cH_c} x:=  [\mu \mapsto -\mu x(\mu)],\quad x\in\Hscr_c,\,\mu\in\cplus.
\end{equation}
When $\cH_c$ is identified as a linear sub-manifold of $\cH_{c,-1}$ via the embedding map $\iota $ above, then $A_c|_{\cX_c}$ is the unique extension of $A_c : \dom{A_c} \to \cH_c$ to a continuous operator $\Hscr_c\to\Hscr_{c,-1}$. Moreover 
\begin{equation}\label{eq:AcExtRes}
  \big((\beta - A_c|_{\cH_{c}})^{-1}[x]\big)(\mu)=\frac{ x(\mu) + \wtsmash \varphi(\mu) y}{\beta + \mu},\quad [x]\in\Hscr_{c,-1},\,\mu\in\cplus,
\end{equation}
where the condition $\mu\mapsto\frac{ x(\mu) + \wtsmash \varphi(\mu) y}{\beta + \mu}\in\Hscr_c$ uniquely determines $\wtsmash \varphi(\cdot) y$, is unitary from $\cH_{c,-1}$ to $\cH_c$.

\item The action of $B_c  : \cU \to \cH_{c,-1}$ is given by
$$
  B_c u := [\mu\mapsto u],\quad u\in\Uscr,\,\mu\in\cplus.
$$
\end{enumerate}
\end{thm}


\noindent\emph{Proof.} Completeness of $\Hscr_{c,-1}$ is proved precisely the same way as completeness of $\Hscr_{c,-1}^d$ was proved in Theorem \ref{thm:hdc-1concrete}: For a Cauchy sequence $[x_n]$ in $\Hscr_{c,-1}$, denote the limit of the Cauchy sequence $z_n:=\mu\mapsto \frac{ x_n(\mu) + \wtsmash \varphi(\mu) y_n}{\beta + \mu}$ in $\Hscr_c$ by $z$. Then $[x_n]\to[x]$ in $\Hscr_{c,-1}$, where $x(\mu)=(\beta+\mu)z(\mu)$, $\mu\in\cplus$.
\begin{enumerate}
\item It is clear that $\iota(\Hscr_c)\subset\Hscr_{c,-1}$, since by \eqref{Acresol}, 
$$
  \mu\mapsto\frac{x(\mu)+\wtsmash\varphi(\mu)y}{\beta+\mu}=(\beta-A_c)^{-1}x\in\Hscr_{c,1}\subset\Hscr_c
$$ 
with $y=-\tau_{c,\alpha} x$. We prove denseness of $\Hscr_c$ in $\Hscr_{c,-1}$ in the proof of assertion two.

Next assume that the function $g(\mu):=\frac{z(\mu)+\wtsmash\varphi(\mu)y}{\alpha+\mu}$ lies in $\dom{A_c}$ for some $\alpha\in\cplus$. We need to prove that $x:=(\alpha-A_c)g\in\Hscr_c$ has the property $[z] = \iota(x)$. We may use formula \eqref{Ac-id} to compute
\begin{align*}
\big( (\alpha - A_c) g \big) (\mu) & = (\alpha+\mu)g(\mu)+\wtsmash\varphi(\mu)C_cg \\
&= z(\mu) + \wtsmash\varphi(\mu) y + \wtsmash\varphi(\mu) C_c g \\
& = z(\mu) + \wtsmash\varphi(\mu) \big( y + C_c g \big) = x(\mu),
\end{align*}
and we can conclude that $[z] = [x]$, where $x\in\cH_c$. 

The converse implication is seen as follows. Assume that $[z]=[x]$ with $x\in\Hscr_c$ and let $\alpha\in\cplus$ be arbitrary. Then $z(\mu)=x(\mu)+\wtsmash\varphi(\mu)\gamma$ for some $\gamma\in\Yscr$, and by \eqref{Acresol} it holds that
$$
  \big((\alpha-A_c)^{-1}x\big)(\mu) = \frac{x(\mu)-\wtsmash\varphi(\mu)\tau_{c,\alpha} x}{\alpha+\mu}
    = \frac{z(\mu)-\wtsmash\varphi(\mu)(\tau_{c,\alpha} x+\gamma)}{\alpha+\mu}.
$$
Choosing $y:=-\tau_{c,\alpha} x-\gamma$, we thus have that $\mu\mapsto\frac{z(\mu)+\wtsmash\varphi(\mu)y}{\alpha+\mu}\in\dom{A_c}$ for every $\alpha\in\cplus$.

\item If $x \in \cH_c$ and $z(\mu) = -\mu x(\mu)$, then $[\beta x - z ] \in \cH_{c,-1}$ since
\begin{equation}\label{eq:invobsres}
  \mu\mapsto\frac{ \beta x(\mu) - z(\mu)}{ \beta + \mu} =  x\in \cH_c;
\end{equation}
take $y=0$ in \eqref{Hc-1}. Moreover, $\beta x\in\Hscr_c\subset \Xscr_{c,-1}$, and it follows that $[z]=[\mu\mapsto -\mu x(\mu)]\in \Hscr_{c,-1}$. This shows that \eqref{eq:AcExt} defines an operator from all of $\Hscr_c$ into $\Hscr_{c,-1}$. If it happens that $x\in\dom{A_c}$ then $A_c|_{\Hscr_c}x = [A_cx]$ by \eqref{Ac-id}, and hence $A_c|_{\Hscr_c}$ is an extension of $A_c$. 

The operator in \eqref{eq:AcExtRes} maps $\Hscr_{c,-1}$ into $\Hscr_c$ and it equals $(\beta-A_c|_{\Hscr_c})^{-1}$, because \eqref{eq:AcExt} gives
$$
  \big[(\beta-A_c|_{\Hscr_c})^{-1}[x]\big]=\left[ \mu\mapsto\frac{x(\mu)}{\beta+\mu} \right],\quad\mu\in\cplus,
$$
and as $[x]\in\Hscr_{c,-1}$, there by \eqref{Hc-1} exists a $y\in\Yscr$ such that $\mu\mapsto\frac{x(\mu)+\wtsmash\varphi(\mu)y}{\beta+\mu}\in\Hscr_c$. By \eqref{Hc-1}, \eqref{Hc-1norm}, and \eqref{eq:AcExtRes}, $(\beta-A_c|_{\Hscr_c})^{-1}$ maps $\Hscr_{c,-1}$ isometrically into $\Hscr_c$. On the other hand, in the completeness proof, we showed how one for an arbitrary $z\in\Hscr_c$ can give an $[x]\in\Hscr_{c,-1}$, such that $(\beta-A_c|_{\Hscr_c})^{-1}[x]=z$, and thus $(\beta-A_c|_{\Hscr_c})^{-1}$ is onto $\Hscr_c$.

\item Combining \eqref{eq:isomeval} with the formula for $A_c|_{\Hscr_c}$, we see that for arbitrary $\alpha\in\cplus$:
\begin{align*}
B_c u & = (\alpha - A_c|_{\cH_c}) e_c(\overline\alpha)^* u \\
& =\left[ \mu \mapsto (\alpha + \mu) \frac{1 - \wtsmash\varphi(\mu) \varphi(\alpha)}{ \mu + \alpha} u \right]=[\mu\mapsto u],\quad u\in\Uscr,\,\mu\in\cplus. \tag*{\qed}
\end{align*}
\end{enumerate}

It follows from Assertion 3 in Theorem \ref{T:Hc-1concrete} that all constant functions are in $\Hscr_{c,-1}$, but here they have non-zero norm in general. This can actually be seen directly in \eqref{Hc-1}, by choosing $y:=\varphi(\beta)u$; then the function in \eqref{Hc-1} is $e_c(\overline\beta)^*u$, and by \eqref{Hc-1norm} $\|[u]\|_{\Hscr_{c,-1}}=\left\|e_c(\overline\beta)^*u\right\|_{\Hscr_c}\neq0$, unless $\Hscr_c=\zero$. Again we can recover $\SmallSysNode_c$ from $A_c$, $B_c$, $C_c$, and $\varphi$ evaluated at some point in $\cplus$. 

So far we only know the resolvent of $A_c|_{\Hscr_c}$ at the single point $\beta$ corresponding to the rigging. Based on \eqref{Acresol} and \eqref{eq:AcExtRes}, it seems plausible to guess that for other $\alpha\in\cplus$ the resolvent at $\alpha$ would be
\begin{equation}\label{eq:AcExtRes2}
  \big((\alpha - A_c|_{\cH_{c}})^{-1}[x]\big)(\mu)=\frac{ x(\mu) + \wtsmash \varphi(\mu) \gamma}{\alpha + \mu},
    \quad [x]\in\Hscr_{c,-1},\,\alpha,\mu\in\cplus,
\end{equation}
for some $\gamma\in\Yscr$, which depends on $[x]$ and $\alpha$.

\begin{prop}
Assume that \eqref{eq:contrconscond} holds. For $[x]\in\Hscr_{c,-1}$, let $y\in\Yscr$ be such that $\mu\mapsto\frac{x(\mu)+\wtsmash\varphi(\mu)y}{\beta+\mu}\in\Hscr_c$. Then \eqref{eq:AcExtRes2} holds with the choice 
$$
  \gamma:=y+(\beta-\alpha)\tau_{c,\alpha} \frac{x(\cdot)+\wtsmash\varphi(\cdot)y}{\beta+\cdot}.
$$
\end{prop}
\begin{proof}
We can use the resolvent formula
$$
  (\alpha - A_c|_{\cH_{c}})^{-1}=(\beta - A_c|_{\cH_{c}})^{-1}+(\beta-\alpha)(\alpha - A_c)^{-1}(\beta - A_c|_{\cH_{c}})^{-1}
$$
together with \eqref{Acresol} and \eqref{eq:AcExtRes} to calculate (for $\mu\in\cplus$):
\begin{equation}\label{eq:AcExtRes3}
\begin{aligned}
  \big((\alpha - A_c|_{\cH_{c}})^{-1}[x]\big)(\mu) &= 
  \frac{x(\mu)+\wtsmash\varphi(\mu)y}{\beta+\mu} \\
    &\quad + (\beta-\alpha) \frac{\frac{x(\mu)+\wtsmash\varphi(\mu)y}{\beta+\mu}-\wtsmash\varphi(\mu)\tau_{c,\alpha}\frac{x(\cdot)+\wtsmash\varphi(\cdot)y}{\beta+\cdot}}{\alpha+\mu}.
\end{aligned}
\end{equation}
Straightforward simplifications show that \eqref{eq:AcExtRes3} minus \eqref{eq:AcExtRes2} equals
\begin{equation}\label{eq:AcExtResDiff}
  \frac{\wtsmash\varphi(\mu)}{\alpha+\mu}\left( y-\gamma+(\beta-\alpha)\tau_{c,\alpha} \frac{x(\cdot)+\wtsmash\varphi(\cdot)y}{\beta+\cdot} \right),
\end{equation}
which proves the claim.
\end{proof}

We illustrate the discussion in this subsection using the case of a constant $\varphi$.

\begin{ex}\label{ex:constantc}
Assume that $\varphi(\mu)=D_c$ for all $\mu\in\cplus$, so that $\varphi\in\Sscr(\cplus;\Uscr,\Yscr)$ if and only if $\|D_c\|\leq1$. Moreover, it follows from \eqref{taualphacadj} that $\tau_{c,\alpha}=0$ and $C_c=0$, and Corollary \ref{cor:ContrRes} then yields that $(\alpha-A_c)^{-1}x = \mu\mapsto x(\mu)/(\alpha+\mu)$, $\mu\in\cplus$, and by Theorem \ref{thm:enprescontainedadjoint}, $(A_cx)(\mu)= -\mu x(\mu)$, $\mu\in\cplus$, for $x\in\dom{A_c}$, where
$$
\begin{aligned}
  \dom{A_c} &= (\alpha-A_c)^{-1}\Hscr_c = \set{\mu\mapsto \frac{z(\mu)}{\alpha+\mu}\bigmid z\in\Hscr_c,\,\mu\in\cplus}.
\end{aligned}
$$ 
Here $\alpha\in\cplus$ is arbitrary, and by Definition \ref{def:altsysnode}, we also have
$$
  \dom{A_c}=\set{x\in\Hscr\mid A_cx\in\Hscr_c} = \set{x\in\Hscr_c\mid \mu\mapsto \mu x(\mu)\in\Hscr_c},
$$
so that \eqref{eq:contrdomAconteq} holds with the additional simplification that we only need to consider $y=0$.

Now Theorem \ref{thm:contrsemiexplicit} gives (for some arbitrary $\alpha\in\cplus$)
\begin{equation}\label{eq:constcontreal}
\begin{aligned}
  \dom{\SmallSysNode_c} &= \\\set{\bbm{x\\u}\in\bbm{\Hscr_c\\\Uscr} \biggmid \frac{(\alpha+\mu)x(\mu)-(1-D_c^*D_c)u}{\alpha+\mu} \in \dom{A_c} } &= \\
  \set{\bbm{x\\u}\in\bbm{\Hscr_c\\\Uscr}\bigmid \mu\mapsto -\mu x(\mu)+(1-D_c^*D_c)u\in\Hscr_c}, \quad&\text{and}\\
  \SysNode_c\bbm{x\\u} = \bbm{\mu\mapsto -\mu x(\mu)+(1-D_c^*D_c)u\in\Hscr_c\\D_cu}.
\end{aligned}
\end{equation}
Note that the arbitrary parameter $\alpha\in\cplus$ in Theorem \ref{thm:contrsemiexplicit} does not appear here. 

In \eqref{eq:constcontreal} the state part is purely for energy accounting, since the output is independent of the current state. If it happens that $D_c$ is isometric, then the energy is preserved without any state needing to absorb energy, and indeed $\Hscr_c=\zero$ as can easily be seen from the reproducing kernel $K_c$ of $\Hscr_c$. In this case the realization consists only of the static operator $D_c$. If $D_c$ is not isometric, then the function $B_cu = \mu\mapsto (1-D_c^*D_c)u$, $\mu\in\cplus$, never lies in $\Hscr_c$ unless it is zero. Thus $B_c$ is \emph{strictly unbounded} (in the terminology of \cite{MaSt06}), and it is interesting that both $A_c$ and $B_c$ are unbounded even though the transfer function $\varphi$ is rational, even constant.

We make the following observations on the dual system node $\SmallSysNode_c^*$: Due to Theorem \ref{thm:contrdual}, the first equality holds in
\begin{equation}\label{eq:constdomadj}
\begin{aligned}
  \dom{\SmallSysNode_c^*} &=\bigg\{\bbm{x\\y}\in\bbm{\Hscr_c\\\Yscr} \bigmid \\
    & \qquad \exists u\in\Uscr:~ \mu\mapsto \mu x(\mu)+D_c^*y-u\in\Hscr_c \bigg\} \\
   &= \bbm{\dom{A_c^*}\\\Yscr},
\end{aligned}
\end{equation}
where as usual
$$
  \dom{A_c^*}=\set{x\in\Hscr_c \mid \exists u\in\Uscr:~ \mu\mapsto \mu x(\mu) -u\in\Hscr_c }.
$$
The second equality in \eqref{eq:constdomadj} is seen as follows: for all $x\in\Hscr_c$ and $y\in\Yscr$, it holds that
$$
  \exists u\in\Uscr:~ \mu\mapsto \mu x(\mu)+D_c^*y-u\in\Hscr_c \quad\Longleftrightarrow\quad 
  \exists v\in\Uscr:~ \mu\mapsto \mu x(\mu)-v\in\Hscr_c,
$$
where we can always take $v:=u-D_c^*y$ and $u:=D_c^*y+v$, depending on in which direction we traverse the equivalence. It is a rare convenience that the domain of a system node decomposes in this way; it is for instance not the case for $\SmallSysNode_c$ itself.

The fact that $C_c^*=0$ can be seen in \eqref{Cc*con} as 
$$
  \|C_c^*y\|_{\Hscr_{c,-1}^d}=\left\|[\mu\mapsto D_c^*y]\right\|_{\Hscr_{c,-1}^d}=0.
$$
A consequence of $C_c^*=0$ is that $\bbm{A_c^*\&C_c^*}=\bbm{A_c^*&0}$, which agrees also with \eqref{defz}--\eqref{defu}, since $\lim_{\eta\to\infty}\wtsmash\varphi(\eta)y=\wtsmash\varphi(\mu)y$, $\mu\in\cplus$, and these terms cancel in \eqref{defz}.

Due to Proposition \ref{prop:xtendstozero}, we have that $\wtsmash\varphi(\cdot)y=\mu\mapsto D_c^*y\in\Hscr_c$ only if $D_c^*y=0$, and thus \eqref{assume1c} holds. The implication \eqref{eq:transfnotinHc}, on the other hand, holds if and only if $D_c^*$ is injective, i.e., $D_c$ has range dense in $\Yscr$. By Theorem \ref{thm:conschar}, $D_c$ has dense range and \eqref{eq:contrconscond} holds if and only if $\SmallSysNode_c$ in \eqref{eq:constcontreal} is conservative. We next construct an example where \eqref{eq:contrconscond} does not hold although we already established that \eqref{assume1c} holds; see Lemma \ref{lem:domequivdeluxe}.2.

Let $D_c:\Uscr\to\Yscr$ be an arbitrary contraction with dense, but not closed, range and let $\SmallSysNode_c$ denote the controllable model for $\varphi(\mu)=D_c$, $\mu\in\cplus$. Choose some $y\in\Yscr$ with $y\not\in\range{D_c}$ and let $x\in\Hscr_c$ and $u\in\Uscr$ be such that \eqref{eq:funinHc} holds. Then $\sbm{x\\y}\in\dom{\SmallSysNode_c^*}$ by Theorem \ref{thm:contrdual}, but for all $\sbm{x\\u}\in\dom{\SmallSysNode_c}$, $\bbm{C_c\&D_c}\sbm{x\\u}=D_cu\neq y$ by construction, and thus $\SmallSysNode_c$ is not conservative according to Theorem \ref{thm:msw}.2. Since \eqref{eq:transfnotinHc2} holds, the condition \eqref{eq:contrconscond} must be false.
\end{ex}

\subsection{Reproducing kernels of the rigged spaces}\label{sec:repkerncontr}
 
We finish our study of the controllable realization by calculating the reproducing kernels associated to the rigged spaces. In this subsection we do not make any additional assumptions, such as \eqref{eq:contrconscond}, unless otherwise indicated.

Recall that the state space $\cH_c$ for the system node $\sbm{ A \& B \\ C \& D}_c$ is a reproducing kernel Hilbert space with reproducing kernel $K_c$ as in \eqref{eq:isomkernel}.  By construction the 1-scaled rigged space $\cH^d_{c,1}$ also consists of $\cU$-valued functions and it is not difficult to see that the point-evaluation map $e^d_{c,1}(\mu)$ is bounded in $\cH^d_{c,1}$-norm and hence $\cH^d_{c,1}$ is also a reproducing kernel Hilbert space. The same is true for $\Hscr_{c,1}$.

Technically the $(-1)$-scaled rigged spaces $\cH^d_{c,-1}$ and $\Hscr_{c,-1}$ are not reproducing kernel Hilbert spaces since they consist of equivalence classes of functions rather than of functions.  However, while point-evaluation is not well-defined on $\cH^d_{c,-1}$, the map $[z] \mapsto z(\mu) - z(\overline\alpha)$ is well-defined for all fixed $\alpha\in\cplus$; this amounts to evaluating the unique member $z$  of the equivalence class 
 $[z]$ normalized to satisfy $z(\overline\alpha) = 0$. Here it is most convenient to choose $\overline\alpha=\overline\beta$, the parameter used to define the $\cH^d_{c,1}$ and $\cH^d_{c,-1}$ norms. 

If we define $e^d_{c,-1}(\mu)  : \cH^d_{c,-1} \to \cU$ by
 $$
   e^d_{c,-1}(\mu) [z] := z(\mu) - z(\overline\beta),\quad [z]\in\Hscr_{c,-1}^d,\, \mu\in\cplus,
 $$
then $e^d_{c,-1}(\mu)$ is also bounded for every $\mu\in\cplus$. More precisely, $\|e^d_{c,-1}\| \leq |\overline\beta-\mu|\,\|e_c(\mu)\|$, since we by \eqref{resA*ext} have
$$
\begin{aligned}
  \left\| e_{c,-1}^d(\mu)[z] \right\|_\Uscr &= \|z(\mu)-z(\overline\beta)\|_\Uscr 
  = \left\| e_c(\mu)(\overline\beta-\mu)\big((\overline\beta-A_c^*|_{\Hscr_c})^{-1}[z]\big) \right\|_\Uscr \\
  &\leq \|e_c(\mu)\| \, |\overline\beta-\mu|\left\|[z]\right\|_{\Hscr_{c,-1}^d},\quad [z]\in\Hscr_{c,-1}^d.
\end{aligned}
$$
We may then consider the function
 $$
 K^d_{c,-1}(\mu, \lambda) =  e^d_{c,-1}(\mu) \big( e^d_{c,-1}(\lambda) \big)^*,\quad\mu,\lambda\in\cplus,
 $$
 to be the reproducing kernel for $\cH^d_{c,-1}$.

Suppose that $x \in \cH_{c,-1}$.  With assumption \eqref{eq:contrconscond} in force, there 
is a unique choice $y_{x}$ of vector $y$ in $\cY$ so that $\mu 
\mapsto \frac{x(\mu) + \widetilde \varphi(\mu) y_x }{ \beta + \mu} \in 
\cH_{c}$.  The space $\cH_{c,-1}$ is defined as equivalence classes 
whereby $[x] = [x']$ means that $ x - x' =   \mu \mapsto \widetilde 
\varphi(\mu)y$, for some choice of $y\in\Yscr$.  To define a point evaluation $e_{c,-1}(\lambda)$ on 
$\cH_{c,-1}$, we have to choose a canonical representative of each 
equivalence class.  We choose as our canonical representative the function 
$\mu \mapsto x(\mu) + \widetilde \varphi(\mu) y_{x}.$  Thus we define 
$e_{c,-1}(\lambda) \colon \cH_{c,-1} \to \cU$ by
$$
  e_{c,-1}(\lambda)  \colon [x] \mapsto x(\lambda) + \widetilde 
  \varphi(\lambda) y_{x},\quad \lambda\in\cplus,
$$
and consider $K_{c,-1}(\mu,\lambda):=e_{c,-1}(\mu)\big(e_{c,-1}(\lambda)\big)^*$, $\mu,\lambda\in\cplus$, to be the reproducing kernel of $\Hscr_{c,-1}$.

\begin{prop}  \label{prop:repkerform-c}
We have the following formulas for the kernel functions associated with the Hilbert spaces $\cH^d_{c,\pm 1}$ and $\cH_{c,\pm1}$ (for $\mu,\lambda\in\cplus$):
 \begin{align}
 K_c(\mu, \lambda)  & = B_c^* (\mu - A_c^*)^{-1} (\overline{\lambda} - A_c|_{\cH_c})^{-1} B_c, \label{kerfunct-Hc} \\ 
 K^d_{c,1}(\mu, \lambda)  &= B_c^* (\mu - A_c^*)^{-1} (\overline\beta - A_c^*)^{-1} \notag \\&\qquad\times (\beta - A_c)^{-1}(\overline{\lambda} - A_c|_{\cH_c})^{-1} B_c, \label{kerfunct-Hc1d} 
\end{align}
\begin{align}
 K^d_{c,-1}(\mu, \lambda)  & = (\overline\beta - \mu) (\beta
    - \overline{\lambda})  B_{c}^{*} (\mu - A_{c}^{*})^{-1} 
    (\overline{\lambda} - A_{c}|_{\cH_{c}})^{-1} B_{c}, \label{kerfunct-Hcminus1d} \\
 K_{c,1}(\mu, \lambda)  &= B_{c}^{*} (\mu - A_{c}^*)^{-1} (\beta - A_{c})^{-1}  \notag \\
 &\qquad \times(\overline{\beta} - A_{c}^{*})^{-1} 
    (\overline{\lambda} - A_{c}|_{\cH_{c}})^{-1} B_{c}, \label{kerfunct-Hc1}\\
 K_{c,-1}(\mu, \lambda)  & = (\beta + \mu)(\overline\beta
    + \overline{\lambda})  B_{c}^{*} (\mu - A_{c}^{*})^{-1} 
    (\overline{\lambda} - A_{c}|_{\cH_{c}})^{-1} B_{c}. \label{kerfunct-Hcminus1}
  \end{align}
The above state-space formulas correspond to the following purely function-theoretic formulas (for $\mu,\lambda\in\cplus$):
 \begin{align}
 K_c(\mu, \lambda)  & = \frac{1 - \varphi(\overline\mu)^* \varphi(\overline\lambda)}{ \mu + \overline{\lambda}}, \notag \\ 
K^d_{c,1}(\mu, \lambda)  &= \frac{K_c(\mu,\lambda) - K_c(\mu,\overline\beta)}{(\overline\beta-\mu)(\beta - \overline{\lambda})}-\frac{K_c(\overline\beta,\lambda) - K_c(\overline\beta,\overline\beta)}{(\overline\beta-\mu)(\beta - \overline{\lambda})} \label{kerfunct-Hc1d2}, \\ 
 K^d_{c,-1}(\mu, \lambda)  & = (\overline\beta - \mu)(\beta - \overline{\lambda})K_c(\mu, \lambda),
  \label{kerfunct-Hcminus1d2} \\
 K_{c,1}(\mu,\lambda) & = \frac{\kappa_c(\mu,\lambda)-\varphi(\overline\mu)^*\tau_{c,\beta}\,\big(\kappa_c(\cdot,\lambda)\big)}{\beta+\mu},\quad\text{where} \label{kerfunct-Hc12}\\
  &\qquad \kappa_c(\mu,\lambda) := \frac{K_c(\mu,\lambda)-K_c(\overline\beta,\lambda)}{\overline\beta-\mu}, \notag \\
K_{c,-1}(\mu, \lambda)  & = (\beta + \mu) (\overline\beta
    + \overline{\lambda})  K_c(\mu,\lambda). \label{kerfunct-Hcminus12}
  \end{align}
  Here the point $\beta \in {\mathbb C}^+$ appearing in the formulas must be chosen to be the same $\beta$ which was used in the rigging $\Hscr_{c,1}\subset\Hscr_c\subset \Hscr_{c,-1}$, so that $\overline\beta$ corresponds to the rigging $\Hscr^d_{c,1}\subset\Hscr_c\subset \Hscr^d_{c,-1}$.
The formulas \eqref{kerfunct-Hcminus1} and \eqref{kerfunct-Hcminus12} depend on Theorem \ref{T:Hc-1concrete} and hence they are established only for the case when \eqref{eq:contrconscond} holds.
 \end{prop}
 
 \begin{proof}  

To see \eqref{kerfunct-Hc}, combine part 3 of Theorem \ref{T:RK} with formula \eqref{eq:isomeval}:
 $$
 K_c(\mu, \lambda) = e_c(\mu) e_c(\lambda)^* =\left( B_c^* (\mu - A_c^*|_{\cH_c})^{-1} \right) \left( (\overline{\lambda} - A_c|_{\cH_c})^{-1} B_c\right).
 $$
 
As for \eqref{kerfunct-Hc1d},  we use that $(\overline\beta - A_c^*)^{-1}$ is a unitary transformation from $\cH_c$ to $\cH^d_{c,1}$.  Hence any $f \in \cH^d_{c,1}$ has the form $f = (\overline\beta - A_c^*)^{-1} g$ for a unique $g \in \cH_c$, and for $\lambda \in {\mathbb C}_+$ we can compute
\begin{align*}
\langle f(\lambda), u \rangle_\cU & = \langle \left( (\overline\beta - A_c^*)^{-1} g \right)(\lambda), u \rangle_\cU \\
 &= \langle (\overline\beta - A_c^*)^{-1} g, e_c(\lambda)^* u \rangle_{\cH_c} \\
 &= \langle g, (\beta - A_c)^{-1} e_c(\lambda)^* u \rangle_{\cH_c} \\
 &= \langle f, (\overline\beta-A_c^*)^{-1}(\beta - A_c)^{-1} e_c(\lambda)^* u \rangle_{\cH_{c,1}^d}.
\end{align*}
Combining this with the fact that the point-evaluation operator in $\Hscr_{c,1}^d$ is the restriction of $e_c(\cdot)$ to $\Hscr_{c,1}^d$, we obtain \eqref{kerfunct-Hc1d}. In order to get \eqref{kerfunct-Hc1d2}, we continue calculating
 $$
\begin{aligned}
 e_c(\mu)(\overline\beta-A_c^*)^{-1}(\beta - A_c)^{-1} e_c(\lambda)^* u &= e_c(\mu)(\overline\beta - A_c^*)^{-1}\frac{e_c(\lambda)^* - e_c(\overline\beta)^*}{\beta - \overline{\lambda}} u \\
  & =\frac{\frac{K_c(\mu,\lambda) - K_c(\mu,\overline\beta)}{\beta - \overline{\lambda}}-\frac{K_c(\overline\beta,\lambda) - K_c(\overline\beta,\overline\beta)}{\beta - \overline{\lambda}}}{\overline\beta-\mu} u,
\end{aligned}
$$
where we used \eqref{eq:isomres} and \eqref{Ac*res} in the first and second equalities, respectively.

The formula \eqref{kerfunct-Hc1} follows immediately on replacing $\overline\beta-A_c^*$ by $\beta-A_c$ in \eqref{kerfunct-Hc1d}, since $(\beta-A_c)^{-1}$ is the appropriate unitary operator from $\Hscr_c$ to $\Hscr_{c,1}$ instead of $(\overline\beta-A_c^*)^{-1}$. Using \eqref{Ac*res} and \eqref{Acresol}, we obtain \eqref{kerfunct-Hc12}. 

In order to establish \eqref{kerfunct-Hcminus1d},  we use that $(\overline\beta - A_c^*|_{\Hscr_c})$ is a unitary transformation from $\cH_c$  to $\cH^d_{c,-1}$.  Thus any $[f] \in \cH^d_{c,-1}$ has the form $[f] = (\overline\beta - A_c^*|_{\Hscr_c}) g$ with $g \in \cH_c$.  Furthermore, from \eqref{eq:alphaAdcalc} we have that the unique representative $f$ for the equivalence class $[f]$ satisfying $f(\overline\beta) = 0$ is given by $f(\mu) = (\overline\beta - \mu) g(\mu)$.  Hence, we have, for $\lambda \in {\mathbb C}^+$,
\begin{equation}\label{eq:Hc-1dkerncalc}
\begin{aligned}
\langle f(\lambda), u \rangle_\cU & = \langle (\overline\beta - \lambda) g(\lambda), u \rangle_\cU \\
& = \langle g, (\beta - \overline{\lambda}) e_c(\lambda)^* u \rangle_{\cH_c} \\
& =  \langle f, (\beta - \overline{\lambda}) (\overline\beta - A_c^*|_{\Hscr_c}) e_c(\lambda)^* u \rangle_{\cH^d_{c,-1}}.
\end{aligned}
\end{equation}

This proves that $\big(e^d_{c,-1}(\lambda)\big)^*u=(\beta - \overline{\lambda}) (\overline\beta - A_c^*|_{\Hscr_c}) e_c(\lambda)^* u$, and combining this with \eqref{Ac*-ext} we obtain \eqref{kerfunct-Hcminus1d2}, again choosing the representative with value zero at $\overline\beta$. Now \eqref{kerfunct-Hcminus1d} follows from \eqref{kerfunct-Hcminus1d2} and \eqref{kerfunct-Hc}. In order to obtain \eqref{kerfunct-Hcminus12} and \eqref{kerfunct-Hcminus1} we compute
\begin{align*}
    \langle [x], e_{c,-1}(\lambda)^* u \rangle_{\cH_{c,-1}} & =
    \langle e_{c,-1}(\lambda) [x], u \rangle_{\cU} = 
    \langle x(\lambda) + \widetilde \varphi(\lambda) y_{x}, u \rangle_{\cU} \\
    &= \left\langle \frac{ x(\lambda) + \widetilde \varphi(\lambda) y_{x}}{ \beta + \lambda}, (\overline{\beta} + \overline{\lambda}) u \right \rangle_{\cU},
\end{align*}
and using \eqref{Hc-1norm}, we further obtain that this equals
\begin{align*}
    \left \langle \mu \mapsto \frac{ x(\mu) + \widetilde \varphi(\mu) 
    y_{x}}{\beta + \mu}, \mu \mapsto K_{c}(\mu, \lambda) 
    (\overline{\beta} + \overline{\lambda}) u \right\rangle_{\cH_{c}} & = \\
    \left\langle \mu \mapsto x(\mu) + \widetilde \varphi(\mu) y_{x},\,
    \mu \mapsto (\beta + \mu) (\overline{\beta} + \overline{\lambda}) 
    K_{c}(\mu, \lambda) \right\rangle_{\cH_{c,-1}}.
\end{align*}
We conclude that 
$$
  e_{c,-1}(\lambda)^* = [\mu\mapsto(\beta + \mu) 
(\overline{\beta} + \overline{\lambda}) K_{c}(\mu, \lambda)],
$$
and since $\mu\mapsto(\overline{\beta} + \overline{\lambda}) K_{c}(\mu, \lambda)\in\Hscr_c$, we have $K_{c,-1}(\mu,\lambda)=(\beta + \mu) (\overline{\beta} + \overline{\lambda}) K_{c}(\mu, \lambda)$.
\end{proof}

The following result follows from \eqref{kerfunct-Hc1d}, \eqref{kerfunct-Hc1}, \eqref{Hdc-1con-norm}, and \eqref{Hc-1norm}. Note that the linear span of kernel functions in $\Hscr_c$ is embedded as dense subspaces of $\Hscr_{c,-1}$ and $\Hscr_{c,-1}^d$. 

\begin{cor}\label{cor:kerneldual}
For all $\mu,\lambda\in\cplus$ and $u,v\in\Uscr$, we have
$$
\begin{aligned}
  \Ipdp{K_{c,1}^d(\mu,\lambda) u}{v}_\Uscr&=\Ipdp{K_c(\cdot,\lambda)u}{K_c(\cdot,\mu)v}_{\Hscr_{c,-1}}\quad\text{and} \\
  \Ipdp{K_{c,1}(\mu,\lambda) u}{v}_\Uscr&=\Ipdp{K_c(\cdot,\lambda)u}{K_c(\cdot,\mu)v}_{\Hscr_{c,-1}^d}.
\end{aligned}
$$
\end{cor}

We now turn our attention to the observable functional model.

\section{The observable co-energy-preserving model} \label{sec:coenpres}

In this section we present the observable co-energy-preserving model realization of a given $\varphi\in\cS({\mathbb C}^+; \cU, \cY)$ which uses the reproducing kernel Hilbert space $\cH_o$ as state space.  
We have already seen in Section \ref{sec:statespaces} that $K_o(\mu \lambda) = \frac{ 1 -\varphi(\mu) \varphi(\lambda)^*}{ \mu + \overline{\lambda}}$ is a positive kernel with associated reproducing kernel Hilbert space denoted as $\cH_o$.  Rather than defining the system node 
$\left[ \begin{smallmatrix} A \& B \\ C \& D \end{smallmatrix}\right]_o$ directly, it is more natural to first define its adjoint $\left[ \begin{smallmatrix} A \& B \\ C \& D \end{smallmatrix}\right]_o^*$.  The adjoint is first defined via the mapping
\begin{equation}   \label{mapping-o}
\SysNode_o^* : \bbm{e_o(\overline\lambda)^* y \\ y } \mapsto \bbm{\lambda e_o(\overline\lambda)^* y \\ \wtsmash\varphi(\lambda) y }, \quad y \in \cY,\, \lambda \in {\mathbb C}^+,
\end{equation}
cf.\ \eqref{eq:SodDefprel}. One can then mimic the proof of Lemma \ref{lem:contrclosable}  to see that this formula is well defined and can be extended uniquely to a well-defined closed operator  $\SmallSysNode_o^*$.  One can then mimic the whole development of Section \ref{sec:enpres} to arrive at the sought-after results for the observable co-energy-preserving case here.

A logically more efficient (if not as intuitively appealing) procedure is to reduce the results for the observable co-energy-preserving case to those of Section \ref{sec:enpres} for the controllable energy-preserving case by the following duality trick.
As was noted in Proposition \ref{prop:dualsysnode}, if $\varphi(\mu)$ is the transfer function of the system node $\sbm{ A \& B \\ C \& C }$, then 
$\sbm{ A^d \& B^d \\ C^d \& D^d }: = \sbm{ A \& B \\ C\& D}^*$ is a system node with transfer function equal to $\wtsmash \varphi(\mu) = \varphi(\overline{\mu})^*$.  Observe that the transformation $\varphi \mapsto \wtsmash \varphi$ maps the Schur class $\cS({\mathbb C}^+; \cU, \cY)$ bijectively to the Schur class $\cS({\mathbb C}^+; \cY, \cU)$.  Given a Schur-class function $\varphi \in \cS({\mathbb C}^+; \cU, \cY)$, let $\sbm{ A \& B \\ C \& D}_c^\sim$ be the controllable energy-preserving canonical functional-model system node constructed as in Section \ref{sec:enpres} but associated with $\wtsmash \varphi$ rather than with $\varphi$.  Then 
its adjoint
$$
   \left( \SysNode_c^\sim \right)^* =: \bbm{ \wtsmash A_c^d \& \wtsmash B_c^d \\ \wtsmash C_c^d \& \wtsmash D_c^d }
$$
is also a system node which has transfer function $(\wtsmash \varphi)^\sim = \varphi$.  Since $\SmallSysNode_c^\sim$ is controllable and energy-preserving by construction, as was observed in Theorem \ref{thm:firstmodel}, it follows that $\sbm{ \wtsmash A^d \& \wtsmash B_c^d \\ \wtsmash C_d^d \& \wtsmash D^d_c}$
is observable and co-energy-preserving.  One can then define the associated {\em observable, co-energy-preserving canonical functional-model system node associated with $\varphi$}  by
$$
  \SysNode_o = \bbm{ \wtsmash A^d_c \& \wtsmash B_c^d \\ \wtsmash C^d_c \& \wtsmash D^d_c }.
$$

Note that every result concerning the controllable energy-preserving canonical functional-model system node $\sbm{ A \& B \\ C \& D}_c$ obtained in Section \ref{sec:enpres} carries over to a corresponding result for $\sbm{ A \& B \\ C \& D}_o$:  simply apply the result from Section \ref{sec:enpres} but with $\wtsmash \varphi$ in place of $\varphi$ and then express the result in terms of operators associated with the adjoint system node $\sbm{ A \& B \\ C \& D }_o = \left( \sbm{ A \& B \\ C \& D}_c^\sim \right)^*$ rather than with $\sbm{ A \& B \\ C \& D}_c^\sim$ itself. In this section we state the most interesting results but leave all of the proofs to the reader.  The reader is invited to supply the proofs by either of the two routes sketched above.

The following result is the observable, co-energy-preserving analogue of Lemma \ref{lem:contrclosable}, Theorem \ref{thm:firstmodel}, and Theorem \ref{thm:contrdual} combined.

\begin{thm}\label{thm:coisom}  
Suppose that we are given a function $\varphi \in \cS({\mathbb C}_+; \cU, \cY)$ and define $\cH_o = \cH(K_o)$ as above.
\begin{enumerate}
\item The mapping  \eqref{mapping-o} which was
defined initially only on elements of the form $\sbm{e_o(\lambda)^* y \\ y } \in \sbm{ \cH_o \\ \cY}$ extends by linearity and limit-closure uniquely to a closed linear operator
\begin{equation}   \label{coisomSdef*}
 \SysNode_o^* : \bbm{ \cH_o \\ \cY } \supset \dom{  \SmallSysNode_o^*} \to \bbm{ \cH_o \\ \cU}
 \end{equation}
 which is a controllable, energy-preserving system node having transfer function equal to $\wtsmash \varphi(\mu)= \varphi(\overline{\mu})^*$, $\mu \in {\mathbb C}^+$.
 
 \item The adjoint of the system node $\SmallSysNode_o^*$ given by \eqref{coisomSdef*}, namely  
\begin{equation}\label{coisomSdef'}
 \SysNode_o  :\bbm{\Hscr_o\\\Uscr}\supset\dom{ \SysNode_o} \to \bbm{\Hscr_o\\\Yscr},
\end{equation}
 is an observable, co-energy-preserving system node with transfer function equal to $\varphi$.  

  \item The system node \eqref{coisomSdef'} can be characterized more directly as follows:
 \begin{align}
  \SysNode_o  &: \bbm{x\\u}\mapsto\bbm{z\\y},\quad\text{where} \label{eq:coisomSdef} \\
  z(\mu) &:= \mu x(\mu)+\varphi(\mu)u-y,\quad \mu\in\cplus,\quad\text{and} \label{defz'} \\
  y  &:= \lim_{\re\eta\to\infty} \eta x(\eta)+\varphi(\eta)u,\quad\text{defined on} \label{defy'} \\
  \dom{ \SysNode_o }  &:= \left\{  \bbm{x\\u}\in\bbm{\Hscr_o\\\Uscr}\bigmid 
  \exists y\in\Yscr:~z\in\Hscr_o ~\text{in \eqref{defz'}}\right\}. \notag
  \end{align}
  For every $\sbm{ x \\ u } \in \dom{ \SmallSysNode_o}$, the $y \in \cY$ such that $z$ given in \eqref{defz'} lies in $\cH_o$ is unique and it is given by \eqref{defy'}.

 \item The kernel functions $K_o(\cdot,\lambda)=e_o(\lambda)^*$, $\lambda \in {\mathbb C}^+$, for the space $\cH_o$ are given by
$$
   e_o(\lambda)^* = (\overline{\lambda} - A_o^*|_{\cH_o})^{-1} C_o^*, \quad \lambda\in\cplus,
$$
 \end{enumerate}
\end{thm}

Remark \ref{rem:leftshift} applies with minor modifications to $\SmallSysNode_o$. The following result (the analogue of Theorem \ref{thm:firstsimilar}) explains the canonical  property for the functional-model system node $\sbm{ A \& B \\ C \& D}_o$.

\begin{thm}   \label{thm:secondsimilar}
Let $\sbm{ A \& B \\ C \& D}$ be an observable and co-energy-preserving realization of $\varphi$ with state space $\cX$.  Define the operator $\Delta:\Hscr_o\to\Xscr$ as the unique bounded linear extension of the mapping
$$
  \Delta : e_o(\lambda)^* y \mapsto \left(\overline{\lambda}-A^*|_\cX\right)^{-1} C^* y, \quad \lambda \in {\mathbb C}^+, \, y \in \cY.
 $$
Then $\Delta$ is unitary from $\cH_o$ to $\cX$, the operator $\sbm{\Delta&0\\0&1_\Uscr}$ maps $\dom{\SmallSysNode_o}$ one-to-one onto $\dom{\SmallSysNode}$, and
 $$
  \bbm{A \& B \\ C \& D} \bbm{ \Delta & 0 \\ 0 & 1_\cU} = \bbm{ \Delta & 0 \\ 0 & 1_\cY} \SysNode_o.
 $$
 \end{thm}

The proof is simply an application of Theorem \ref{thm:firstsimilar} to $\SmallSysNode^*$ and $\SmallSysNode_o^*$.

\subsection{Explicit formulas for the observable model}

The following result provides formulas involving the resolvent of $\sbm{ A \& B \\ C \& D}_o$.

\begin{thm}   \label{T:resolA}
The main operator $A_o$ of $\SmallSysNode_o$ is given explicitly by
$$
  (A_ox)(\mu)=\mu x(\mu)-\lim_{\re\eta\to\infty} \eta x(\eta),\quad\mu\in\cplus,
$$
for $x$ in $\dom{A_o}=\set{x\in\Hscr_o\mid\exists y\in\Yscr:~\mu\mapsto \mu x(\mu)-y\in\Hscr_o}$, and the observation operator is
$$
  C_ox=\lim_{\re\eta\to\infty} \eta x(\eta),\quad x\in\dom{A_o}.
$$

The resolvent of $A_o$ is given by 
\begin{equation}   \label{Aores}
  \left( ( \alpha - A_o)^{-1} x \right)(\mu) = \frac{ x(\mu) - x(\alpha)}{\alpha - \mu}, \quad \alpha, \mu \in {\mathbb C}^+,\, x \in \cH_o.
\end{equation}
Denoting the control operator of $\SmallSysNode_o$ by $B_o$, we have the following formulas:
\begin{align}  
\left(A_o(\alpha - A_o)^{-1}x\right)(\mu) & = \frac{ \mu x(\mu) - \alpha x(\alpha)}{ \alpha - \mu}, \quad \alpha, \mu \in {\mathbb C}^+, \, x \in \cH_o,  \notag \\
\left( (\alpha - A_o|_{\cH_o})^{-1} B_o u \right) (\mu) & = \frac{\varphi(\mu) - \varphi(\alpha)}{\alpha - \mu} u, 
\quad \alpha, \mu \in {\mathbb C}^+, \, u \in \cU,  \label{Aores2}\\
C_o(\alpha - A_o)^{-1}x &= x(\alpha),\quad \alpha \in\cplus,\, x\in\Hscr_o \notag.
\end{align}
\end{thm}

Just as in Subsection \ref{S:contexp} for the controllable, energy-preserving case,  the formula \eqref{Aores} for the resolvent of $A_o$ suggests a way to concretely identify the $(-1)$-scaled rigged space $\cH_{o,-1}$ defined abstractly as the completion of the space $\cH_o$ in the norm
$$
  \| x\| = \| (\beta - A_o)^{-1}\|_{\cH_o}.
$$
Namely we define
\begin{equation}   \label{Ho-1con}
\cH_{o,-1} = \left\{ x : {\mathbb C}^+ \to \cY \mid \mu \mapsto \frac{ x(\mu) - x(\beta)}{\beta - \mu} \in \cH_o \right\}
\end{equation}
with norm given by
\begin{equation}  \label{Ho-1con-norm}
\| x \|_{\cH_{o,-1}} = \left\| \mu \mapsto \frac{ x(\mu) - x(\beta)}{\beta - \mu} \right\|_{\cH_o}.
\end{equation}
We emphasize again that the $\cH_{o,-1}$ norm (and inner product) depends on the choice of $\beta \in {\mathbb C}_+$; different choices of $\beta$ give different norms although all such norms are equivalent. The elements of $\Hscr_{o,-1}$ are equivalence classes of functions modulo constant terms.
We have the following analogue of Theorem \ref{thm:hdc-1concrete}:

\begin{thm}\label{T:Ho-1concrete}
Let the space $\cH_{o,-1}$ be given by \eqref{Ho-1con} and \eqref{Ho-1con-norm}.
\begin{enumerate}
\item The map $\iota : x \mapsto [x]$ embeds $\cH_o$ into  $\cH_{o,-1}$ as a dense subspace.  A given element $[z] \in \cH_{o,-1}$
is of the form $\iota(x)$ for some $x \in \cH_o$ if and only if the function $\displaystyle{\mu \mapsto \frac{z(\mu) - z(\beta)}{\beta - \mu}}$, $\mu\in\cplus$, is not only in $\cH_o$ but in fact is in $\dom{A_o} = \cH_{o,1} \subset \cH_o$.  When this is the case, the equivalence class representative $x$ for $[z]$, for which $x \in \cH_o$, is uniquely determined by the decay condition at infinity:  
$$
\lim_{\re \eta \to \infty}  x(\eta) = 0.
$$

\item Define an operator $A_o|_{\cH_o} : \cH_o \to \cH_{o, -1}$ by
$$
  A_o|_{\cH_o}x := [\mu \mapsto  \mu x(\mu)],\quad x\in\Hscr_o,\,\mu\in\cplus.
$$
When $\cH_o$ is identified as a linear sub-manifold of $\cH_{o,-1}$, then $A_o|_{\cH_o}$ is the unique extension of $A_o : \dom{A_o} \to \cH_o$ to an operator in $\Bscr(\Hscr_o;\Hscr_{o,-1})$.
Moreover, $(\beta - A_o|_{\cH_o})^{-1}$ is a unitary map from $\cH_{o,-1}$ to $\cH_o$. 

\item With $\cH_{o, -1}$ identified concretely as in \eqref{Ho-1con}, the action of $B_o : \cU \to \cH_{o,-1}$ is given by
$$
  B_ou := [ \mu\mapsto \varphi(\mu)  u ],\quad u\in\Uscr,\,\mu\in\cplus.
$$
\end{enumerate}
\end{thm}

We shall now present formulas for the action of the operators of \emph{the dual of} $ \SmallSysNode_o $  on generic functions in their domains. For $\alpha \in {\mathbb C}^+$ we define $\tau_{o,\alpha} : \cH_o \to \cU$ by
$$
  \tau_{o,\alpha} := B_o^* (\alpha - A_o^*)^{-1}, \quad \alpha\in\cplus,
$$
and it follows from \eqref{Aores2} that 
$$
  (\tau_{o,\alpha})^* u=  \mu\mapsto  \frac{ \varphi(\mu) - \varphi(\overline\alpha)}{\overline\alpha-\mu} u,\quad u\in\Uscr,\,\mu\in\cplus.
$$
The map $\tau_{o,\alpha}$ enters into the explicit formula for the resolvent of  $A_o^*$ acting on generic elements of $\Hscr_o$, as described in the following analogue of Corollary \ref{cor:ContrRes} and Theorem \ref{thm:enprescontainedadjoint}:

\begin{thm}\label{thm:Aores}
The following claims are true:
\begin{enumerate}
\item For arbitrary $\sbm{x\\y}\in\dom{\SmallSysNode_o^*}$, if we set $u:=\bbm{B_o^*\&D_o^*}\sbm{x\\y}$, then we get
$$
  \bbm{A_o^*\&C_o^*}\bbm{x\\y}=\mu\mapsto -\mu x(\mu)-\varphi(\mu)u+y,\quad\mu\in\cplus.
$$
It follows that
$$
\begin{aligned}
  \dom{\SmallSysNode_o^*}&\subset\bigg\{\bbm{x\\y}\in\bbm{\Hscr_c\\\Yscr} \bigmid \\
   &\qquad \exists u\in\Uscr:~ \mu\mapsto -\mu x(\mu)-\varphi(\mu)u+y\in\Hscr_o\bigg\}.
\end{aligned}
$$

\item For $x \in \dom{A_o^*}$ the function $A_o^* x \in \cH_o$ satisfies the identity
$$
\left( A_o^* x\right) (\mu) = -\mu x(\mu) - \varphi(\mu)  B_o^* x,\quad \mu\in\cplus,
$$
and in particular, 
\begin{equation}  \label{domAo*}
 \dom{A_o^*} \subset   \{ x \in \cH_o \mid \exists u \in \cU :~   \mu \mapsto \mu x(\mu) +  \varphi(\mu) u \in \cH_o\}.  
\end{equation}

\item If we know $A_o^*$, then we can characterize $\dom{\SmallSysNode_o^*}$ and $\bbm{B_c^*\&D_c^*}$ as follows, for an arbitrary $\lambda\in\cplus$:
$$
\begin{aligned}
  \dom{\SmallSysNode_o^*} &= \set{\bbm{x\\y}\in\bbm{\Hscr_o\\\Yscr}\bigmid x-e_o(\overline\lambda)^*y\in\dom{A_o^*}}, \\
  \bbm{B_c^*\&D_c^*}\bbm{x\\y} &= \tau_{o,\lambda}(\lambda-A_o^*)\big(x-e_o(\lambda)^*y\big)+\wtsmash\varphi(\overline\lambda);
\end{aligned}
$$
neither of these two depends on the choice of $\lambda\in\cplus$.

\item We have the following formula for the resolvent of $A_o^*$:
\begin{equation}   \label{Ao*res}
\big( (\overline\alpha - A_o^*)^{-1} x \big)(\mu) = \frac{ x(\mu) - \varphi(\mu) \, \tau_{o,\overline\alpha}\, x}{\overline\alpha + \mu},
  \quad \alpha,\mu \in {\mathbb C}^+,\,x \in \cH_o,
\end{equation}
and the action of this resolvent on kernel functions of $\Hscr_o$ is
\begin{equation}   \label{Ao*res2}
(\overline\alpha - A_o^*)^{-1} e_o(\lambda)^*y = 
  \frac{e_o(\lambda)^*-e_o(\alpha)^*}{\overline\alpha-\overline\lambda}y,
  \quad \alpha,\lambda \in {\mathbb C}^+,\,y \in \cY.
\end{equation}
\end{enumerate}
\end{thm}

The formula \eqref{Ao*res2} is useful when calculating the reproducing kernel of $\Hscr_{o,1}$. 

Similar to Lemma \ref{lem:domequivdeluxe}, with an added assumption, it is possible to strengthen the containment in \eqref{domAo*} to an equality. Moreover, we obtain a characterization of when the observable energy-preserving realization is in fact conservative, cf.\ the last assertion of Theorem \ref{T:coiso-real}.

\begin{thm}\label{thm:assume2} The following two conditions are equivalent:
\begin{enumerate}
\item For some (or equivalently for every) $\alpha\in\cplus$, the state space $\cH_o$ has the property
\begin{equation}   \label{assume2}
  \cH_o\cap \set{\mu\mapsto \frac{\varphi(\mu)u}{\alpha+\mu}\bigmid u\in\Uscr} = \{0\}.
 \end{equation}
\item We have both
\begin{equation}\label{eq:schurnotinHo}
 \Hscr_o \cap \set{\varphi(\cdot)u\mid u\in\Uscr}=\zero\quad\text{and}
\end{equation}
\begin{equation}  \label{Ao*idassume2}
\dom{A_o^*} = \{ x \in \cH_o \mid \exists u \in \cU:~ \mu \mapsto \mu x(\mu) +  \varphi(\mu) u \in \cH_o\}.
\end{equation}
\end{enumerate}
The conditions \eqref{assume2}--\eqref{Ao*idassume2} hold together with the implication $\varphi(\cdot)u=0\Rightarrow u=0$ if and only if $\SmallSysNode_o$ is conservative. This is in turn true if and only if $1 - \varphi(\cdot)^* \varphi(\cdot)$ has maximal factorable minorant in the right half-plane sense equal to 0.

When \eqref{eq:schurnotinHo} holds, a given $x \in \cH_o$ in $\dom{A_o^*}$ determines the function $\varphi(\cdot)u$ appearing in \eqref{domAo*} uniquely through the formula
\begin{equation}\label{eq:schurnotinHoCons}
   \varphi(\mu)u =  \varphi(\mu) B_o^*x,\quad \mu\in\cplus.
\end{equation}
If $\SmallSysNode_o$ is conservative, then $x\in\dom{A_o^*}$ further determines the vector $u\in\Uscr$ uniquely in \eqref{eq:schurnotinHoCons}.
\end{thm}

Recall that $\cH^d_{o,-1}$ is defined to be the completion of $\cH_o$ in the norm $\|x \| = \left\| (\overline\beta -A_o^*)^{-1} x \right\|$.
With assumption \eqref{assume2} in force,  Theorem \ref{thm:assume2} assures us that $\dom{ A_o^* }$ is given by \eqref{Ao*idassume2}.  
Then  formula \eqref{Ao*res} suggests the following concrete identification of the $(-1)$-rigged space:
\begin{equation}   \label{Hod-1}
 \cH^d_{o,-1}  := \left\{ x : {\mathbb C}^+ \to \cY \bigmid \exists u\in \cU:~ \mu \mapsto \frac{x(\mu) + \varphi(\mu) u}{ \overline\beta + \mu} \in \cH_o\right\}.
\end{equation}
Now the assumption \eqref{assume2} guarantees us that the choice of $\varphi(\cdot)u$ in \eqref{Hod-1} is uniquely determined by $x$ (whenever at least one suitable $u\in\Uscr$ exists).
For $x \in \cH^d_{o,-1}$ we set
\begin{equation}   \label{Hod-1norm}
  \| x \|_{\cH^d_{o,-1}} := \left\| \mu\mapsto \frac{ x(\mu) + \varphi(\mu) u}{\overline\beta + \mu} \right\|_{\cH_o}.
\end{equation}
Thus elements of $\cH^d_{o,-1}$ are equivalence classes of functions modulo the subspace $\varphi (\cdot)\, \cU$. These equivalence classes are denoted as $[x]$ where $x$ is any particular representative of the equivalence class.   We summarize the properties of the space $\cH^d_{o,-1}$ as follows; see also Theorem \ref{T:Hc-1concrete}:

\begin{prop}\label{P:Hdo-1concrete}  
Assume that \eqref{assume2} holds and let the space $\cH^d_{o,-1}$  be given by \eqref{Hod-1}--\eqref{Hod-1norm} as above.
Then:
\begin{enumerate}
\item The map $\iota : x \mapsto [x]$ embeds $\cH_o$ into $\cH^d_{o,-1}$ as a dense subspace. A given element $[z] \in \cH^d_{o,-1}$ is of the form $\iota(x)$ for some $x \in \cH_o$ if and only if there is a choice of $u$ so that the function
$\displaystyle{\mu \mapsto \frac{z(\mu) + \varphi(\mu) u}{ \overline\beta + \mu}}$, $\mu\in\cplus$, is not only in $\cH_o$ but also in $\cH^d_{o,1} = \dom{A_o^*}$. This choice of $\varphi(\cdot)u$ is then unique.

\item Define an operator $A_o^*|_{\cH_o} : \cH_o \to \cH^d_{o,-1}$ by
$$
  A_o^*|_{\cH_o}x := [\mu\mapsto -\mu x(\mu)],\quad x\in\Hscr_o,\,\mu\in\cplus.
$$
When $\cH_o$ is identified as a linear sub-manifold of $\cH^d_{o,-1}$ as in statement 1, then $A_o^*|_{\cH_o}$ is the unique extension of $A_o^* : \dom{A_o^*} \to \cH_o$ to an operator in $\Bscr(\Hscr_o;\Hscr_{o,-1}^d)$. The resolvent of $A_o^*|_{\Hscr_o}$ is given by 
$$
  \big((\overline\alpha - A_o^*|_{\cH_{c}})^{-1}[x]\big)(\mu)=\frac{ x(\mu) + \varphi(\mu) u}{\overline\alpha + \mu},\quad [x]\in\Hscr_{o,-1}^d,\,\alpha,\mu\in\cplus,
$$
where the condition $\mu\mapsto\frac{ x(\mu) + \varphi(\mu) u}{\overline\alpha + \mu}\in\Hscr_o$ uniquely determines $\varphi(\cdot) u$. Moreover, $(\overline\beta - A_o^*|_{\cH_{c}})^{-1}$ is unitary from $\cH_{c,-1}$ to $\cH_c$.

\item The action of $C_o^* : \cY \to \cH^d_{o,-1}$ is given by
$$
  C_o^*y := [ \mu \mapsto y],\quad y\in\Yscr,\,\mu\in\cplus.
$$
\end{enumerate}
\end{prop}

We next present a collection of reproducing-kernel formulas. This is the dual version of Proposition \ref{prop:repkerform-c} and Corollary \ref{cor:kerneldual}. Again, $\Hscr_{o,-1}$ and $\Hscr_{o,-1}^d$ are not spaces of functions, but we can identify them with the Hilbert spaces with reproducing kernels
$$
\begin{aligned}
  K_{o,-1}(\mu,\lambda) &= e_{o,-1}(\mu)\big(e_{o,-1}(\lambda)\big)^*\quad\text{and}\\
  K_{o,-1}^d(\mu,\lambda) &= e_{o,-1}^d(\mu)\big(e_{o,-1}^d(\lambda)\big)^*,
\end{aligned}
$$
respectively, where
$$
\begin{aligned}
  e_{o,-1}(\mu)[z] &:= z(\mu)-z(\beta),\quad [z]\in\Hscr_{o,-1},\,\mu\in\cplus,\quad\text{and}\\
  e_{o,-1}^d(\mu) &:(\overline\beta-A_o^*|_{\Hscr_o})x \mapsto (\overline\beta+\mu)x(\mu),\quad x\in\Hscr_o,\,\mu\in\cplus,
\end{aligned}
$$
are bounded operators that point-evaluate well-chosen representatives of the equivalence classes in $\Hscr_{o,-1}$ and $\Hscr_{o,-1}^d$.

\begin{prop}  \label{prop:repkerform-o}
We have the following formulas for the kernel functions
associated with the reproducing kernel Hilbert spaces $\cH_o$, 
 $\cH_{o,\pm1}$, and $\cH^d_{o,\pm1}$ (with $\mu,\lambda\in\cplus$):
 \begin{align}
  K_o(\mu,\lambda)  & = C_o (\mu - A_o)^{-1} (\overline\lambda - A_o^*|_{\cH_o})^{-1} C_o^* 
  = \frac{1 - \varphi(\mu) \varphi(\lambda)^*}{\mu + \overline\lambda}, \notag \\
  K_{o,1}(\mu,\lambda)  & = C_o (\mu - A_c)^{-1} (\beta - A_o)^{-1} (\overline{\beta} - A_o^*)^{-1}
   (\overline\lambda - A_o^*|_{\cH_o})^{-1} C_o^* \notag \\
  &= \frac{K_o(\mu,\lambda)-K_o(\mu,\beta)}{(\beta-\mu)(\overline\beta-\overline\lambda)} 
  -\frac{K_o(\beta,\lambda)-K_o(\beta,\beta)}{(\beta-\mu)(\overline\beta-\overline\lambda)} \notag \\
   K_{o,-1}(\mu,\lambda)  & = (\beta - \mu) (\overline\beta - \overline\lambda) C_o (\mu - A_o)^{-1} (\overline\lambda - A_o^*|_{\cH_o})^{-1} C_o^* \notag \\
  &= (\beta - \mu) (\overline\beta - \overline\lambda) K_o(\mu,\lambda), \notag \\
  K_{o,1}^d(\mu,\lambda) & = C_o (\mu - A_c)^{-1} (\overline\beta - A_o^*)^{-1} (\beta - A_o)^{-1}
   (\overline\lambda - A_o^*|_{\cH_o})^{-1} C_o^* \notag \\
  &= \frac{\kappa_o(\mu,\lambda)-\varphi(\mu)\tau_{o,\overline\beta}\,\big(\kappa_o(\cdot,\lambda)\big)}{\overline\beta+\mu},\quad\text{where} \notag \\
  &\qquad \kappa_o(\mu,\lambda) := \frac{K_o(\mu,\lambda)-K_o(\beta,\lambda)}{\beta-\mu}, \notag \\
  K_{o,-1}^d(\mu,\lambda) & = (\overline\beta + \mu) (\beta + \overline\lambda) C_o (\mu - A_o)^{-1} (\overline\lambda - A_o^*|_{\cH_o})^{-1} C_o^* \label{eq:Kod-1kernel} \\
  &= (\overline\beta + \mu) (\beta + \overline\lambda) K_o(\mu,\lambda). \label{eq:Kod-1kernel2}
\end{align}
Here $\beta$ is the parameter used in the construction of the rigging $\dom{A_o}\subset\Hscr_o\subset\Hscr_{o,-1}$ as usual. The formulas \eqref{eq:Kod-1kernel} and \eqref{eq:Kod-1kernel2} have only been established under the assumption \eqref{assume2}.

Moreover, for all $\mu,\lambda\in\cplus$ and $y,v\in\Yscr$, we have
$$
\begin{aligned}
  \Ipdp{K_{o,1}(\mu,\lambda) y}{v}_\Yscr&=\Ipdp{K_o(\cdot,\lambda)y}{K_o(\cdot,\mu)v}_{\Hscr_{c,-1}^d}\quad\text{and} \\
  \Ipdp{K_{o,1}^d(\mu,\lambda) y}{}_\Yscr&=\Ipdp{K_o(\cdot,\lambda)y}{K_o(\cdot,\mu)v}_{\Hscr_{c,-1}}.
\end{aligned}
$$
 \end{prop}

This completes our study of the observable functional model.

\section{Recovering the classical de Branges-Rovnyak models}\label{sec:recovering}

In this section we use the so-called internal Cayley transformation to recover the original de Branges-Rovnyak models \eqref{eq:deBobs} and \eqref{eq:deBcontr}. This Cayley transformation is described in detail in \cite[\S 12.3]{StafBook}; here we only include the small fragments of the theory that we need.

\subsection{The internal Cayley transformation}
Following \cite[Thm 12.3.6]{StafBook}, the \emph{Cayley transform with parameter $\alpha\in\res A$} of the system node $\SmallSysNode$ is the bounded operator $\sbm{\dA&\dB\\\dC&\dD}$ from $\sbm{\Xscr\\\Uscr}$ into $\sbm{\Xscr\\\Yscr}$ defined by
\begin{equation}\label{eq:cayley}
\begin{aligned}
  \dA &:= (\overline\alpha+A)(\alpha-A)^{-1},\quad \dB := \sqrt{2\re\alpha} \, (\alpha-A|_\Xscr)^{-1}B,\\
  \dC &:= \sqrt{2\re\alpha} \, C(\alpha-A)^{-1},\quad \text{and}\quad \dD := \whsmash\Dfrak(\alpha),
\end{aligned}
\end{equation}
where $A$, $A|_\Xscr$, and $B$ are as in Definition \ref{def:altsysnode}, and $C$ and $\whsmash\Dfrak$ are given by \eqref{eq:Cdef} and \eqref{eq:transfundef}, respectively.

We interpret the bounded operator $\sbm{\dA&\dB\\\dC&\dD}$ described by \eqref{eq:cayley} as the connecting operator of a discrete-time system with the same input space $\Uscr$, state space $\Xscr$, and output space $\Yscr$ as $\SmallSysNode$:
\begin{equation}\label{eq:discrete}
  \bbm{x(k+1)\\y(k)}=\bbm{\dA&\dB\\\dC&\dD}\bbm{x(k)\\u(k)},\quad k\in\zplus.
\end{equation}
As in the introduction, the transfer function of the system \eqref{eq:discrete} is \begin{equation}\label{eq:CaylTransf}
	\whsmash\dD(z)=z\dC(1-z\dA)^{-1}\dB+\dD.
\end{equation}
The reader should be warned that the transfer function of a discrete-time system is defined as $\dC(z-\dA)^{-1}\dB+\dD$ in \cite{StafBook}, but the results can be translated from one setting to the other by interchanging $z$ and $1/z$. 

Recall that a discrete-time system with input space $\Uscr$, state space $\Xscr$, and output space $\Yscr$ is (scattering) passive, energy preserving, or co-energy preserving, if and only if the connecting operator is contractive, isometric, or co-isometric, respectively, from $\sbm{\Xscr\\\Uscr}$ to $\sbm{\Xscr\\\Yscr}$; see e.g.\ \cite[\S5]{StafMTNS02}.

We use the linear fractional transformation  
\begin{equation}\label{eq:CpDmap}
  z_\alpha(\mu):=\frac{\alpha-\mu}{\overline\alpha+\mu},\quad\mu\in\cplus
  \quad\Longleftrightarrow\quad
  \mu_\alpha(z)=\frac{\alpha-\overline\alpha z}{1+z},\quad z\in\D,
\end{equation}
also referred to as a Cayley transformation, to map $\cplus$ one-to-one onto $\D$. In the sequel, we often abbreviate $z_\alpha(\cdot)=z(\cdot)$ and $\mu_\alpha(\cdot)=\mu(\cdot)$. By combining the well-known resolvent identity
$$
  (\mu-A|_\Xscr)^{-1} - (\alpha-A|_\Xscr)^{-1} = 
    (\alpha-\mu)(\alpha-A)^{-1}(\mu-A|_\Xscr)^{-1},\quad \mu,\alpha\in\res A,
$$
with the definition \eqref{eq:transfundef} of the transfer function $\whsmash\Dfrak$, one can verify that the transfer function in \eqref{eq:CaylTransf} satisfies
\begin{equation}\label{eq:CpDmapTransf}
  \whsmash\dD(z)=\whsmash\Dfrak\big(\mu_\alpha(z)\big),\quad \frac1z\in\res \dA,
\end{equation}
if $\SmallSysNode$ and $\sbm{\dA&\dB\\\dC&\dD}$ are related by \eqref{eq:cayley}.

\begin{rem}\label{rem:cayleydual}
According to \cite[Thms 3.1 and 3.2]{StafMTNS02}, the Cayley transform $\sbm{\dA&\dB\\\dC&\dD}$ is a contraction (isometric) for some/for all $\alpha\in\cplus$ if and only if the original system $\SmallSysNode$ is passive (energy preserving). Moreover, $\sbm{\dA&\dB\\\dC&\dD}$ is controllable (observable) if and only if $\SmallSysNode$ is controllable (observable).  The convention here that the transfer function has the form $z \dC (1 - z\dA)^{-1}\dB+\dD$ rather than $\dC(z - \dA)^{-1}\dB+\dD$ has no influence on these general facts; see also \cite[Sect.\ 12.3]{StafBook}.

It is easy to show that the adjoint of $\sbm{\dA&\dB\\\dC&\dD}$ in \eqref{eq:cayley} is the Cayley transform of the dual of $\SmallSysNode$ with parameter $\overline\alpha\in\res{A^*}$ along lines similar to the proof of \cite[Lem.\ 6.2.14]{StafBook}. Hence $\sbm{\dA&\dB\\\dC&\dD}$ is a co-isometry for some/for all $\alpha\in\cplus$ if and only if $\SmallSysNode$ is a co-energy-preserving system node.
\end{rem}

Intertwinement of discrete-time systems is defined in self-evident analogy to the continuous-time case; see \eqref{eq:oDiscrInt} below.

\subsection{The observable co-energy-preserving models}

It follows immediately from Theorem \ref{T:resolA} and \eqref{eq:cayley} that the internal Cayley transform with parameter $\alpha\in\cplus$ of the observable co-energy preserving model $\SmallSysNode_o$ for the Schur function $\varphi$ on $\cplus$ is the operator $\sbm{\dA_o&\dB_o\\\dC_o&\dD_o}$, where

\begin{equation}\label{eq:obscayleyfirst}
 \begin{aligned}
  (\dA_o x)(\mu) &= \frac{\overline\alpha+\mu}{\alpha-\mu}\, x(\mu) - \frac{2\re\alpha}{\alpha-\mu}\, x(\alpha), \quad x\in\Hscr_o,\,\mu\in\cplus, \\
  (\dB_o u)(\mu) &= \sqrt{2\re\alpha}\,\frac{\varphi(\mu)-\varphi(\alpha)}{\alpha - \mu}u, \quad u\in\Uscr,\,\mu\in\cplus,\\  
  \dC_o x &= \sqrt{2\re\alpha}\,x(\alpha),\quad x\in\Hscr_o, \quad \text{and}\\
  \dD_o u &= \varphi(\alpha) u, \quad u\in\Uscr.
\end{aligned}
\end{equation}

The system \eqref{eq:obscayleyfirst} is observable and isometric, and by \eqref{eq:CpDmapTransf} the transfer function $\phi_\alpha$ of \eqref{eq:obscayleyfirst} satisfies 
\begin{equation}\label{eq:CpDmapTransf2inv}
  \phi_\alpha(z)=\varphi\big(\mu_\alpha(z)\big),\quad z\in\D.
\end{equation}
We denote the Hilbert space with reproducing kernel 
\begin{equation}\label{eq:KoAlpha}
  \mathrm K_{o,\alpha}(z,w):=\frac{1-\phi_\alpha(z)\phi_\alpha(w)^*}{1-z\overline w},\quad z,w\in\D,
\end{equation}
by $\mathrm H_{o,\alpha}$. By assertion 4 of Theorem \ref{T:coiso-real}, the operator $\sbm{\dA_o&\dB_o\\\dC_o&\dD_o}$ must be unitarily similar to the de corresponding Branges-Rovnyak discrete-time model realization $\sbm{\mathrm A_o&\mathrm B_o\\\mathrm C_o&\mathrm D_o}$ in \eqref{eq:deBobs}, constructed using the transfer function $\phi_\alpha\in\cS(\D;\cU, \cY)$ in \eqref{eq:CpDmapTransf2inv}. The following result describes this unitary similarity: 

\begin{prop}\label{prop:obsdbrrecover}
For arbitrary $\varphi\in\Sscr(\cplus;\Uscr,\Yscr)$ and $\alpha\in\cplus$, the following claims are true:
\begin{enumerate}
\item Let $\SmallSysNode_o$ be the canonical observable co-energy-preserving model for $\varphi\in\Sscr(\cplus;\Uscr,\Yscr)$. Then $\sbm{\dA_o&\dB_o\\\dC_o&\dD_o}$ in \eqref{eq:obscayleyfirst} is the Cayley
transform with parameter $\alpha$ of $\SmallSysNode_o$. 

\item The following linear operator maps $\mathrm H_{o,\alpha}$ unitarily onto $\Hscr_o$:
\begin{equation}\label{eq:XioDef}
  (\Xi_\alpha \, \xi)(\mu) := \frac{\sqrt{2\re\alpha}}{\overline\alpha+\mu} \,\xi\big(z_\alpha(\mu)\big),
    \quad \xi\in\mathrm H_o,\,\alpha,\mu\in\cplus,
\end{equation}
where $\alpha$ is the same in \eqref{eq:CpDmap} and \eqref{eq:XioDef}. The inverse is
\begin{equation}\label{eq:XioInv}
  \big((\Xi_\alpha)^{-1} \, \zeta\big)(z) = \frac{\sqrt{2\re\alpha}}{1+z} \,\zeta\big(\mu_\alpha(z)\big),
    \quad \alpha\in\cplus,\,\zeta\in\Hscr_o,\,z\in\D.
\end{equation}

\item Let $\sbm{\mathrm A_o&\mathrm B_o\\\mathrm C_o&\mathrm D_o}$ be 
the de Branges-Rovnyak model\ realization in \eqref{eq:deBobs} of $\phi_\alpha$. The 
operator $\Xi_\alpha$ intertwines $\sbm{\dA_o&\dB_o\\\dC_o&\dD_o}$ and 
$\sbm{\mathrm A_o&\mathrm B_o\\\mathrm C_o&\mathrm D_o}$:
\begin{equation}\label{eq:oDiscrInt}
  \bbm{\dA_o\Xi_\alpha&\dB_o\\\dC_o\Xi_\alpha&\dD_o} 
    = \bbm{\Xi_\alpha\mathrm A_o&\Xi_\alpha \mathrm B_o\\\mathrm C_o&\mathrm D_o}.
\end{equation}
\end{enumerate}
\end{prop}
\begin{proof}
We leave it to the reader to verify assertion 1 as described above \eqref{eq:obscayleyfirst}. In 
order to prove assertion 2, we for notational reasons first relate $w$ to $\mu$ as $z$ is related to 
$\lambda$ in \eqref{eq:CpDmap}: 
\begin{equation}\label{eq:CpDmap2}
  w_\alpha(\lambda) := \frac{\alpha-\lambda}{\overline\alpha+\lambda},\quad\lambda\in\cplus 
    \quad \Longleftrightarrow \quad
    \lambda_\alpha(w) = \frac{\alpha-\overline\alpha w}{1+w},\quad w\in\D.
\end{equation}

The key to the unitarity of $\Xi_\alpha$ is the following relationship between the reproducing kernels of $\Hscr_o$ and $\mathrm H_{o,\alpha}$:
\begin{equation}\label{eq:ocaylkern}
  \mathrm K_o\big(z(\mu),w(\lambda)\big) = \frac{(\overline\alpha+\mu)(\alpha+\overline\lambda)}{2\re\alpha} K_o(\mu,\lambda),\quad\mu,\lambda\in\cplus,
\end{equation}
which follows from the fact that
$$
\begin{aligned}
  \frac{1-\phi_\alpha\big(z(\mu)\big)\phi_\alpha\big(w(\lambda)\big)^*}{1-z(\mu)\overline {w(\lambda)}}
  &=\frac{1-\varphi(\mu)\varphi(\lambda)^*}{1-\frac{\alpha-\mu}{\overline\alpha+\mu}
    {\frac{\overline\alpha-\overline\lambda}{\alpha+\overline\lambda}}} \\
  &= \frac{(\overline\alpha+\mu)(\alpha+\overline\lambda)}{2\re\alpha} 
    \frac{1-\varphi(\mu)\varphi(\lambda)^*}{\mu+\overline\lambda}.
\end{aligned}
$$
Combining \eqref{eq:ocaylkern} with \eqref{eq:XioDef} gives that the action of $\Xi_\alpha$ on kernel functions $\mathrm e_o$ in $\mathrm H_{o,\alpha}$ is
\begin{equation}\label{eq:XioKern}
  \Big(\Xi_\alpha \mathrm e_o\big(w(\lambda)\big)^*y\Big)(\mu) = \frac{\alpha+\overline\lambda}{\sqrt{2\re\alpha}} e_o(\lambda)^*y,\quad 
  \lambda\in\cplus,\,y\in\Yscr.
\end{equation}
It now follows that $\Xi_\alpha$ is isometric, since (using \eqref{eq:ocaylkern} in the second equality)
$$
\begin{aligned}
  \Ipdp{\Xi_\alpha \mathrm e_o\big(w(\lambda)\big)^*y}{\Xi_\alpha \mathrm e_o\big(z(\mu)\big)^*\gamma}_{\Hscr_o} &= 
    \Ipdp{\frac{\alpha+\overline\lambda}{\sqrt{2\re\alpha}} e_o(\lambda)^*y}
      {\frac{\alpha+\overline\mu}{\sqrt{2\re\alpha}} e_o(\mu)^*\gamma}_{\Hscr_o}\\ 
   &= \Ipdp{\mathrm K_o\big(z(\mu),w(\lambda)\big)y}{\gamma}_\Yscr \\
   &= \Ipdp{\mathrm e_o\big(w(\lambda)\big)^*y}{\mathrm e_o\big(z(\mu)\big)^*\gamma}_{\mathrm H_o}.
\end{aligned}
$$
The equation \eqref{eq:XioKern} moreover implies that the range of $\Xi_\alpha$ contains 
the dense subspace $\spn\set{e_o(\lambda)^*y\mid \lambda\in\cplus,\,y\in\Yscr}$ of $\Hscr_o$. We conclude that $\Xi_\alpha$ is unitary as claimed. Formula \eqref{eq:XioInv} follows from \eqref{eq:XioDef} by denoting the right-hand side of \eqref{eq:XioDef} by $\zeta(\mu)$, changing the variable from $\mu\in\cplus$ to $z\in\D$ using \eqref{eq:CpDmap}, and solving for $\xi(z)$.

The following calculations use \eqref{eq:deBobs} and prove \eqref{eq:oDiscrInt}:
$$
\begin{aligned}
 (\Xi_\alpha\mathrm B_o u)(\mu) &= 
    \frac{\sqrt{2\re\alpha}}{\overline\alpha+\mu}\,
    \frac{\phi\big(z(\mu)\big)-\phi(0)}{z(\mu)}u
    = \sqrt{2\re\alpha}\, \frac{\varphi(\mu)-\varphi(\alpha)}{\alpha-\mu}u \\
    &= (\dB_o u)(\mu), \\
 (\Xi_\alpha \mathrm A_o \xi)(\mu) &= \frac{\sqrt{2\re\alpha}}{\overline\alpha+\mu}\,
    \frac{\xi\big(z(\mu)\big)-\xi(0)}{z(\mu)} 
    = \sqrt{2\re\alpha}\, \frac{\xi\big(z(\mu)\big)-\xi(0)}{\alpha-\mu},
\end{aligned}
$$

$$
\begin{aligned}
 (\dA_o\Xi_\alpha\xi)(\mu) &= \frac{\overline\alpha+\mu}{\alpha-\mu}\,
    \frac{\sqrt{2\re\alpha}}{\overline\alpha+\mu}\, \xi\big(z(\mu)\big)
    -\frac{2\re\alpha}{\alpha-\mu}\,\frac{\sqrt{2\re\alpha}}{\overline\alpha+\alpha}
    \xi\big(z(\alpha)\big)\\
  &= \sqrt{2\re\alpha} \, \frac{\xi\big(z(\mu)\big)-\xi(0)}{\alpha-\mu}
    = (\Xi_\alpha \mathrm A_o \xi)(\mu), \quad \mathrm{and} \\
  \dC_o\Xi_\alpha\xi &= \sqrt{2\re\alpha}\, (\Xi_\alpha\xi)(\alpha) 
    = \sqrt{2\re\alpha}\, \frac{\sqrt{2\re\alpha}}{\overline\alpha+\alpha}\, \xi\big(z(\alpha)\big)
    = \xi(0) = \mathrm C_o\xi, 
\end{aligned}
$$
valid for all $\xi\in \mathrm H_o$, $\mu\in\cplus$, and $u\in\Uscr$.
\end{proof}

Note the interesting fact that
$$
  K_{o,-1}^d(\mu,\lambda)=(2\re\beta)\,\mathrm K_o\big(z_\beta(\mu),w_\beta(\lambda)\big),\quad\mu,\lambda\in\cplus,
$$
cf.\ \eqref{eq:Kod-1kernel} and \eqref{eq:ocaylkern}. We have no explanation for this coincidence.

The operator $\Xi_\alpha$ is called the inverse Cayley transform in \cite{StafBook}. It is the frequency-domain analogue of the inverse Laguerre transform; see \cite[Thm 12.3.1 and Def.\ 12.3.2]{StafBook}. This unitary mapping can be used for transferring knowledge from the very well-known disk setting to the half-plane setting. For instance, by \eqref{eq:XioDef}, the condition \eqref{assume2} holds if and only if the only function in $\mathrm H_{o,\alpha}$ of the form $\phi_\alpha(\cdot)u$ is the zero function; also note that by \eqref{eq:CpDmapTransf2inv} we have $\varphi(\mu)u=0$ for all $\mu\in\cplus$ if and only if $\phi_\alpha(z)u=0$ for all $z\in\D$. Thus, the conditions \eqref{assume2} and $\varphi(\cdot)u=0\Rightarrow u=0$ hold if and only if the conditions \eqref{Uounitary1} and \eqref{Uounitary2} hold with $\phi=\phi_\alpha$. By the last assertion of Theorem \ref{T:coiso-real}, this is the case if and only if the corresponding observable co-isometric realization $\mathbf U_o$ is unitary, which by Remark \ref{rem:cayleydual} is equivalent to $\SmallSysNode_o$ being conservative. This provides an alternative proof of the statement that \eqref{assume2} and $\varphi(\cdot)u=0\Rightarrow u=0$ hold together if and only if $\SmallSysNode_o$ is conservative; see Theorem \ref{thm:assume2}. It is moreover easy to see that the internal Cayley transformation can be used to convert the statements (3) in Theorem \ref{thm:discrconschar} to statement (4) in Theorem \ref{thm:conschar} and the corresponding statement in Theorem \ref{thm:assume2}.

\subsection{The controllable energy-preserving models}

  We have seen in \eqref{eq:discrDBobs} and \eqref{eq:discrDBcontr} that the model space ${\rm H}_c = \cH({\rm K}_c)$ over ${\mathbb D}$ arises in the same way as ${\rm H}_o = \cH({\rm K}_o)$ but with $\wtsmash \phi(z) = \phi(\overline{z})^*$ in place of $\phi$.  Similarly for the models over ${\mathbb C}^+$, the model space $\cH_c = \cH(K_c)$ arises in the same way as $\cH_o = \cH(K_o)$ but with $\wtsmash \varphi(\mu) = \varphi(\overline{\mu})^*$ in place of $\varphi(\mu)$; see \eqref{o/ckernels}.  
If we assume that $\phi_\alpha$ and $\varphi$ are related according to \eqref{eq:CpDmapTransf2inv}, then we see that
\begin{equation}\label{eq:tildetransfun}
  \wtsmash \varphi(\mu) = \varphi(\overline{\mu})^* = \phi\big(z_\alpha(\overline\mu)\big)^*
 = \phi\left( \overline{z_{\overline\alpha}(\mu)}\right)^* = \wtsmash \phi\big(z_{\overline\alpha}(\mu)\big).
\end{equation}
This suggests that the appropriate mapping of $\cplus$ onto $\D$ for the energy-preserving setting should be $\mu\mapsto z_{\overline\alpha}(\mu)$ rather than $\mu\mapsto z_\alpha(\mu)$. Indeed, defining
\begin{equation}\label{eq:KcAlpha}
  \mathrm K_{c,\overline\alpha}(z,w):=\frac{1-\phi_{\overline\alpha}(\overline z)^*\phi_{\overline\alpha}(\overline w)}{1-z\overline w},\quad z,w\in\D,
\end{equation}
we obtain
\begin{equation}\label{eq:ccaylkern}
  \mathrm K_{c,\overline\alpha}\big(z_{\overline\alpha}(\mu),w_{\overline\alpha}(\lambda)\big) = \frac{(\alpha+\mu)(\overline\alpha+\overline\lambda)}{2\re\alpha} K_c(\mu,\lambda),\quad \mu,\lambda\in\cplus,
\end{equation}
and this leads to the following unitary similarity result for the discrete-time controllable realizations:

\begin{prop}\label{prop:contrdbrrecover}
Let $\varphi\in\Sscr(\cplus;\Uscr,\Yscr)$ and $\SmallSysNode_c$ be the controllable energy-preserving realization of $\varphi$ given in \eqref{eq:SodDef}. For arbitrary $\alpha\in\cplus$, the following claims are true:
\begin{enumerate}
\item The Cayley transform with parameter $\alpha$ of the canonical controllable energy-preserving model $\SmallSysNode_c$ is the isometry $\sbm{\dA_c&\dB_c\\\dC_c&\dD_c}:\sbm{\Hscr_c\\\Uscr}\to\sbm{\Hscr_c\\\Yscr}$ given by
\begin{align}
  (\dA_cx)(\mu) &= \frac{\overline\alpha-\mu}{\alpha+\mu}\,x(\mu)-\frac{2\re\alpha}{\alpha+\mu}\,\wtsmash\varphi(\mu)
    \tau_{c,\alpha} x,\quad x\in\Hscr_c,\,\mu\in\cplus, \notag \\
  \dB_c u &= \sqrt{2\re\alpha}\,e_c(\overline\alpha)^* u, \quad u\in\Uscr, \notag \\
  \dC_c x &= \sqrt{2\re\alpha}\,\tau_{c,\alpha} x,\quad x\in\Hscr_c,\quad\text{and} \label{eq:contrcayley}\\
  \dD_c u &= \varphi(\alpha) u,\quad u\in\Uscr. \notag
\end{align}

\item Let $\mathrm H_{c,\alpha}$ be the Hilbert space with reproducing kernel \eqref{eq:KcAlpha}. Then $\Xi_{\overline\alpha}$ in \eqref{eq:XioDef} is unitary from $\mathrm H_{c,\alpha}$ to $\Hscr_c$.

\item The operator $\Xi_{\overline\alpha}$ intertwines $\sbm{\dA_c&\dB_c\\\dC_c&\dD_c}$ in \eqref{eq:contrcayley} and $\sbm{\mathrm A_c&\mathrm B_c\\\mathrm C_c&\mathrm D_c}$ in \eqref{eq:deBcontr}:
\begin{equation}\label{eq:discrcontrint}
  \bbm{\dA_c\Xi_{\overline\alpha}&\dB_c\\\dC_c\Xi_{\overline\alpha}&\dD_c} =
  \bbm{\Xi_{\overline\alpha}\mathrm A_c&\Xi_{\overline\alpha}\mathrm B_c\\\mathrm C_c&\mathrm D_c}.
\end{equation}
\end{enumerate}
\end{prop}

\noindent\emph{Proof.} The formula for $\dB_c$ follows from \eqref{eq:cayley} and \eqref{eq:isomeval}, and the formula for $\dC_c$ from \eqref{eq:cayley} combined with \eqref{eq:taucdef}. By \eqref{Ac-id} and \eqref{eq:taucdef}, we have
$$
  \big(A_c(\alpha-A_c)^{-1}x\big)(\mu) = -\mu \big((\alpha-A_c)^{-1}x\big)(\mu) 
    - \wtsmash \varphi(\mu) \tau_{c,\alpha} x,\quad x\in\Hscr_c,\,\mu\in\cplus,
$$
and combining this with \eqref{Acresol} and \eqref{eq:cayley} gives the formula for $\dA_c$. Due to \eqref{eq:ccaylkern} we have
\begin{equation}\label{eq:XicKern}
  \Xi_{\overline\alpha} \, \mathrm e_c\big(w_{\overline\alpha}(\lambda)\big)^*u = \frac{\overline\alpha+\overline\lambda}{\sqrt{2\re\alpha}} e_c(\lambda)^*u,\quad \lambda\in\cplus,\,u\in\Uscr,
\end{equation}
and the unitarity of $\Xi_{\overline\alpha}$ follows from the argument in the proof of Proposition \ref{prop:obsdbrrecover}.2.

We have that $\dB_c=\Xi_{\overline\alpha}\mathrm B_c$, since by Theorem \ref{T:iso-real} and \eqref{eq:XicKern}, it for all $u\in\Uscr$ holds that
$$
\begin{aligned}
  \Xi_{\overline\alpha} \mathrm B_c u = \Xi_{\overline\alpha} \mathrm e_c(0)^* u 
    = \Xi_{\overline\alpha} \mathrm e_c\big(w_{\overline\alpha}(\overline\alpha)\big)^* u
    = \frac{\overline\alpha+\alpha}{\sqrt{2\re\alpha}}\, e_c(\overline\alpha)^* u = \dB_c u.
\end{aligned}
$$
Moreover $\mathrm C_c=\dC_c\Xi_{\overline\alpha}$, because by Theorem \ref{T:iso-real}, \eqref{eq:tildetransfun}, the unitarity of $\Xi_{\overline\alpha}$, and \eqref{taualphacadj}, we have the following equalities, valid for all $x\in\mathrm H_{c,\alpha}$ and $y\in\Yscr$:
$$
\begin{aligned}
  \Ipdp{\mathrm C_cx}{y}_\Yscr &= \Ipdp{x}{z\mapsto\frac{\wtsmash\phi(z)-\wtsmash\phi(0)}{z}y}_{\mathrm H_c} \\
  &=\Ipdp{\Xi_{\overline\alpha} x}{\Xi_{\overline\alpha}\left(z\mapsto\frac{\wtsmash\phi(z)-\wtsmash\phi(0)}{z}\right)y}_{\Hscr_c} \\
  &= \Ipdp{\Xi_{\overline\alpha} x}{\left(\mu\mapsto \frac{\sqrt{2\re\alpha}}{\alpha+\mu}
    \frac{\wtsmash\phi\big(z_{\overline\alpha}(\mu)\big)-\wtsmash\phi\big(z_{\overline\alpha}(\overline\alpha)\big)}{z_{\overline\alpha}(\mu)}\right)y}_{\Hscr_c} \\
  &= \Ipdp{\Xi_{\overline\alpha} x}{\left(\mu\mapsto \sqrt{2\re\alpha}\,
    \frac{\wtsmash\varphi(\mu)-\wtsmash\varphi(\overline\alpha)}{\overline\alpha-\mu}\right)y}_{\Hscr_c} \\
  &= \Ipdp{\sqrt{2\re\alpha}\,\tau_{c,\alpha}\,\Xi_{\overline\alpha}\, x}{y}_\Yscr.
\end{aligned}
$$

Finally, by Theorem \ref{T:iso-real}, \eqref{eq:contrcayley}, and \eqref{eq:tildetransfun}, it holds for $x\in\Hscr_c$ and $\mu\in\cplus$ that
\begin{align*}
  (\dA_c\Xi_{\overline\alpha} x)(\mu) &= \frac{\overline\alpha-\mu}{\alpha+\mu}\,\frac{\sqrt{2\re\alpha}}{\alpha+\mu}\,x\big(z_{\overline\alpha}(\mu)\big)
    - \frac{2\re\alpha}{\alpha+\mu}\,\wtsmash\varphi(\mu)\,\tau_{c,\alpha}\, \Xi_{\overline\alpha}\, x \\
  &= \frac{\sqrt{2\re\alpha}}{\alpha+\mu}\Big(z_{\overline\alpha}(\mu)\, x\big(z_{\overline\alpha}(\mu)\big)
    - \wtsmash\phi\big(z_{\overline\alpha}(\mu)\big)\,\mathrm C_c\, x\Big) \\
  &= (\Xi_{\overline\alpha}\mathrm A_c x)(\mu) \tag*{\qed}.
\end{align*}

\section{Final remarks}\label{sec:conc}

We have developed a realization theory for arbitrary $\varphi\in\Sscr(\cplus;\Uscr,\Yscr)$ that is completely analogous to the classical case worked out by de Branges and Rovnyak on the complex unit disk $\D$. The same general principles carry over from the discrete case, but unboundedness of most of the involved operators makes it more complicated to work out the details. By avoiding linear fractional transformations, we obtain more insight into intricacies specific to continuous-time systems, such as the Hilbert space riggings $\cH_{c,1}\subset \cH_c\subset\cH_{c,-1}$, $\cH_{o,1}\subset \cH_o\subset\cH_{o,-1}$, $\cH^{d}_{c,1}\subset \cH_c\subset\cH^{d}_{c,-1}$, and $\cH^{d}_{o,1}\subset \cH_o\subset\cH^{d}_{o,-1}$. 

Formulas for the canonical models $\SmallSysNode_c$ and $\SmallSysNode_o$, as well as their component operators, are summarized in the following tables:

\kern 5mm\noindent
{\bf Formulas related to $\SmallSysNode_o$}.  
\Symb{\cH_o}{$\Hscr(K_o)$, the state space of the observable model}
\Symb{\cH_{o,1}}{$\set{x\in\Hscr_o\mid\exists y\in\Yscr:~ \mu\mapsto \mu x(\mu)-y\in\Hscr_o}$, the domain of $A_o$}
\Symb{\cH_{o,-1}}{$\displaystyle{\set{z\mid \mu\mapsto\frac{ z(\mu) - z(\alpha)}{\alpha - \mu}\in\Hscr_o}}$, equivalence classes modulo constants}
\Symb{A_o}{$x\mapsto \big(\mu\mapsto \mu x(\mu)-\lim_{\re\eta\to\infty} \eta x(\eta)\big)$, maps $\Hscr_{o,1}$ boundedly into $\Hscr_o$}
\Symb{A_o|_{\cH_o}}{$x\mapsto[\mu\mapsto \mu x(\mu)]$, element of $\Bscr(\Hscr_o;\Hscr_{o,-1})$}
\Symb{B_o}{$u\mapsto [\mu\mapsto \varphi(\mu)u]$, operator in $\Bscr(\Uscr,\Hscr_{o,-1})$}
\Symb{C_o}{$x\mapsto\lim_{\re\eta\to\infty} \eta x(\eta)$, element of $\Bscr(\Hscr_{o,1};\Yscr)$}
\Symb{(\alpha-A_o|_{\cH_o})^{-1}}{$x\mapsto\left(\displaystyle{\mu\mapsto\frac{ x(\mu) - x(\alpha)}{\alpha - \mu}}\right)$, operator in $\Bscr(\Hscr_{o,-1};\Hscr_o)$}
\Symb{\dom{\SmallSysNode_o}}{$\left\{\sbm{x\\u}\in\sbm{\Hscr_o\\\Uscr}\bigmid\exists y\in\Yscr:~\mu\mapsto \mu x(\mu)+\varphi(\mu)u-y\in \cH_o \right\}$}
\Symb{\SmallSysNode_o}{$\bbm{x\\u}\mapsto\bbm{ \mu \mapsto \mu x(\mu)+\varphi(\mu)u-y\\y}$, where $x$ and $u$ determine $y$ via $y=\lim_{\re\eta\to\infty} \eta x(\eta)+\varphi(\eta)u$}

\kern 5mm\noindent
{\bf Formulas related to $\SmallSysNode_c$}.  
\Symb{\cH_c}{$\Hscr(K_c)$, the state space of the controllable model}
\Symb{(\alpha-A_c)^{-1}}{$x\mapsto \displaystyle{\frac{ x(\mu) - \wtsmash\varphi(\mu) \tau_{c,\alpha}\,x}{ \alpha + \mu}}$, element of $\Bscr(\Hscr_c;\Hscr_{c,1})$}
\Symb{A_c}{$x\mapsto \big(\mu\mapsto -\mu x(\mu) - \wtsmash \varphi(\mu) C_c x\big)$, $x\in\dom{A_c}$}
\Symb{C_c}{$\tau_{c,\alpha}(\alpha-A_c)$}
\Symb{\SmallSysNode_c}{$\bbm{x\\u}\mapsto\bbm{\mu\mapsto -\mu x(\mu)-\varphi(\overline\mu)^*\gamma_\lambda +\big(1-\varphi(\overline\mu)^*\varphi(\overline\lambda)\big)u\\\gamma_\lambda+\varphi(\overline\lambda)u}$, \\ where
$\gamma_\lambda=C_c\big(x-e_c(\lambda)^*u\big)$, for arbitrary $\lambda\in\cplus$}

\kern 3mm\noindent 
The following formulas are valid only under the assumption \eqref{eq:contrconscond}:
\Symb{\cH_{c,1}}{$\set{x\in\Hscr_c\mid\exists y\in\Yscr:~ \mu\mapsto \mu x(\mu)+\wtsmash\varphi(\mu)y\in\Hscr_c}$; this is the domain of $A_c$}
\Symb{\cH_{c,-1}}{$\left\{ x : {\mathbb C}^+ \to \cY \bigmid \exists y\in \cY :~ \mu \mapsto \displaystyle{\frac{x(\mu) +  \wtsmash \varphi(\mu) y}{ \beta + \mu}} \in \cH_c\right\}$, consists of equivalence classes modulo the subspace $\wtsmash \varphi(\cdot)\Yscr\subset\Hscr_{c,-1}$}
\Symb{\dom{\SmallSysNode_c}}{$\set{\sbm{x\\u}\in\sbm{\Hscr_c\\\Uscr} \bigmid\exists y\in\Yscr: \mu\mapsto -\mu x(\mu)-\wtsmash\varphi(\mu)y+u\in\Hscr_c}$}
\Symb{A_c|_{\cH_c}}{$[\mu\mapsto -\mu x(\mu)]$, lies in $\Bscr(\Hscr_c;\Hscr_{c,-1})$}
\Symb{B_c}{$u\mapsto[\mu\mapsto u]$, lies in $\Bscr(\Uscr,\Hscr_{c,-1})$}

Note that the reason for making the assumption \eqref{eq:contrconscond} is that it allows us to characterize the spaces $\Hscr_{c,\pm1}$ and $\dom{\SmallSysNode_c}$. The formulas for $A_c$ and $C_c$ form a circle definition. This can be avoided in case $\wtsmash\varphi(\cdot)\cap\Hscr_c=\zero$, since the $A_c$ be defined without using $C_c$; see Remark \ref{rem:removecircle}.

We next describe how to derive Theorems \ref{T:coenpres-real} and \ref{T:enpres-real} from \cite{ArKuStCanon}
by using the same method as was used in \cite{ArSt2Canon} to derive
Theorems 1.2 and 1.3, replacing the unit disk by the right half-plane. The multiplication operator $M_\varphi$ induced by a Schur function
$\varphi \in \cS({\mathbb C}^+;\cU, \cY)$ defines a contraction from
$H^2({\mathbb C}^+;\cU)$ into $H^2({\mathbb C}^+;\cU)$.  The graph of
this operator is a maximal nonnegative subspace $\widehat\Wfrak$ of
the Kre\u\i n space $H^2({\mathbb C}^+;\cW)$, where $\cW = \cU
\boxplus -\cY$ (i.e., the Kre\u\i n space $\cW$ is the orthogonal sum
of $\cU$ and the anti-space $-\cY$ of $\cY$).  This subspace is
invariant under multiplication by the function $\lambda \mapsto
e^{-\lambda}$.  The inverse Laplace transform maps $\widehat\Wfrak$
onto a maximal nonnegative right-shift invariant subspace $\Wfrak_+$
of the Kre\u\i n space $L^2({\mathbb R}^+;\cW)$, which using the
terminology of \cite{ArKuStCanon} is called a (time domain) passive
future behavior in $\cW$.  

In \cite{ArKuStCanon} three different
canonical state/signal realizations of $\Wfrak_+$ are constructed, one
which is controllable and energy preserving, another which is
observable and co-energy preserving, and a third which is simple and
conservative.  These three realizations are given in
the time domain setting, but they can be mapped into frequency domain
realizations by arguing as in \cite[Section 9]{ArSt2Canon}, with
the unit disk $\mathbb D$ replaced by the right half-plane ${\mathbb
C}^+$.  From these frequency domain realizations one can recover the
input/state/output realizations in Theorems \ref{T:coenpres-real} and \ref{T:enpres-real} (as well as an
additional simple conservative one) by using the fundamental
decomposition $\cW = \cU \boxplus -\cY$ of $\cW$ to get
input/state/output representations of scattering type of the canonical
state/signal representations in \cite{ArSt2Canon}, as was done in \cite[Section 10]{ArSt2Canon} in the discrete-time setting.

Finally, we mention that a planned project for the future is to develop a canonical-model of a conservative closely-connected (or simple) system node realization of a Schur-class function over ${\mathbb C}^+$.

 
\def\cprime{$'$} \def\cprime{$'$}

\end{document}